\documentclass{iig}

\usepackage[latin1]{inputenc}
\usepackage[all]{xy}
\usepackage{paralist, calc}
\usepackage{comment}

\usepackage{color} 

\newcommand\setn{\ensuremath{\{1,\ldots,n\}}}

\newcommand\R{\ensuremath{\mathbb{R}}}
\newcommand\N{\ensuremath{\mathbb{N}}}

\newcommand\F{\ensuremath{\mathbb{F}}}
\newcommand\Z{\ensuremath{\mathbb{Z}}}
\newcommand\C{\ensuremath{\mathbb{C}}}

\newcommand\Stab{\ensuremath{\mathrm{Stab}}}

\newcommand\cod{\ensuremath{\mathrm{Cod}(\theta)}}

\DeclareMathOperator{\Ad}{\ensuremath{\rm{Ad}}}
\DeclareMathOperator{\ad}{\ensuremath{\rm{ad}}}

\renewcommand{\labelenumi}{(\roman{enumi})}
\renewcommand{\labelenumii}{(\roman{enumii})}
\renewcommand{\labelenumiii}{\roman{enumiii})}
\renewcommand{\labelenumiv}{(\roman{enumiv})}

\renewcommand{\theenumi}{\labelenumi}
\renewcommand{\theenumii}{\labelenumii}
\renewcommand{\theenumiii}{\labelenumiii}
\renewcommand{\theenumiv}{\labelenumiv}

\newcommand{\refttb}[1]{(TTB{\ref{#1}})}
\newcommand{\reftw}[1]{(Tw{\ref{#1}})}
\newcommand{\refkmg}[1]{(KMG{\ref{#1}})}

\newcommand{\reftbn}[1]{(TBN{\ref{#1}})}
\newcommand{\refrgd}[1]{(RGD{\ref{#1}})}

\newcommand\goth{\ensuremath{\mathfrak{h}}}
\newcommand\gotg{\ensuremath{\mathfrak{g}}}

\newcommand\calD{\ensuremath{\mathcal{D}}}

\newcommand\calF{\ensuremath{\mathcal{F}}}
\newcommand\calG{\ensuremath{\mathcal{G}}}

\newcommand\calU{\ensuremath{\mathcal{U}}}

\newcommand\KP{\ensuremath{\tau_{{KP}}}}

\newcommand\filt{{\mrm{filt}}}

\DeclareMathOperator{\Hom}{\mathrm{Hom}}
\DeclareMathOperator{\Aut}{\mathrm{Aut}}
\DeclareMathOperator{\proj}{\mathrm{proj}}

\renewcommand\phi{\ensuremath{\varphi}}
\renewcommand\epsilon{\ensuremath{\varepsilon}}

\newcommand\mrm[1]{\mathrm{#1}}
\newcommand{\pro}{{\rm proj}}

\newcommand{\ra}{\to}

\newcommand{\Gall}{{\rm Gall}}

\newtheorem{thm}{Theorem}[section]
\newtheorem{cor}[thm]{Corollary}
\newtheorem{lem}[thm]{Lemma}
\newtheorem{prop}[thm]{Proposition}

\newtheorem{main}{Theorem}

\theoremstyle{remark}

\theoremstyle{definition}
\newtheorem{defin}[thm]{Definition}
\newtheorem{ex}[thm]{Example}

\newtheorem{rem}[thm]{Remark}

\numberwithin{equation}{section}

\renewcommand{\emph}{\textbf}
\renewcommand{\labelenumi}{{\rm(\roman{enumi})}}
\renewcommand{\labelenumii}{{\rm(\alph{enumii})}}
\numberwithin{equation}{section}

\begin{document}

\date{\today}
\title{{\bf On topological twin buildings and topological split Kac--Moody
groups}}
\author[Hartnick]{Tobias Hartnick}
\address{Technion \\
Department of Mathematics \\
Haifa \\
Israel}
\email{hartnick@tx.technion.ac.il}
\author[K\"ohl]{Ralf K\"ohl}
\address{Universit\"at Gie\ss en \\ Mathematisches Institut \\ Arndtstra\ss e 2
\\ 35392 Gie\ss en \\ Germany}
\email{ralf.koehl@math.uni-giessen.de}
\thanks{n\'e Gramlich}
\author[Mars]{Andreas Mars}
\address{TU Darmstadt \\
Fachbereich Mathematik \\
Schlo\ss gartenstra\ss e 7 \\
64289 Darmstadt \\
Germany}
\email{mars@mathematik.tu-darmstadt.de}

\maketitle

\begin{abstract}
\noindent
We prove that a two-spherical split Kac--Moody group over a local field
naturally provides a topological twin building in the sense of 
\cite{kramer2002}. This existence result and the local-to-global principle for twin building topologies combined with the theory of Moufang foundations as introduced and studied by M\"uhlherr, Ronan, and Tits allows one to immediately obtain a classification of two-spherical split Moufang topological twin buildings whose underlying Coxeter diagram contains no loop and no isolated vertices. 
%
%
%
%
%
%
%
\end{abstract}

\section{Introduction}
The objective of topological geometry is to study (incidence) geometries,
whose underlying sets are equipped with a topology with respect to which the
natural geometric operations are continuous. Among the most prominent 
examples are the {compact projective planes}, i.e., projective planes whose
point and line sets are compact Hausdorff spaces such that the maps that   
assign to two distinct points the unique line joining them and to two      
distinct lines the unique point incident to both are continuous maps. A    
detailled account on connected compact projective planes can be found in   
\cite{salzmann}. 

The gate property of Tits buildings 
(\cite[Section 4.9]{abramenko-brown}) and the resulting projection maps between
opposite panels 
generalize these geometric operations of joining points and intersecting 
lines in projective planes, thus leading to the concept of {compact 
generalized polygons} and, more generally, {topological (spherical) 
buildings}; see \cite{burns-spatzier}, \cite{kramer-polygons}. Various
subclasses of 
these topological geometries have been classified, see for example
\cite{burns-spatzier}, \cite{Grundhoefer/Knarr/Kramer:1995}, \cite{Grundhoefer/Knarr/Kramer:2000}, \cite{salzmann}.

An important extension of the class of spherical buildings is given by the class
of    
{twin buildings} \cite{tits-twin-buildings}. By definition, a twin building
consists of a pair of (possibly        
non-spherical) buildings together with a codistance function (or twinning),
which allows one to define the notion of opposition and   
(co-)projections from one half of the twin building to 
the other, extending the corresponding notions in the spherical case. While
spherical twin buildings are just another way of looking at
spherical buildings (\cite[Proposition 1]{tits-twin-buildings}), in the 
non-spherical case it is a very rigid property for a building to be part of
a twin building. The existence of co-projections in twin buildings opens the
door for the development of a
theory of topological twin buildings, by requiring co-projections to be
continuous. In order to be able to develop a rich theory, it seems necessary to
maintain some compactness assumption, for instance 
compactness of the panels. In the spherical case this assumption is equivalent
to compactness of the whole building, whereas in the non-spherical
case it is not.

An axiomatic definition of topological twin buildings along
the lines just described was first given by Kramer in \cite{kramer2002}, where he 
provides a geometric underpinning for the 
proof of Bott periodicity in \cite{Mitchell:1988}, which was based on a somewhat
ad
hoc notion of a topological $BN$ pair.
While the article \cite{kramer2002} describes in detail an explicit model of the
twin building
associated with a loop group, the theory is developed in an abstract way which
is independent of the existence of such a model and hence
does not rely on the underlying twin building being affine.

At the time of writing of
\cite{kramer2002} examples of non-discrete, non-affine, non-spherical
topological twin buildings were not well-understood---although in principle available through work by Kac and Peterson \cite{kac-peterson-topology} and by Tits \cite{tits-algebra}. During the last decade the theory
of Kac--Moody groups (R\'emy \cite{remy})
 and their group topologies (Gl\"ockner, Hartnick, K\"ohl \cite{final-group-topologies}) has been developed to a point where many such examples
can be described without problem. 

The purpose of the present article is twofold: on the one
hand, to show that the examples of twin buildings associated with two-spherical split
Kac--Moody group over local fields are indeed topological twin buildings in the
sense of Kramer; on the other hand, to revisit and extend the general theory in
light of these new examples up to the point necessary to obtain partial
classification results. An important feature of our presentation is that we
treat the connected and the totally-disconnected case as well as the
characteristic $0$ and the positive characteristic case largely simultaneously.

The idea to use group topologies on certain Kac--Moody groups to construct new
examples of topological
twin
buidlings was already suggested in the first author's master thesis
\cite{hartnick}, where also a group-theoretic criterion for twin building topologies
based on the theory
of RGD systems is stated (see \cite[Section~4.2]{hartnick}). However, carrying out this program---even in the case
of complex
Kac--Moody groups---became only possible after some rather technical insights
concerning direct limit topologies had been obtained in
\cite{final-group-topologies}. Using a minor variation of the results from \cite{final-group-topologies} it is straight-forward to carry out the program suggested in \cite{hartnick} for Kac--Moody groups over any local field of characteristic $0$, and to associate with every such group a topological twin building. Finally, using additional techniques from \cite{remy} one can extend the result to local fields of positive characteristics. One thereby arrives at the following result:

\medskip
\noindent {\bf Theorem \ref{HartnickConjecture}.}
{\em Let $G$ be a
two-spherical simply connected split
Kac--Moody group over a local field
and let $\KP$ be the Kac--Peterson topology on $G$. Then the associated twin
building endowed with the quotient topology is a strong topological twin building.

If the local field equals the field of real or of complex numbers, then $G$ is connected, otherwise totally disconnected.
%
}

\medskip
Theorem \ref{HartnickConjecture} provides a rich supply of topological twin
buildings; for definitions we refer to Section~\ref{section3.1} and to Definition~\ref{strongtop}. In the course of the proof of Theorem \ref{HartnickConjecture} we will explicitly construct and study the \emph{Kac--Peterson topology}; we refer to Definition~\ref{kpdef} and Remark~\ref{KacPetersonOriginal} for its definition, to \cite[Section~4G]{kac-peterson-regular-functions} for its original appearance and to \cite{kumar} (where it is called the \emph{analytic topology}) and to \cite{final-group-topologies} for subsequent discussions in the literature.

Following the latter, we will provide a universal characterization of the Kac--Peterson topology in the general case. We emphasize that in case of a non-spherical Kac--Moody group $G$ the Kac--Peterson topology is $k_\omega$, but not locally compact and not metrizable; cf.\ Proposition~\ref{universalKP} and Remark~\ref{notlocallycompact}. In case of a spherical Kac--Moody group, the Kac--Peterson topology coincides with the Lie group topology; cf.\ Corollary \ref{cor-panels-cpt} and Remark \ref{Rem:SphericalSubgroups}. Furthermore, the subspace topology induced on bounded subgroups turns these into algebraic Lie groups; cf.\ Corollary~\ref{bounded}.

The hypothesis in Theorem~\ref{HartnickConjecture} that the Kac--Moody group be two-spherical stems from Proposition~\ref{prop:multiplication-open}. That result has been announced in \cite[Section~4G]{kac-peterson-regular-functions} without the requirement that the Kac--Moody group be two-spherical. However, to the best of our knowledge there is no published proof of that statement available in the literature. Our own proof, which is based on a combinatorial argument by M\"uhlherr given in the appendix, unfortunately requires two-sphericity.

\medskip
It is natural to ask whether a classification of topological twin buildings is possible under
some natural conditions, mimicking known classification results in the spherical
case. We will provide such a classification result in the following restricted setting: Let us call a topological twin building \emph{$k$-split}, if all its
rank 
two residues are {\em compact Moufang polygons} and its rank one residues 
(considered as topological Moufang sets) are {\em projective lines over $k$}. We say that a $k$-split topological twin building is \emph{of tree type} if the underlying Coxeter diagram is a tree. In particular, such twin buildings are two-spherical and Moufang. 

Given a $k$-split topological twin building of tree type we can associate a Dynkin diagram in the usual way, using the fact that rank two residues are either generalized triangles or 
quadrangles or hexagons. Given a topological twin building $\Delta$ we denote by $\mathcal D(\Delta)$ the associated Dynkin diagram. With this notation understood our classification result can be formulated as follows:
\medskip

\noindent {\bf Theorem \ref{ClassificationIntro}.}
{\em Let $k$ be a local field. Then the map $[\Delta]\mapsto [\mathcal D(\Delta)]$ induces a bijection
between isomorphism classes of $k$-split topological twin buildings of tree type and
isomorphism classes of simply connected simple $\{3,4,6\}$-labelled graphs, where edges
labelled $4$ or $6$ are directed. }

\medskip
Note that the classification of abstract $k$-split twin buildings over fields is based on the theory of abstract foundations as developed by M\"uhlherr
\cite{muehlherr-locally-split}, M\"uhlherr--Ronan \cite{Muehlherr/Ronan:1995}, Ronan--Tits \cite{Ronan/Tits:1987}. The uniqueness part of our classification result is based on a topological formulation of this theory (cf.\ Theorem~\ref{ClassificationMain}), whereas the existence part is a direct consequence of Theorem \ref{HartnickConjecture}. 

Using the concept of normal coverings of Coxeter diagrams from \cite{muehlherr-locally-split} it is actually possible to extend Theorem \ref{ClassificationIntro} to $k$-split topological twin buildings of arbitrary type, but we refrain from doing so in the present article. We simply note that for arbitrary diagrams the Moufang foundation will not uniquely determine the twin building, so the precise classification statement becomes necessarily more complicated (cf.\ also the discussion on page~\pageref{discussion} before Theorem~\ref{ClassificationMain}). 

\medskip

This article is organized as follows: In Section \ref{SecTwin} we recall basic
definitions concerning twin buildings and RGD systems. We then establish a
couple of basic combinatorial properties of twin buildings. The main original result of that section (which stems from the first author's master thesis \cite{hartnick})
is 
Theorem~\ref{thm:projection-formula}, where we provide an explicit formula for
co-projections in the twin building associated with an RGD system in terms of
group data. In Section~\ref{section3} we introduce topological twin buildings and
develop their basic point-set topological properties. Our main contribution
in that section is the local-to-global result Theorem~\ref{LTG}, which was previously only
known in the spherical case. Our proof is based on the concept of Bott--Samelson desingularizations of Schubert varieties taken from \cite{kramer2002}. Section~\ref{SecClass} applies this local-to-global result to establish the uniqueness part of Theorem \ref{ClassificationIntro} (cf.\ Corollary \ref{CorTreeSuff}). We then digress in Section \ref{SecConnected} to re-visit \cite{Knarr:1990}, \cite{kramer-polygons}, \cite{kramer2002}, \cite[Theorem~3.3.10]{hartnick} and to discuss the topology of connected topological twin buildings. The main result here is Theorem~\ref{topsoltits}, which is a topological version of the Solomon--Tits theorem. The remainder of the article is then devoted to the explicit construction of topological twin buildings. In Section \ref{section:3.7} we develop the theory of topological groups with RGD system. In particular, we provide a list of conditions on the
topology of such a group, which guarantee that the associated twin building
becomes a topological twin building when equipped with the quotient topology
(see Theorem~\ref{thm:top-twin-building}).
In Section~\ref{section6} this result is applied to the case of split Kac--Moody
groups over a local field. We use the Kac--Peterson topology on these groups in order to establish Theorem~\ref{HartnickConjecture} and the existence part of Theorem~\ref{ClassificationIntro}. We close with some remarks concerning closure relations in Kac--Moody symmetric spaces.

\medskip
\noindent
{\bf Acknowledgements.} The authors thank Peter Abramenko, Helge Gl\"ockner,
Guntram Hainke,
Aloysius Helminck, Max Horn, Linus Kramer, Rupert McCallum, Bernhard M\"uhlherr, Andrei Rapinchuk, Nils Rosehr, Markus Stroppel and
Stefan Witzel for
several valuable comments, questions and discussions concerning the topic
of the present article. Moreover, they are indebted to Rupert McCallum and to an anonymous referee for a thorough proof-reading that resulted in many extremely helpful comments, suggestions and corrections. Furthermore, they are indebted to Bernhard M\"uhlherr for providing the combinatorial argument in the appendix.

The first author was partially supported by  ERC grant agreement n° 306706 {\em Ergodic Group Theory}. He also acknowledges the hospitality of 
Institut Mittag-Leffler 
during part of the revision of the article. The second and third authors gratefully acknowledge financial support by the Deutsche Forschungsgemeinschaft through funding of the research proposals GR 2077/5 and GR 2077/7.

\section{Twin buildings and RGD systems}\label{SecTwin}

\subsection{Twin buildings and their combinatorics} \label{2.1}

Buildings can be studied from the point of view of simplicial complexes (as done
in \cite{tits-finite-bn-pairs}) or, equivalently, from the point of view of
chamber systems (as introduced in \cite{Tits:1981}).  The book
\cite{abramenko-brown} is a comprehensive
introduction into the theory of buildings that explains both concepts in
detail, also including the theory of twin buildings.

In the present article we will study twin building topologies using the chamber
system approach to buildings. Throughout this article we reserve the letters $(W, S)$ to denote a Coxeter
system, which is always assumed to be of finite rank $|S|$.
We then denote by $\leq$ the associated Bruhat order and by $l = l_S$ the
associated length function on $W$. Given $J \subset S$ we denote by $W_J$ the subgroup of $W$ generated by $J$.
\begin{defin}
Let $(W, S)$ be a Coxeter system. A \index{building}\textbf{building} of type
$(W, S)$ is a pair $(\Delta, \delta)$ consisting of a set of chambers $\Delta$
together with a distance function $\delta: \Delta \times \Delta \to W$
satisfying the following axioms, where $x, y \in \Delta$ and $\delta(x,y) = w$:
\begin{enumerate}[(Bu1)]
 \item $w = 1$ if and only if $x = y$,
 \item if $z \in \Delta$ such that $\delta(y,z) = s \in S$, then $\delta(x,z)
\in \{ws, w\}$. If additionally $l(ws) > l(w)$, then $\delta(x,z) = ws$.
 \item If $s \in S$, there exists $z \in \Delta$ such that $\delta(y,z) = s$ and
$\delta(x,z) = ws$.
\end{enumerate}
\end{defin}
A building is called \textbf{spherical} if $W$ is finite. If $\Delta$ is
spherical, then $c, d \in \Delta$ are called \index{opposite}\textbf{opposite},
if $\delta(c,d) = w_0$, where $w_0$ denotes the longest element of $(W, S)$. 

For
every $c \in
\Delta$ and every subset $S'
\subseteq S$ we define the $S'$-{\bf residue}
$R_{S'}(c)$ to be
\[R_{S'}(c) := \{ d \in \Delta \mid \delta(c,d) \in W_{S'} = \langle s \mid s \in S'
\rangle \};\]
the collection of all $S'$-residues in $\Delta$ will be denoted  
${\rm Res}_{S'}(\Delta)$. 

\begin{rem}
Using the above definition, a building of rank one
is simply a set without further structure. In order to be able to develop a
meaningful theory for such buildings, one has to require additional
properties, such as the existence of a prescribed rank one group of
automorphisms; cf.\ \cite{Medts/Segev:2009}, \cite{Timmesfeld:2001}. In the
present article we will not deal with this situation and therefore only study
buildings whose Coxeter systems/Dynkin diagrams do not admit isolated points. 
\end{rem}

\begin{lem}[{\cite[Lemma~5.16 and
Corollary~5.30]{abramenko-brown}}]\label{residuesdisjoint}
Any residue of a building $\Delta$ is again a building. For any $S' \subseteq
S$ the elements of 
${\rm Res}_{S'}(\Delta)$ partition $\Delta$.  
\end{lem}

The building $\Delta$ will be called \emph{$k$-spherical} if all residues of rank $\leq k$ are spherical buildings. Every building is one-spherical, and a building of type $(W,S)$ is $|S|$-spherical if and only if it is spherical. For our purposes the class of two-spherical buildings will play a key role.

\medskip

The role residues of rank or co-rank one play is a particularly important one. Those of rank one, i.e.,
the elements of
${\rm Pan}_s(\Delta) := {\rm Res}_{\{s\}}(\Delta)$ are called
\textbf{$s$-panels}; as a convention, we write $P_s(c)$ instead of
$R_{\{s\}}(c)$. The residues of co-rank one, i.e., the elements of $\mathcal V_s
:= {\rm Res}_{S \setminus\{s\}}(\Delta)$ are called $s$-{\bf vertices}. There is
a canonical embedding \label{VertexEmbedding}
\begin{eqnarray*}
\iota : \Delta & \hookrightarrow & \prod_{s \in S} \mathcal V_s  \\ c & \mapsto
& \left(R_{S \backslash \{ s \}}(c)\right)_{s \in S},
\end{eqnarray*} 
which, from the simplicial complexes point of view on buildings, simply maps a
maximal simplex onto the tuple consisting of its vertices.

\medskip

A building is {\bf thin}, if each panel contains
exactly two elements, and {\bf thick}, if each panel contains at least three
elements. For a given chamber $c$ and a residue $R$ there exists a unique 
chamber $d \in R$ such that $$l(\delta(c,d)) = \min\{l(\delta(c,x)) \mid 
x \in R\},$$ see \cite[Proposition 5.34]{abramenko-brown}. This chamber $d$ is
called the \textbf{projection} of $c$ 
onto $R$ and is denoted by $\proj_R(c)$.
\begin{ex} \label{ex:coxeter-complex}
Let $(W, S)$ be a Coxeter system. Then $\Delta := W$ and
$\delta: \Delta \times \Delta \to W : (x,y) \mapsto x^{-1}y$ yields a 
(thin) building of type $(W, S)$, denoted by $\Delta(W,S)$. For any three
chambers $x, y, z \in \Delta$ one has
$\delta(x,z) = x^{-1}z = x^{-1}yy^{-1}z = \delta(x,y)\delta(y,z)$; see also
\cite[Lemma~5.55]{abramenko-brown}. Any thin building of type $(W, S)$ is 
isometric to $\Delta(W,S)$, cf.\ \cite[Exercise 4.12]{abramenko-brown}.
\end{ex}
Let $\Delta$ be a building of type $(W, S)$. A subset of $\Delta$ which is 
isometric to $\Delta(W,S)$ is 
called an \textbf{apartment} of $\Delta$.

\begin{defin}
A \index{twin buildings}\textbf{twin building} of type $(W, S)$ is a triple 
$((\Delta_+, \delta_+), (\Delta_-, \delta_-), \delta^*)$ consisting of two
buildings $(\Delta_+, \delta_+)$ and $(\Delta_-, \delta_-)$ of type $(W, S)$ 
and a \textbf{codistance} function $\delta^*: (\Delta_+ \times \Delta_-) 
\cup (\Delta_- \times \Delta_+) \to W$ subject to the following conditions, 
where $x \in \Delta_\pm$, $y \in 
\Delta_\mp$ and $\delta^*(x,y) = w$:
\begin{enumerate}[(Tw1)]
 \item \label{item:tw-1} $\delta^*(y,x) = w^{-1}$,
 \item \label{item:tw-2} if $z \in \Delta_\mp$ such that $\delta_\mp(y,z) = s
\in S$, and $l(ws) < l(w)$, then $\delta^*(x,z) = ws$, and
 \item \label{item:tw-3} if $s \in S$, then there exists $z \in \Delta_\mp$ such
that $\delta_\mp(y,z) = s$ and $\delta^*(x,z) = ws$. \qedhere
\end{enumerate}
A twin building is called \emph{spherical}, resp.\ \emph{$k$-spherical} if both of its halves have the corresponding property.
\end{defin}
Morphisms of twin buildings are defined as follows:
\begin{defin}
For $j \in \{ 1,2\}$ let $\Delta^{(j)} =  ((\Delta_+^{(j)}, \delta^{(j)}_+), 
(\Delta_-^{(j)},\delta^{(j)}_-), \delta^{*, (j)})$ be twin buildings
such that the type $(W^{(1)}, S^{(1)})$ is a sub-Coxeter system of the type
$(W^{(2)},S^{(2)})$. 
A \emph{morphism} $\phi: \Delta^{(1)} \to \Delta^{(2)}$
is an isometry
\[\phi: \Delta^{(1)}_+ \cup  \Delta^{(1)}_- \to  \Delta^{(2)}_+ \cup
\Delta^{(2)}_-,\]
i.e., a map that
preserves distances and codistances. 
\end{defin}
\begin{ex}\label{thintwinbuilding}
Let $(W,S)$ be any Coxeter system and let $\Delta_{\pm} := W$ and $\delta_{\pm}$
as in Example \ref{ex:coxeter-complex}. Moreover, define
$\delta^* : (\Delta_+ \times
\Delta_-) \cup (\Delta_- \times \Delta_+) \to W$ by $\delta^*(v,w) := v^{-1}w$.
Then
$((\Delta_+,\delta_+),(\Delta_-,\delta_-),\delta^*)$ is a thin twin building,
and any thin twin
building of type $(W,S)$ is isometric to this twin building, cf.\
\cite[Exercise~5.164]{abramenko-brown}. In this case the distance and the
codistance function are related by
the formula \[\delta^*(x,z) = x^{-1}z = x^{-1}yy^{-1}z = 
\delta^*(x,y)\delta_\mp(y,z), \quad {x \in \Delta_{\pm}, y,z \in
\Delta_{\mp}};\] see also
\cite[Lemma~5.173(4)]{abramenko-brown}. 
\end{ex}

A subset of a twin building isomorphic to a thin twin building is called a {\bf
twin apartment}. Twin apartments can be described efficiently via the notion of
opposition. Here two chambers $c \in \Delta_{\pm}$, $d \in \Delta_{\mp}$ are
called \textbf{opposite}, if $\delta^*(c,d) = 1$. Every pair of opposite
chambers is contained in a unique twin apartment by
\cite[Proposition~5.179(1)]{abramenko-brown}. Conversely, a pair of apartments
$(\Sigma_+, \Sigma_-)$ forms a twin apartment if and only if each chamber in
$\Sigma_{\pm}$ is opposite to 
exactly one chamber of $\Sigma_{\mp}$; cf.\
\cite[Proposition~5.173(5)]{abramenko-brown}. 

The notion of opposition can be extended
to residues by calling two residues \textbf{opposite}, if they have the 
same type and contain a pair of opposite chambers. Any pair of opposite 
residues is a twin building with respect to the restrictions of the 
distance and co-distance functions, cf.\
\cite[Exercise~5.166]{abramenko-brown}.

Given a spherical residue $R \subseteq \Delta_{\pm}$ and a chamber 
$c \in \Delta_{\mp}$, there exists a unique chamber $d \in R$ such 
that $\delta^*(c,d)$ is maximal in the set $\delta^*(c, R)$ with
respect to the Bruhat order, 
cf.~\cite[Lemma 5.149]{abramenko-brown}. This chamber 
$d$ is called the \textbf{co-projection} of $c$ onto $R$ and is
denoted by $\proj^*_R(c)$.

Given a chamber $c \in \Delta_{\pm}$ and an element $w \in W$ we denote by
\begin{align*} 
	E_w(c) & := \{ d \in \Delta_{\pm} \mid \delta_\epsilon(c, d) = w \in
W\},\\
	E^*_w(c) & := \{ d \in \Delta_{\mp}\mid \delta^*(c, d) = w \in W\}
\end{align*}	
the \textbf{Schubert cell}, respectively, \textbf{co-Schubert cell} of radius
$w$ and centre $c$. The sets $E_{\leq w}(c)$, $E_{< w}(c)$, $E_{\leq w}^*(c)$, $E_{<
w}^*(c)$ are defined accordingly. $E_{\leq w}(c)$ and $E_{\leq w}^*(c)$ are
called  \textbf{Schubert varieties}, respectively, \textbf{co-Schubert
varieties}.  Moreover, we define
\[\Delta_w := \{ (c,d) \in (\Delta_+ \times 
\Delta_{-}) \cup (\Delta_- \times \Delta_+) \mid \delta^*(c,d)=w \}.\]
The following combinatorial observations concerning twin buildings are quite
useful.

\begin{lem}\label{doubleprojection}
Let $\Delta$ be a twin building, let $s, t \in S$ be distinct, let $c_1, c_2 \in \Delta_+$
such that $\delta_+(c_1,c_2) = s$, and let $d \in \Delta_-$ such that $\{(c_1,
d), (c_2, d)\} \subset \Delta_1$. Then $\pro^*_{P_t(d)}(c_1) =
\pro^*_{P_t(d)}(c_2)$.
\end{lem}

\begin{proof}
Let $a_1 := \pro^*_{P_t(d)}(c_1)$. Then $\delta^*(c_1,a_1) = t$, whence
$\delta^*(c_2,a_1) \in \{ t, st \}$ by \cite[Lemma~5.139]{abramenko-brown}. On
the other hand, $\delta^*(c_2,d) = 1$ and, thus, $\delta^*(c_2,a_1) \in \{ 1, t
\}$ by \cite[Lemma~5.139]{abramenko-brown}. We conclude $\delta^*(c_2,a_1) = t$
and so $\pro^*_{P_t(d)}(c_1) = a_1 = \pro^*_{P_t(d)}(c_2)$.
\end{proof}

\begin{lem}[{\cite[Lemma~5.156]{abramenko-brown}}]\label{Tits1}
Let $\Delta$ be a thick twin building. Then for every pair $c_1, c_2 \in
\Delta_{\pm}$ there exists $d \in \Delta_{\mp}$ such that $\{(c_1, d), (c_2,
d)\} \subset \Delta_1$.
\end{lem}

 Let $c \in \Delta_{\pm}$ be a chamber and let $\Sigma$ be a twin apartment of
$\Delta$ containing $c$. Then the map $\rho = \rho_{c,\Sigma}: \Delta 
\to \Sigma$ which fixes $c$ pointwise and maps every twin apartment containing 
$c$ isometrically onto $\Sigma$ is called the \textbf{retraction} onto 
$\Sigma$ centred at $c$.

Since every two chambers are contained in a common twin apartment
(\cite[Proposition~5.179(3)]{abramenko-brown}), the retraction
$\rho$ preserves distances from $c$. Moreover, $\rho$ is distance-decreasing, 
i.e., $\delta(\rho(d), \rho(e)) \leq \delta(d,e)$ for any two chambers $d,
e \in\Delta$, where $\delta$ is to be interpreted as $\delta_+$, $\delta_-$ or
$\delta^*$, whichever one makes sense.

\begin{lem}[{\cite[Lemma~5.140(1)]{abramenko-brown}}]
\label{lemma:building-combinatorics}
Let $c \in \Delta_\pm$, $d, e \in \Delta_\mp$ be chambers, let 
$\delta^*(c,d)=w$, and let
$\delta_\mp(d,e) = v$.
Then $\delta^*(c,e) = wv'$, where $v'$ is a subexpression of $v$.
\end{lem}
\begin{proof}
Let $\rho = \rho_{c, \Sigma}$ be the retraction map onto some twin apartment 
$\Sigma$ containing $c$. Then 
$\delta^*(c,\rho(d)) = \delta^*(c,d) = w$ as $\rho$ preserves distances 
from $c$. Since $\rho$ is 
distance-decreasing, one has $\delta_\mp(\rho(d), \rho(e)) 
\leq v$. We conclude
\begin{eqnarray*}
	\delta^*(c,e) & = & \delta^*(c,\rho(e)) \\
& \stackrel{\ref{thintwinbuilding}}{=} & \delta^*(c,\rho(d))
\delta_\mp(\rho(d),\rho(e)) \in \{ wv' \mid v' \leq v\}.
\qedhere
\end{eqnarray*}
\end{proof}

\begin{lem} \label{lemma:nice-chamber-exists2}
Let $\Delta$ be a thick twin building, let $1 \neq w = s_1 \cdots s_k \in W$ be
reduced, and let $c_\pm \in \Delta_\pm$ be opposite chambers. 
Then there exists a chamber $d \in \Delta_-$ with 
$\delta^*(c_+, d) = 1$ and $\delta^*(E^*_w(c_-), d) = \{s_k\}$.
\end{lem}
\begin{proof}
By the definition of co-projections (see above or \cite[Lemma
5.149]{abramenko-brown}) there 
is a unique chamber $a \in P_{s_1}(c_-)$ such that $\delta^*(c_+,a) =
s_1$.
Since 
$\Delta$ is thick, there exists $a_1 \in P_{s_1}(c_-) 
\backslash \{c_-, a\}$. Then $\delta^*(c_+,a_1) =
1$ and, by axiom 
(Tw2), for all 
$x \in E^*_w(c_-)$ one has $\delta^*(a_1,x) = s_2 \cdots s_{k}$.
By induction we obtain a gallery $a_1$, \ldots, $a_{k-1}$ 
such that $\delta^*(c_+, a_i) = 1$ and such that for all $x \in E^*_w(c_-)$ 
one has $\delta^*(a_i,x) = s_{i+1} \cdots s_{k}$. Thus the chamber
$d := a_{k-1}$ has the desired properties.
\end{proof}

\subsection{Twin buildings from RGD systems}\label{RGD system}

A group $G$ \textbf{acts by isometries} on a twin building 
$\Delta = ((\Delta_+, \delta_+), (\Delta_-, \delta_-), \delta^*)$ if it 
acts on each half and preserves the distances and the 
codistance. A twin building is called \textbf{homogeneous} if it admits a 
group action by isometries which is transitive on each half. 

In this section we describe a class of homogeneous twin buildings using group
theory. 
For the necessary background information on reflection groups and their
associated root
systems we refer to \cite[Sections~1.5, 3.4]{abramenko-brown} or to
\cite{humphreys-coxeter}. For more details on RGD systems we strongly recommend
to consult \cite[Chapters 7, 8]{abramenko-brown} or \cite{caprace-lecture-notes}

\begin{defin}
Let $G$ be a group and let $\{U_\alpha\}_{\alpha \in \Phi}$ be a family of 
subgroups of $G$, indexed by some \index{root system}root system $\Phi$ of 
type $(W, S)$, let $\Phi^+$ be a subset of positive roots, and let $T$ be a
subgroup of $G$. The triple $(G, \{U_\alpha\}_{\alpha \in \Phi}, T)$ is called
an \index{RGD system}\textbf{RGD system} of type $(W, S)$ if it satisfies
the following assertions.
\begin{enumerate}[({RGD}1)]
\item [(RGD0)]For each root $\alpha \in \Phi$, one has $U_\alpha \neq \{1\}$.
\item For each \index{prenilpotent pair of roots}prenilpotent pair 
$\{\alpha, \beta\} \subseteq \Phi$ of distinct roots, one has 
$[U_\alpha, U_\beta] \subseteq \langle U_\gamma \mid \gamma \in ]\alpha,
\beta[ \rangle$. \\ (Cf.\ \cite[Sections~8.5.2, 8.5.3]{abramenko-brown} for a definition of a prenilpotent pair, the ``closed'' interval $[\alpha,\beta]$ and the ``open'' interval $]\alpha,\beta[$.)
\item \label{item:mu}
For each $s \in S$ there exists a function $\mu_s : U_{\alpha_s} \backslash
\{ 1 \} \to G$ such that for all $u \in U_{\alpha_s} \backslash \{1\}$ and
$\alpha \in \Phi$ one has $\mu_s(u) \in U_{-\alpha_s}uU_{-\alpha_s}$ and
$\mu_s(u)U_\alpha\mu_s(u)^{-1} = U_{s(\alpha)}$.
\item For each $s \in S$ one has $U_{-\alpha_s} \nsubseteq U_+ := 
\langle U_\alpha \mid \alpha \in \Phi^+ \rangle$.
\item $G = T.\langle U_\alpha \mid \alpha \in \Phi \rangle$.
\item The group $T$ normalises every $U_\alpha$.
\end{enumerate}
The tuple $(\{U_\alpha\}_{\alpha \in \Phi}, T)$ is called a {\bf root group
datum}, the $U_\alpha$ are called the {\bf root subgroups}, and the 
$G_\alpha := \langle U_{\pm\alpha} \rangle$ are called the {\bf rank one
subgroups}. 
\end{defin}
Occasionally, for pairwise distinct simple roots $\alpha_1$, ..., $\alpha_r$, we use the notation $G_{\alpha_1, \dots, \alpha_r}$ for the group generated by $G_{\alpha_1} \cup \dots \cup G_{\alpha_r}$. These groups are then referred to as {\bf fundamental rank $r$ subgroups}.

A root group datum $(\{U_\alpha\}_{\alpha \in \Phi}, T)$ is called 
\textbf{$\F$-locally split} if $T$ is abelian and if there is a field $\F$ such 
that $G_\alpha \cong \mrm{(P)SL}_2(\F)$ and $\{U_\alpha, U_{-\alpha}\}$ is 
isomorphic to the canonical root group datum of $\mrm{(P)SL}_2(\F)$. The 
RGD system is called \textbf{centred} if $G$ is generated by its root
subgroups, i.e., if $G = \langle U_\alpha \mid \alpha \in \Phi \rangle$. 

Root group data give rise to $BN$-pairs in the sense of the following
definition:
\begin{defin}
Let $G$ be a group and let $B, N$ be subgroups of $G$. The pair $(B, N)$ is 
called a {\textbf{$BN$-pair}} for $G$, if $G$ is generated by $B$ and $N$, the 
intersection $T := B \cap N$ is normal in $N$, and the quotient group 
$W := N/T$ admits a set of generating involutions $S$ such that
\begin{enumerate}[(BN1)]
 \item \label{item:bn-1}for all $w \in W$ and $s \in S$ one has 
$wBs \subseteq BwsB \cup BwB$, and
 \item \label{item:bn-2}$sBs \nsubseteq B$ for each $s \in S$.
\end{enumerate}
Two $BN$-pairs $(B_+, N)$ and $(B_-, N)$ of the same group $G$ satisfying 
$B_+ \cap N = B_- \cap N$ yield a {\bf twin $BN$-pair} $(B_+, B_-, N)$, if 
the following additional assertions hold:
\begin{enumerate}[(TBN1)]
 \item \label{item:tbn-1}for $\epsilon \in \{+, -\}$ and all $w \in W$, 
$s \in S$ such that $l(sw) < l(w)$, one has 
$B_\epsilon s B_\epsilon w B_{-\epsilon} = B_\epsilon sw B_{-\epsilon}$, and
 \item \label{item:tbn-2}for each $s \in S$ one has $B_+s \cap B_- = \emptyset$.
\end{enumerate}
\end{defin}
If $B, N$ is a $BN$-pair for $G$ and $S$ is as above then the quadruple $(G, B,
N, S)$ is called a
\index{Tits system}\textbf{Tits system} with {\bf Weyl group} $W$. The notion 
of a \textbf{twin Tits system} $(G, B_+, B_-, N, S)$ is defined 
accordingly. We remark that the pair $(W, S)$ is a Coxeter system; cf.\
\cite[Theorem~6.56(1)]{abramenko-brown}.

A group $G$  with a
$BN$-pair admits a \textbf{Bruhat decomposition} 
$G = \bigsqcup_{w \in W} BwB$, cf.~\cite[Theorems 6.17 and
6.56(1)]{abramenko-brown}, and 
a group $G$ with a twin $BN$-pair
admits a \textbf{Birkhoff decomposition} 
$G = \bigsqcup_{w \in W} B_\epsilon w B_{-\epsilon}$, cf.~\cite[Proposition
6.81]{abramenko-brown}.
The groups $B_+$, $B_-$ and their conjugates are called {\bf Borel
subgroups}. 

Important examples arise from root group data:
\begin{prop}[{\cite[Theorem 8.80]{abramenko-brown}}] \label{prop:rgd-twin-tits}
\index{twin $BN$-pair!from RGD systems}
Let $G$ be a group with a root group datum 
$(\{U_\alpha\}_{\alpha \in \Phi}, T)$ of type $(W,S)$ and for each $s \in S$ 
let 
$\mu_s: U_{\alpha_s} \backslash \{ 1 \} \to U_{-\alpha_s}U_{\alpha_s} 
U_{-\alpha_s}$ be the map provided by 
{\rm \refrgd{item:mu}}. Then the groups
\begin{eqnarray*}
	N   & := & T . \langle \mu_s(u) \mid u \in U_\alpha \backslash \{1\},
 s \in S \rangle, \\
	B_+ & := & T . U_+,\\
	B_- & := & T . U_-
\end{eqnarray*}
yield a twin $BN$-pair $(B_+, B_-, N)$ of the group $G$.
\end{prop}
\begin{defin}
Let $G$ be a group with a root group datum $(\{U_\alpha\}_{\alpha \in \Phi}, T)$ of type $(W,S)$, let $(B_+, B_-, N)$ be the associated twin $BN$-pair and let $\Pi \subset \Phi$ be a set of simple roots. We say that the root group datum has the \emph{finite prolongation property} if the following holds: For every finite sequence $(\alpha_1, \dots, \alpha_n) \in (\Pi \cup -\Pi)^n$ there exists a finite prolongation $(\alpha_1, \dots, \alpha_N) \in (\Pi \cup -\Pi)^N$ such that
\[B_+B_- \cap (U_{\alpha_1}\cdots U_{\alpha_n}) \subset (U_+ \cap (U_{\alpha_1}\cdots U_{\alpha_N}))\cdot(T\cap (U_{\alpha_1}\cdots U_{\alpha_N}))\cdot(U_- \cap (U_{\alpha_1}\cdots U_{\alpha_N})).\]  
\end{defin}

For a twin Tits system $(G, B_+, B_-, N, S)$ with Weyl group $W$ define 
$\Delta_{\pm} := G/B_{\pm}$. Given $gB_{\pm},hB_{\pm} \in \Delta_{\pm}$ using
the Bruhat decomposition let 
$$\delta_{\pm}(gB_{\pm}, hB_{\pm}) := w \in W \quad \mbox{if and only if}
\quad B_{\pm}g^{-1}hB_{\pm} = B_{\pm}wB_{\pm}.$$ Similarly using the
Birkhoff decomposition instead, given $gB_{\pm} \in \Delta_{\pm}$ and 
$hB_{\mp} \in \Delta_{\mp}$ let $$\delta^*(gB_{\pm}, hB_{\mp}) := w \in W
\quad \mbox{if and only if} \quad B_{\pm}g^{-1}hB_{\mp} = B_{\pm}wB_{\mp}.$$
Then $((\Delta_+, \delta_+), (\Delta_-, \delta_-),
\delta^*)$ is a twin building of type $(W,S)$, see 
\cite[Theorem 6.56 and Definition 6.82]{abramenko-brown}.
\begin{defin} The above twin building is denoted
\[\Delta(G, B_+, B_-, N, S) := ((\Delta_+, \delta_+), (\Delta_-, \delta_-),
\delta^*)\] and referred to as the twin building
\textbf{associated with} the twin Tits system $(G, B_+, B_-, N, S)$. 
If the twin 
Tits system arises from a RGD system $(\{U_\alpha\}_{\alpha \in \Phi}, T)$, 
then the associated twin building is also denoted by $\Delta(G,
\{U_\alpha\}_{\alpha \in \Phi}, T)$. 
\end{defin} 

The following observation is due to Bernhard M\"uhlherr. It relies on the property $(dco)$ for Moufang polygons, cf.\ \cite[Definition~5.1]{m}: For $d \in \mathbb{N}$, a Moufang polygons admits property $(dco)$ if, given an arbitrary chamber $c$, any pair of chambers opposite $c$ can be joined by a gallery of length at most $d$ consisting of chambers opposite $c$ only.    

\begin{prop}\label{lemma:finiteprolongation}
Every two-spherical root group datum for which there exists $d \in \mathbb{N}$ such that each residue of rank two of the associated twin building satisfies $(dco)$ has the finite prolongation property.

For infinite fields $\mathbb{F}$, this in particular applies to any $\mathbb{F}$-locally split two-spherical root group datum without generalized octagons as residues.  
\end{prop}
\begin{proof}
The first statement is \cite[Theorem~4.5]{m}; it suffices to define $\alpha_{n+1}, ..., \alpha_N$ in such a way that it contains any sequence of length $n(2(d+16)^n+1)$ of positive or negative simple roots, cf.\ \cite[Section~5]{m}.

The second statement follows from \cite[Lemma~5.2]{m} and the comment thereafter. 
\end{proof}

\subsection{A formula for co-projections onto panels}
Let $G$ be a group with root group datum $(\{U_\alpha\}_{\alpha \in \Phi}, T)$
and 
let \[\Delta := \Delta(G, \{U_\alpha\}_{\alpha \in \Phi}, T) = ((\Delta_+,
\delta_+), (\Delta_-, \delta_-), \delta^*)\] be the associated 
twin building. The goal of this section is to derive a formula for 
co-projections (cf.\ \cite[Lemma~5.149]{abramenko-brown}) onto panels of 
$\Delta$ in terms of the group structure of $G$. 

\begin{lem} \label{lemma:b-orbits}
Let $B_+$, $B_-$ be the opposite Borel subgroups provided by 
Proposition~\ref{prop:rgd-twin-tits}, let $c_-$ be the chamber of $\Delta_-$ fixed by $B_-$,
let $c \in \Delta_+$ be a chamber, and let $\delta^*(c_-,c) = w$. Then 
$B_-.c = B_- w B_+$.
In particular, $B_-.c$ is represented by a unique double coset of the 
Birkhoff decomposition and every such double coset corresponds to a $B_-$-orbit.
\end{lem}
\begin{proof}
The first statement is evident, as $\delta^*(B_-,gB_+) = \delta^*(c_-,c) = w$ if
and only if $B_-.c = B_-gB_+ = B_-wB_+$. By \cite[Lemma 6.70]{abramenko-brown}
the group $B_-$ acts transitively on the chambers at codistance $w$ from $c_-$,
which implies the second statement.
\end{proof}

\begin{lem} \label{lemma:combinatorics-of-borel-groups}
For each $w \in W$ one has 
$w^{-1}B_+wB_- \subseteq B_+B_-$.
\end{lem}

\begin{proof}
We proceed by an induction on $l(w)$. As the case $l(w) = 0$ is trivial, we
may assume $l(w) > 0$. Then there exist $s \in S$, $w' \in W$ such that $w =
sw'$ 
and $l(w) = l(w') + 1$.
For $x \in B_+wB_-$ we have 
$sx \in B_+sB_+wB_- \stackrel{\reftbn{item:tbn-1}}{=} B_+swB_- = B_+w'B_-$.
Therefore, by induction, $w^{-1}x = {w'}^{-1} s x \in {w'}^{-1}B_+w'B_-
\subseteq B_+B_-$.
\end{proof}

\begin{rem} \label{rem:rho-continuous}
The multiplication map 
$m : U_+ \times T \times U_- \to B_+B_- : (u_+,t,u_-) \mapsto u_+tu_-$ 
is bijective by
\cite[Section~8.8]{abramenko-brown}.
Therefore also  
\begin{eqnarray*}
\psi\colon B_- \hookrightarrow B_+B_- & \to & U_+ 
\backslash B_+B_- \\  b_- & \mapsto & U_+b_-
\end{eqnarray*}
 is a bijection and allows one to define a map
\begin{eqnarray*}
\pi\colon B_+B_-  & \to & B_-  \\
x & \mapsto & \psi^{-1}(U_+x).
\end{eqnarray*}
Finally, for $w \in W$, Lemma~\ref{lemma:combinatorics-of-borel-groups}
allows one to define
\begin{eqnarray*}
	\rho_w\colon B_+wB_- & \to & B_- \\
	x & \mapsto & \pi(w^{-1}x).
\end{eqnarray*}
\end{rem}

\begin{prop} \label{prop:prop-of-rho-w}
Let $x \in B_+wB_-$. Then $x \in B_+w\rho_w(x)$.
\end{prop}
\begin{proof}
By the Birkhoff decomposition of $G$ there exist $u_\epsilon \in U_\epsilon$, $w
\in W$, $t \in T$ such that $x = u_+ wtu_-$.
Since $w^{-1}u_+ w \in U_+U_-$, there exist $u^1_\epsilon \in U_\epsilon \cap 
w^{-1} U_+ w$ such that $w^{-1}u_+ w = u^1_+u^1_-$, whence $x = w w^{-1}u_+
wtu_- = wu^1_+u^1_-tu_-$.
Thus $w^{-1}x = u^1_+u^1_-tu_-$, so $\rho_w(x) = u^1_-tu_-$, and therefore $x =
w u^1_+\rho_w(x)$.
As $u^1_+ \in U_+ \cap  w^{-1} U_+ w$, there exists $u_2 \in U_+$ such that
$u^1_+ = w^{-1}u_2 w$.
We conclude $x = w u^1_+ \rho_w(x) = u_2 w \rho_w(x) \in B_+w\rho_w(x)$.
\end{proof}

We can now establish an explicit formula for co-projections onto panels:

\begin{thm}[{\cite[Theorem~4.3.5]{hartnick}}] \label{thm:projection-formula}
Let $(G, \{U_\alpha\}_{\alpha \in \Phi})$ be an RGD system, let $\Delta$ be the
associated twin building, let $(W,S)$ be the associated Weyl group, let $c_+ =
gB_+ \in \Delta_+$, let $c_- = hB_- \in \Delta_-$, let $\delta^*(c_+, c_-) = w
\in W$, and let $s \in S$ such that $l(ws) > l(w)$. Then \[\pro^*_{P_s(c_-)}(c_+)
= h \rho_w(g^{-1}h)^{-1} sB_-.\]
\end{thm}

\begin{proof}
One needs to prove that \[\delta^*(gB_+,h \rho_w(g^{-1}h)^{-1} sB_-) = ws\] and
\[\delta_-(hB_-, h \rho_w(g^{-1}h)^{-1} sB_-) = s.\]
Since $\delta^*(gB_+,hB_-) = w$, we have $g^{-1}h \in B_+wB_-$.
Proposition~\ref{prop:prop-of-rho-w} allows us to conclude that there exists
$b_+ \in B_+$ such that
$g^{-1}h = b_+ w \rho_w(g^{-1}h)$, whence $g^{-1} = b_+ w \rho_w(g^{-1}h)
h^{-1}$.
Therefore \[g^{-1}h \rho_w(g^{-1}h)^{-1} s = b_+ w \rho_w(g^{-1}h)h^{-1}h
\rho_w(g^{-1}h)^{-1} s = b_+ w s \in B_+wsB_-,\]
which shows that $\delta^*(gB_+,h \rho_w(g^{-1}h)^{-1} sB_-) = ws$.
Similarly, \[h^{-1}h \rho_w(g^{-1}h)^{-1} s = \rho_w(g^{-1}h)^{-1} s \in
B_-sB_-,\]
and so $\delta_-(hB_-, h \rho_w(g^{-1}h)^{-1} sB_-) = s$.
\end{proof}

\begin{cor}[{\cite[Corollary~4.3.6]{hartnick}}]
Let $(G, \{U_\alpha\}_{\alpha \in \Phi})$ be an RGD system, let $\Delta$ be 
the associated twin building, let $(W,S)$ be the associated Weyl group, 
 let $c_+ = gB_+ \in \Delta_+$, and let $c_- = hB_- \in \Delta_-$.  
If $\delta^*(c_+, c_-) = 1 \in W$, then for each $s \in S$ one has
\[\pro^*_{P_s(c_-)}(c_+) = h \pi(g^{-1}h)^{-1} sB_-.\]
\end{cor}

\begin{proof}
This follows from Theorem~\ref{thm:projection-formula}, as $\rho_1 = \pi$.
\end{proof}
In Sections \ref{section3} and \ref{section6} we will use these projection
formulae in 
order to derive the continuity of co-projections; in that context the 
following observation will become important:

\begin{lem}\label{ContinuityFoldingMain}
Let $\tau$ be a group topology on $G$ and equip $T$, $U_{\pm}$ and $B_+B_-$ with the 
subspace topologies. Assume that the 
continuous bijection $m : U_+ \times T \times U_- \to B_+B_- : (u_+,t,u_-) \mapsto u_+tu_-$ is 
open, i.e., a homeomorphism. Then the map $\rho_w$ introduced in Remark~\ref{rem:rho-continuous} is continuous for every $w \in W$.
\end{lem}
\begin{proof}
If $m$ is open, also $\psi$ is open and, therefore, $\pi$ and $\rho_w$ are
continuous; cf.\ Remark~\ref{rem:rho-continuous}.
\end{proof}

\section{Topological twin buildings} \label{section3}

\subsection{Axioms for topological twin buildings}\label{section3.1}
Throughout this section let $\Delta = ((\Delta_+,\delta_+), (\Delta_-,\delta_-),
\delta^*)$ be a thick
twin building of 
type $(W, S)$. In order to avoid pathologies we will always assume that the Coxeter diagram of $(W,S)$ has no isolated vertices. By a \emph{topology $\tau$ on $\Delta$} we will always mean a pair of topologies $\tau_{\pm}$ on $\Delta_{\pm}$. Given such topologies we equip the set $\Delta_+ \cup \Delta_-$ with the direct sum topology, i.e.\ $\Delta_+$ and $\Delta_-$ are clopen subsets of $\Delta_+ \cup \Delta_-$.

We recall
that a space $X$ is the {\bf direct limit} of subspaces $X_i$, denoted $X = \lim_\to
X_i$, if (a) $X = \bigcup_i X_i$ and (b) $U \subset X$ is open if and only if $U \cap X_i$ is open for all $i$.

\begin{defin}\label{def:top-twin}
Let $\Delta$ be a thick twin building (of type $(W,S)$ without isolated vertices in the Coxeter diagram) and $\tau$ a topology on $\Delta$. Then the pair $(\Delta, \tau)$ is called a 
\textbf{topological twin building} if it satisfies the following axioms:
\begin{enumerate}[(TTB1)]
\item \label{item:ttb-1}$\tau$ is a Hausdorff topology.
\item For each $s \in S$ and each
$c \in \Delta_\pm$ the map
\begin{eqnarray*}
E_1^*(c) & \to &
\Delta_+ \cup \Delta_- \\
d & \mapsto & \pro^*_{P_s(c)}(d)
\end{eqnarray*}
is continuous.
\item \label{item:ttb-3} There exist chambers $c_\pm \in \Delta_\pm$ such that
\[\Delta_{\pm} = \lim_\to E_{\leq w}(c_\pm).\]
\item \label{item:ttb-4} For each $s \in S$ there exists a compact panel
$P \in {\rm Pan}_s(\Delta_{\pm})$.
\end{enumerate}
A \emph{morphism} of topological twin buildings is a morphism of the
underlying
twin buildings that, additionally, is continuous with respect to the twin
building topologies.
\end{defin}
Our definition of a topological twin building is chosen in such a way that it assumes only a minimal set of axioms which we need to develop a non-trivial theory. Our definition is slightly different from Linus Kramer's original definition \cite[p.~169]{kramer2002}. Depending on the applications one has in mind one may want to add further axioms. We discuss various possible alternative axiomatizations in Section \ref{SecAxiomsDiscussion} below.

\subsection{Basic point-set topology of topological twin buildings}
In this section we investigate basic point-set topological properties of
topological twin buildings. Initially we will not make use of the compactness assumption (TTB4); the results in this section will hold
for every thick twin building $\Delta =
((\Delta_+,\delta_+),(\Delta_-,\delta_-), \delta^*)$, which is endowed with a
topology satisfying axioms (TTB1), (TTB2), and (TTB3).
\begin{lem} \label{lemma:nice-chamber-exists}
Let $c_\pm \in \Delta_\pm$ be opposite chambers and let 
$w \in W\setminus\{1\}$. 
Then there exists an open neighbourhood $U$ of $c_+$ in $\Delta_+$, which does
not
intersect $E_w^*(c_-)$.
\end{lem}
\begin{proof} Fix a reduced expression $w= s_1 \cdots s_k$ with $s_j \in S$.
By Lemma \ref{lemma:nice-chamber-exists2} there exists $d_- \in \Delta_-$ with
$\delta^*(c_+, d_-) = 1$ and $\delta^*(x, d_-) = s_k$ for all $x \in
E^*_w(c_-)$. 
For every $d' \in P_{s_k}(d_-) \setminus \{d_-, \pro^*_{P_{s_k}(d_-)}(c_+)\}$ we
have
\[\{c_+\} \cup E^*_w(c_-) \subset E_1^*(d').\]
By (TTB2) the restriction of $\pro^*_{P_{s_k}(d')}$
defines a 
continuous map $$f: E^*_1(d') \to P_{s_k}(d') = P_{s_k}(d_-).$$ Therefore the
lemma 
follows from the fact that $P_{s_k}(d')$ is Hausdorff, \refttb{item:ttb-1}, and
that $E^*_w(c_-) \subseteq f^{-1}(d_-)$.
\end{proof}

\begin{prop}\label{BigCoCellOpen}\label{prop:co-schubert-cell-is-open}
For every $c \in \Delta_{\pm}$ the co-Schubert cell $E_1^*(c)$ is open.
\end{prop}

\begin{proof} By symmetry we may assume $c \in \Delta_-$. By
\refttb{item:ttb-3} it suffices to show that for $c_0 \in \Delta_+$ the set
$E_1^*(c)$ is relatively open in $E_{\leq w}(c_0)$ for all $w \in W$, i.e., any
$c_+ \in E_{\leq w}(c_0) \cap E_1^*(c)$ is an interior point. By
Lemma~\ref{lemma:building-combinatorics}
the function $\delta^*(\cdot, c)$ takes only finitely many values on $E_{\leq
w}(c_0)$, and for each non-trivial value $w_j$ Lemma
\ref{lemma:nice-chamber-exists} produces an open neighbourhood $U_j$ of $c_+$ in
$\Delta_+$ with $U_j \cap E_{w_j}^*(c) = \emptyset$.  Then $E_{\leq w}(c_0) \cap
\bigcap U_j$ is an open subset of $E_{\leq w}(c_0)$ containing $c_+$ and
contained in $E_{1}^*(c)$, i.e., $c_+$ is an interior point of $E_{\leq w}(c_0)
\cap E_{1}^*(c)$.
\end{proof}
As a first application we deduce:
\begin{lem} \label{residueclosed}
Let $J \subset S$. Then every  $J$-residue in $\Delta_{\pm}$ is closed. 
\end{lem}
\begin{proof} We will prove that for a $J$-residue $R \subset \Delta_+$ the
set $\Delta_+ \backslash R$ is open by showing that an arbitrary 
$c \in \Delta_+ \backslash R$ is an interior point.
Let 
$d \in R$ so that $ \delta_+(d,c) \not \in \langle J \rangle$, let
$(\Sigma_+,\Sigma_-)$ be a twin apartment that contains both $c$ and $d$, and 
denote by $e$ the unique chamber in $\Sigma_- \cap E_1^*(c)$. For every $f\in R$
we have
\[\delta^*(f,e) \stackrel{\ref{lemma:building-combinatorics}}{\in} \langle J
\rangle \delta^*(d,e)
\stackrel{\ref{thintwinbuilding}}{=} \langle J \rangle \delta_+(d,c)
\delta^*(c,e) = \langle J \rangle \delta_+(d,c),\]
whence there exists $s \in S \backslash J$ with $\delta^*(f,e) \geq s$ for all 
$f \in R$. This shows $c \in E_1^*(e) \subset \Delta_+ \backslash R$, i.e.,
by Proposition \ref{prop:co-schubert-cell-is-open} $c$ is an interior point of
$\Delta_+ \backslash R$.
\end{proof}

Given a panel $P \subset \Delta_{\pm}$ and a chamber $c \in P$ we denote by
$P^\times = P^\times_c$ the \emph{pointed panel} $P \setminus \{c\}$. 
A pointed panel is open in its ambient panel by \refttb{item:ttb-1}.
 
\begin{prop}\label{PanelsHomeo1}
Let $P \subset \Delta_{\pm}$ and $Q \subset \Delta_{\mp}$ be opposite panels. 
Then the map $c \mapsto \pro^*_Q(c)$ restricts to a homeomorphism $p^*_{PQ}: P
\to Q$.
\end{prop}

\begin{proof} By \cite[Proposition~5.152]{abramenko-brown} the maps 
$p^*_{PQ}$ and $p^*_{QP}$ are mutually inverse bijections. Hence it remains 
only to establish their continuity. For this let $c \in P$ and 
$d := p^*_{PQ}(c) \in Q$ its projection. Then $P^\times_c \subset E_1^*(d)$, 
whence the restriction of $p^*_{PQ}$ to $P^\times_c$ is continuous by 
(TTB2). Since the open subsets $\{P^\times_c \,|\, c
\in P\}$ 
cover $P$, this implies continuity of $p^*_{PQ}$.
\end{proof}

Combining this with Lemma \ref{Tits1} we obtain:
\begin{cor}\label{PanelsHomeo}
Panels of the same type are pairwise homeomorphic.
\end{cor}
In particular, we deduce that --- in the presence of axioms (TTB1), (TTB2) and (TTB3) --- axiom (TTB4) is equivalent to the following, a priori stronger, axiom:
\begin{itemize}
\item[{\rm (TTB4$+$)}] For each $s \in S$ every panel
$P \in {\rm Pan}_s(\Delta_{\pm})$ is compact.
\end{itemize}
As another application of the corollary we now define a functor ${\rm type}$ from the category of pointed topological
twin buildings of a fixed type $(W,S)$ to the category of topological spaces as
follows: Given a topological twin building $\Delta$ we set
\[{\rm type}(\Delta, c) := \bigcup_{s \in S} P_s(c)\]
and refer to ${\rm type}(\Delta, c)$ as the \emph{topological type} of the twin 
building $\Delta$ at $c$. Every based morphism $\phi$ of (pointed) topological
twin buildings then induces a continuous map ${\rm type}(\phi)$ between the
corresponding topological types by restriction. Corollary~\ref{PanelsHomeo} ensures that, up to homeomorphism, ${\rm type}(\Delta, c)$ does not depend on
the choice of the basepoint $c$. The local-to-global principle established in Section~\ref{3.4} below shows that the
topology of $\Delta$ is uniquely determined by its topological type.

As another application of Proposition \ref{prop:co-schubert-cell-is-open} we
show:
\begin{prop}\label{ResiduesTTB} Let $\Delta$ be a twin building with a topology $\tau$ and $\Delta'$ a pair of opposite residues in $\Delta$.
\begin{itemize}
\item[(i)] $\Delta'$ is a twin building with respect to the restricted distance and codistance.
\item[(ii)] If $(\Delta, \tau)$ satisfies (TTB1), (TTB2), (TTB3), then so does $(\Delta', \tau|_{\Delta'})$.
\item[(iii)] If $(\Delta, \tau)$ is a topological twin building, then so is $(\Delta', \tau|_{\Delta'})$.
\end{itemize}
\end{prop}
\begin{proof} For (i) see e.g.\ \cite[Exercise~5.166]{abramenko-brown}. For (ii) and (iii) observe first that axioms \refttb{item:ttb-1} and (TTB2) descend from $\Delta$ to arbitrary subsets, while (TTB3) descends to arbitrary closed subsets. By Lemma \ref{residueclosed} it descends in particular to $\Delta'$. Finally, if $(\Delta, \tau)$ is a topological twin building, then it satisfies not only (TTB4) but also (TTB4+), and this property obviously descends to $(\Delta', \tau|_{\Delta'})$.
\end{proof}

\subsection{Gallery spaces and Bott--Samelson desingularizations}\label{BSdes}

In this section we provide tools that will allow us to study the global 
point-set topology of topological twin buildings. Throughout this section we fix a topological twin budiling $(\Delta, \tau)$.

For a reduced word $s_1\cdots s_k \in W$ define a \emph{gallery} of type $(s_1,
\dots, s_k)$  as a tuple $(c_0, \dots, c_k) \in (\Delta_{\pm})^{k+1}$ satisfying
$c_{i} \in {\rm Pan}_{s_{i}}(c_{i-1})$ for $i= 1,\dots, k$. The chamber $c_0 \in
\Delta_{\pm}$ is called the {\bf initial chamber} of the gallery.
The set of all galleries of type $(s_1, \dots, s_k)$ with initial chamber $c_0$
is denoted by $\Gall(s_1, \dots, s_k; c_0)$; it is endowed with the subspace
topology induced from $(\Delta_{\pm})^{k+1}$. 

The natural \emph{projection} and \emph{stammering maps} allow one to pass
between different gallery spaces:
\begin{eqnarray*}
\pi_{s_1, \dots, s_k; c_0}: \Gall(s_1, \dots, s_k; c_0) & \to & \Gall(s_1,
\dots, s_{k-1}; c_0) \\ (c_0, \dots, c_{k}) & \mapsto & (c_0, \dots, c_{k-1}),\\
s_{s_1, \dots, s_k; c_0}: \Gall(s_1, \dots, s_{k-1}; c_0) & \to & \Gall(s_1,
\dots, s_{k}; c_0) \\ (c_0, \dots, c_{k-1}) & \mapsto & (c_0, \dots, c_{k-1},
c_{k-1}).
\end{eqnarray*}

The following observation of Linus Kramer's provides a key insight into the topological
structure of topological twin buildings. Recall from Corollary~\ref{PanelsHomeo}
that panels of the same type are pairwise homeomorphic.

\begin{prop}[{\cite[p.~170,~171]{kramer2002}}]\label{KBS}
For every $c_0 \in \Delta_{\pm}$ the gallery space $\Gall(s_1, \dots, s_k; c_0)$
is a locally trivial fibre bundle over $ \Gall(s_1, \dots, s_{k-1}; c_0)$ with
fibre $P_{s_k}(c_0)$ via $\pi_{s_1, \dots, s_k; c_0}$. The stammering map
$s_{s_1, \dots, s_k; c_0}$ defines a global section of this bundle.
\end{prop}
\begin{proof} By symmetry we may assume $c_0 \in \Delta_+$. For each $e \in
\Delta_-$ define 
\[U_e := \{(c_0, \dots, c_{k-1}) \in  \Gall(s_1, \dots, s_{k-1}; c_0)\,|\,
c_{k-1} \in E^*_1(e)\}.\]
By Proposition \ref{BigCoCellOpen} the family $(U_e)_{e \in \Delta_-}$ provides
an open covering of $\Gall(s_1, \dots, s_{k-1}; c_0)$. By Proposition
\ref{PanelsHomeo1}, for each $e \in \Delta_-$ the map
\begin{eqnarray}
h_e: U_e \times P_{s_k}(e) & \to &  \pi_{s_1, \dots, s_k; c_0}^{-1}(U_e) \notag
\\
(c_0, \dots, c_{k-1}, d) & \mapsto & (c_0, \dots, c_{k-1},
\pro^*_{P_{s_k}(c_{k-1})}(d)) \label{HomeosGall}
\end{eqnarray}
is a homeomorphism, which in view of Corollary \ref{PanelsHomeo} provides the
desired local trivialization. The final claim is obvious.
\end{proof}

\begin{rem} \label{ttb4needed}
By \refttb{item:ttb-4} and Corollary \ref{PanelsHomeo} panels are compact, thus
so are the gallery spaces by Proposition~\ref{KBS}. Hence for each reduced word
$s_1 \cdots s_k$ the (surjective) endpoint map
\begin{eqnarray*}
p_{s_1, \dots, s_k; c_0}: \Gall(s_1, \dots, s_k; c_0) & \to & E_{\leq s_1\cdots
s_k}(c_0)\\ (c_0, \dots, c_k) & \mapsto & c_k
\end{eqnarray*}
is a quotient map by \refttb{item:ttb-1}. 
\end{rem}

This remark implies:

\begin{cor}\label{SchubertCompact}
Schubert varieties --- in particular, spherical residues --- are compact. 
\end{cor}

The corollary has important consequences for the point-set topology of
topological twin buildings. 

\begin{defin}
A Hausdorff topological space $X$ is called a 
\textbf{$k_\omega$ space} if there exists a 
countable ascending sequence $K_1 \subseteq K_2 \subseteq \cdots \subseteq X$ 
of compact sets such that $X$ is the direct limit of the $K_n$, i.e., $X =
\bigcup_{n \in \N} K_n$ and such that 
$U \subseteq X$ 
is open if and only if $U \cap K_n$ is open in $K_n$ for each $n$ with
respect to the subspace topology. 
\end{defin}

We refer the reader to \cite{k-omega} for an overview over the theory of
$k_\omega$ spaces; the benefits of the theory of $k_\omega$ spaces for studying
twin buildings and Kac--Moody groups are clearly visible in
\cite{final-group-topologies}.
Key properties of $k_\omega$ spaces for the present article are:

\begin{prop}[{\cite{k-omega}, \cite[Proposition~4.2]{final-group-topologies}}]
\label{prop:komega-properties} Each $\sigma$-compact locally compact Hausdorff
space is $k_\omega$. Moreover, the category of $k_\omega$ spaces is closed under
taking closed subspaces, finite products, Hausdorff quotients and countable
disjoint unions. Every $k_\omega$ space is 
paracompact, Lindel\"of and normal.
\end{prop}
\begin{cor}\label{TTBKOmega}
Each half of a topological twin building is a $k_\omega$ space, in particular paracompact, Lindel\"of
and normal. It is compact if and only if it is spherical. 
\end{cor}
\begin{proof} The first statement follows from
Proposition~\ref{prop:komega-properties}, Corollary~\ref{SchubertCompact} and
\refttb{item:ttb-3}, whereas the second statement then follows from the fact
that a closed subset of a direct limit of compact spaces is compact if and only
if it is already contained in one of the compact spaces.
\end{proof}
In fact, if $c_{\pm}$ are as in (TTB3) then the sequences $(E_{\leq w}(c_\pm))_{w \in W}$ are explicit $k_\omega$-sequences for $\Delta_{\pm}$. It then follows from the general theory of $k_\omega$ spaces that a sequence $(K_n) \subset \Delta_{\pm}$ is a $k_\omega$-sequence for $\Delta_{\pm}$ if and only if each $K_n$ is compact and for every $w \in W$ there exists $n \in \mathbb N$ such that $E_{\leq w}(c_\pm) \subset K_n$. In particular, we deduce that in the presence of axioms
\refttb{item:ttb-1}, (TTB2),
\refttb{item:ttb-4}, axiom (TTB3) is equivalent to the following strengthened version:
\begin{itemize}
\item[{\rm (TTB3$+$)}] For every chamber $c_\pm \in \Delta_\pm$,
\[\Delta_{\pm} = \lim_\to E_{\leq w}(c_\pm).\]
\end{itemize}
In other words, in the definition of a topological twin building the choices made in axioms (TTB3) and (TTB4) are inessential. Note that Corollary \ref{TTBKOmega}
 is based on the interplay of axioms \refttb{item:ttb-3}
and \refttb{item:ttb-4}. 

Another example for this interplay is provided by the
following
proposition:
\begin{prop}\label{Prop:discrete}
Let $(\Delta, \tau)$ be a topological twin building whose panels are discrete. Then $\Delta$ itself is
discrete.
\end{prop}
\begin{proof} By assumption all panels are discrete and compact, hence finite.
Consequently, Schubert varieties are finite and Hausdorff, hence discrete.
Therefore the proposition follows from \refttb{item:ttb-3}.
\end{proof}

This proposition allows one to locally characterize discrete topological twin buildings. Such local-to-global results are the topic of the next section.

\subsection{A local-to-global principle} \label{3.4}

Using the tools developed in the preceding section we derive a 
local-to-global principle for topological twin buildings. Recall that the ${\rm type}$
functor associates with every morphism of topological twin buildings a
continuous map between the topological types, a concept which is meaningful by Corollary~\ref{PanelsHomeo}. 

We intend to show that this
property characterizes morphisms of topological twin buildings among all twin
building morphisms. 

More precisely, let $\Delta^{(1)}, \Delta^{(2)}$ be
topological twin buildings and let $\phi: \Delta^{(1)} \to \Delta^{(2)}$ be a
morphism of the underlying abstract twin buildings. Choose $c \in
\Delta_\pm^{(1)}$; then the map
\[{\rm type}_{c}(\phi):  \bigcup_{s \in S} P_s(c) \to  
\bigcup_{s \in S} P_s(\phi(c))\]
between the topological types of the $\Delta^{(j)}$ can still be defined, but
need not be continuous. 

Now we have:

\begin{thm}[Local-to-global principle for twin building topologies]\label{LTG} 
Let $\Delta^{(1)}, \Delta^{(2)}$ be topological twin buildings, let $\phi: \Delta^{(1)} \to 
\Delta^{(2)}$ be a morphism of the underlying twin buildings, and let $c \in \Delta_\pm^{(1)}$. Then $\phi$ is
continuous if and only if ${\rm type}_c(\phi)$ is continuous.
\end{thm}
In the spherical case the above local-to-global result was first proved in
\cite[Proposition~3.5]{Boedi/Kramer:1995} using a coordinatization procedure
which, however, is not available in the general case. (Observe that the
coordinatization given in \cite{kramer2002} provides coordinates on Schubert
cells rather than co-Schubert cells.)
\begin{proof} Assume that ${\rm type}_c(\phi)$ is continuous. Since $\phi$ is
a morphism, for arbitrary
opposite panels $P$ and $Q$ of $\Delta^{(1)}$ one has
\begin{eqnarray*}
\phi \circ p^*_{PQ} = p^*_{\phi(P) \phi(Q)} \circ \phi,
\end{eqnarray*}
where $p^*_{PQ} : P \to Q$ and $p^*_{\phi(P)\phi(Q)} : \phi(P) \to \phi(Q)$
are the projection homeomorphisms from Proposition \ref{PanelsHomeo1}.
By Lemma \ref{Tits1} and the continuity of ${\rm type}_c(\phi)$ this implies
that for each panel $P$ of $\Delta^{(1)}$ the restriction $\phi_{|P}$ is
continuous. This shows that the statement that $\mathrm{type}_c(\phi)$ be continuous is in fact independent of $c \in \Delta_\pm$.

For an arbitrary reduced word 
$s_1\cdots s_k \in W$ the morphism
$\phi$ 
induces a map 
\begin{eqnarray*}\label{MapCtsToShowLTG}
\phi_{s_1, \dots, s_k; c}: \Gall(s_1, \dots, s_k; c) \to \Gall(s_1, \dots, s_k;
\phi(c)).
\end{eqnarray*}
We will prove by induction on $k$ that $\phi_{s_1, \dots, s_k; c}$ is 
continuous. For $k=0$ there is nothing to show, so we immediately turn to
the case $k > 0$.
Then by Proposition~\ref{BigCoCellOpen} the sets
\[U_e := \{(c,c_1 \dots, c_{k-1}) \in  \Gall(s_1, \dots, s_{k-1}; c)\,|\, 
c_{k-1} \in E^*_1(e)\}\]
provide an open covering of 
$\Gall(s_1, \dots, s_{k-1}; \phi(c))$ (cf.\ the proof of 
Proposition~\ref{KBS}). The homeomorphisms $h_e$ and $h_{\phi(e)}$ from
\eqref{HomeosGall} on page \pageref{HomeosGall} yield a commuting diagram
\[\begin{xy}\xymatrix{
U_e \times P_{s_k}(e) \ar[rr]^{h_e}
\ar[dd]_{\phi_{|U_e} \times \phi_{|P_{s_k}(e)}} && 
\pi^{-1}_{s_1,...,s_k;c}(U_e) 
\ar[dd]^{\phi_{|\pi^{-1}_{s_1,...,s_k;c}(U_e)}}\\ \\
U_{\phi(e)} \times P_{s_k}(\phi(e)) \ar[rr]^{h_{\phi(e)}} && 
\pi^{-1}_{s_1,...,s_k;\phi(c)}(U_{\phi(e)}).
}\end{xy}\]
The left vertical arrow is continuous, because $\phi_{|U_e}$ is continuous 
by the induction hypothesis and $\phi_{|P_{s_k}(e)}$ is 
continuous by continuity of $\mathrm{type}_e(\phi)$ as shown above. Therefore also
the right vertical arrow is continuous. As the sets
$\pi^{-1}_{s_1,...,s_k;c}(U_e)$
form an open covering of $ \Gall(s_1, \dots, s_k; c)$, this implies 
continuity of the map $\phi_{s_1, \dots, s_k; c}$. 

By Remark
\ref{ttb4needed} the endpoint map $p_{s_1, \dots, s_k; c}: \Gall(s_1,
\dots, s_k; c) \to E_{\leq s_1\cdots s_k}(c)$ is a quotient map, so that 
$\phi_{|E_{\leq s_1\cdots s_k}(c)}$ is continuous.
 
Therefore $\phi$ is continuous by \refttb{item:ttb-3}.
\end{proof}

\begin{cor}\label{AutoContinuity}
Let $\Delta$ be a topological twin building and let $\phi$ be an automorphism of the underlying twin 
building. If, for each type $s \in S$, there exists a panel $P_s$ of type
$s$ such that $\phi_{|P_s}$ is continuous, then $\phi$ is a homeomorphism.  
\end{cor}

\begin{proof}
Continuity of $\phi$ is immediate by (the proof of) Theorem \ref{LTG}. By axioms
\refttb{item:ttb-1} and \refttb{item:ttb-4} and Corollary \ref{PanelsHomeo},
each $\phi_{|P_s} : P_s \to
\phi(P_s)$ is a
bijective quotient map, i.e., open. Hence also continuity of $\phi^{-1}$ 
follows from Theorem \ref{LTG}, and $\phi$ is a homeomorphism. 
\end{proof}

\subsection{The axioms revisited}\label{SecAxiomsDiscussion}
We have seen that a reasonably deep theory of twin building topologies can be developed assuming only the axioms (TTB1), (TTB2), (TTB3), (TTB4) given above. On the other hand, these axioms seem to be of fundamental importance in order to be able to develop a meaningful theory. 

(TTB2) is the core axiom underlying --- in an explicit or implicit\footnote{As pointed out in \cite{kramer2002}, in the case of spherical topological building one can actually deduce these continuity properties by making sufficiently strong compactness assumptions and using an open mapping theorem. However, this approach is clearly limited to compact situations.} way --- all approaches to topological geometry: The geometric operations should be continuous on reasonable domains. We have chosen the weakest possible formulation which requires projections between opposite panels to be only separately continuous in each variable. In certain situations it is advantageous to assume joint continuity of such projections, or even joint continuity for projection between chambers of fixed, but possibly non-trivial, codistance. One would then replace (TTB2) by one of the following stronger axioms: 
\begin{itemize}
\item[(TTB2$+$)] For each $s \in S$ the 
map
\begin{align*}
 	p_s: \Delta_1 & \to \Delta_+ \cup \Delta_{-}\\
	(c,d) & \mapsto \pro^*_{P_s(c)}(d)
\end{align*}
is continuous. 
\item[(TTB2$++$)] For each $s \in S$ and $w \in W$ the 
map
\begin{align*}
 	p_s: \Delta_w & \to \Delta_+ \cup \Delta_{-}\\
	(c,d) & \mapsto \pro^*_{P_s(c)}(d)
\end{align*}
is continuous. 
\end{itemize}
We prefer to use the weakest possible axiom as part of our definition and to assume these additional properties only when needed. 

The Hausdorff axiom (TTB1) is standard in topological geometry.

The purpose of the compactness axiom (TTB4) is less obvious, and it seems tempting to try to develop at least the basic theory without it. However, all attempts to develop a sufficiently rich theory in order to allow meaningful applications have failed so far, even in the case of spherical buildings of rank two. We thus cannot avoid (TTB4) at this point.

Axiom (TTB3) has no counterpart in classical topological geometry and so requires some justification. Let us call a pair $(\Delta, \tau)$ a \emph{weak topological twin building} if it satisfies (TTB1), (TTB2) and (TTB4). 

As we have seen above in Section~\ref{BSdes}, after choosing base points $c_{\pm} \in \Delta_{\pm}$ the corresponding combinatorial Schubert cells still yield a canonical topological cell structure on $(\Delta, \tau)$. It is thus natural to refine the topology $\tau$ by passing to the weak topology
\[\tau' := \{U \subset \Delta_+ \cup \Delta_-\,|\, U \cap (E_{\leq w}(c_+) \cup
E_{\leq w}(c_-)) \in \tau|_{E_{\leq w}(c_+) \cup E_{\leq w}(c_-)}\}\]
with respect to this cell decomposition. 

We refer to $(\Delta, \tau')$ as the \emph{Schubert completion} of $(\Delta, \tau)$ and call two weak topological twin buildings \emph{Schubert equivalent} if their Schubert completions coincide. 

Schubert equivalent weak topological twin buildings have homeomorphic Schubert varieties, while their topology at infinity may be different; for the purposes of classification it seems unnatural to distinguish between them. 

Moreover we observe the following:
\begin{prop}\label{prop:TTB3}
If $(\Delta, \tau)$ is a weak topological twin building, then its Schubert completion $(\Delta, \tau')$ 
is a topological twin building. 
\end{prop}
\begin{proof} 
Firstly, since $\tau'$ refines $\tau$, the Schubert completion $(\Delta, \tau')$ inherits
(TTB1). Also, (TTB3) holds by definition. As for (TTB4), assume $P \in {\rm Pan}_s(\Delta_{\pm})$ is
compact with respect to $\tau$. Let $\{U_\alpha\} \subset \tau'|_{P}$ be a
covering of $P$ and choose $w$ sufficiently large so that $P \subset E_{\leq w}(c_{\pm})$.
Then $\{U_\alpha\} = \{U_\alpha \cap E_{\leq w}(c_{\pm})\} \subset \tau|_P$,
whence there exists a finite subcovering, showing that $(P, \tau'|_P)$ is
also compact. 

It remains to establish (TTB2).  For this, fix $s \in S$ and $c \in 
\Delta_+ \cup \Delta_-$. Then the map $\phi_{s,c}: E_1^*(c) \to \Delta_+ \cup \Delta_-$ given by $d \mapsto {\rm proj}^*_{P_s(c)}(d)$ takes values in $P_s(c)$. Given base chambers $c_{\pm} 
\in 
\Delta_{\pm}$ we can choose $w \in W$ so that $P_s(c)\subset E := E_{
\leq w}(c_+) \cup E_{\leq w}(c_-)$. We then obtain for all $A \subset  \Delta_+ \cup \Delta_-$,
\begin{equation}\label{SchubertCompletionTrick}
 \phi_{s,c}^{-1}(A)=\phi_{s,c}^{-1}(A \cap E).
\end{equation}
Now, if $U \in \tau'$, then by definition there exists $V \in \tau$ such that $U \cap E = V \cap E$. Since $\phi_{s,c}$ is $\tau$-continuous we then deduce from \eqref{SchubertCompletionTrick} that
\[
 \phi_{s,c}^{-1}(U)=\phi_{s,c}^{-1}(U \cap E)=\phi_{s,c}^{-1}(V \cap E)=\phi_{s,c}^{-1}(V) \in \tau \subset \tau',
\]
which shows $\tau'$-continuity of $\phi_{s,c}$ and thereby establishes (TTB2).
\end{proof}
\begin{cor} Every Schubert equivalence class of weak topological twin buildings has a unique representative which is a topological twin building.
\end{cor}
This corollary is the reason why we allow ourselves to include (TTB3) into our definition.
\begin{rem}
It is common practice in homotopy theory to replace a topology by its compactly
generated counterpart. Schubert completion provides a re-topologization
procedure which is similar in flavour. 

However, we should warn the reader that we do not know
whether this procedure preserves the homotopy type. For
example, our proof of the topological Solomon--Tits theorem in Section~\ref{topsoltits} relies crucially on
(TTB3), and it is not clear to us whether this assumption can be dropped.
\end{rem}
Having justified the necessity of our axioms, we should also mention some axioms which we do not assume, but other authors might find desirable. The following is a non-exhaustive list of such properties; throughout we equip the spaces ${\rm Res}_J^{\pm}$ --- and, in particular, the vertex sets $\mathcal V_s^{\pm}$ --- with the quotient topology with respect to the canonical map $\Delta_{\pm} \to {\rm Res}_J^{\pm}$. We then equip the product $ \prod \mathcal V_s^\pm$ with the product topology.
\begin{itemize}
\item[(TTB1$+$)] The vertex sets
$\mathcal V_s^{\pm}$, $s \in S$, are Hausdorff.
\item[(TTB5)] The set $\Delta_1 = \{ (c,d) \in 
(\Delta_+ \times \Delta_{-}) \cup (\Delta_- \times \Delta_+) \mid 
\delta^*(c,d) = 1 \}$ of opposite chambers is open.
\item[(TTB6)] For every $s \in S$ the canonical map 
$\Delta_{\pm} \to \mathcal V_s^{\pm}$  is open.
\item[(TTB6$+$)] For every $J \subset S$ the canonical map 
$\Delta_{\pm} \to {\rm Res}_J^{\pm}$  is open.
\item[(TTB7)] The diagonal embedding (see page \pageref{VertexEmbedding})
\begin{eqnarray*}
\iota : \Delta_\pm & \hookrightarrow & \prod_{s \in S} \mathcal V_s^\pm  \\ c &
\mapsto & \left(R_{S \backslash \{ s \}}(c)\right)_{s \in S}
\end{eqnarray*} 
is open and, hence, a homeomorphism onto its image.
\end{itemize}
For example, Linus Kramer's original definition of a topological twin building\footnote{We note Kramer
formulates these axioms in terms of the simplicial complex approach to
buildings; in view of the examples we have in mind, it appeared convenient to us
to reformulate the theory in the language of chamber systems. A detailled discussion of the two approaches---the simplicial complex one and the chamber system one---can be found in \cite{abramenko-brown}.} involves the axioms (TTB1$+$), (TTB2$+$),  (TTB4$+$) and (TTB7). The relevance of Axiom (TTB7) is that it allows one to pass freely between the simplicial complex picture and the chamber system picture of twin buildings in a topological context. From the simplicial point of view, it seems indeed more natural to start from a topology on the vertices and to induce a topology on chambers, rather than the other way round. On the other hand, for the theory presented here, (TTB7) is not strictly needed. 

The above axioms are actually not quite independent:
\begin{prop}\label{AxiomDependencies} Let $\tau$ be a topology on a twin building $\Delta$.
\begin{itemize}
\item[(i)] (TTB1$+$) implies (TTB1).
\item[(ii)] (TTB2$++$) implies (TTB2$+$) implies (TTB2).
\item[(iii)] (TTB6) and (TTB6$+$) are equivalent.
\item[(iv)]  If $(\Delta, \tau)$ is a topological twin building, then (TTB6) implies that the space ${\rm Res}_J^{\pm}$ (and in particular the vertex spaces $\mathcal V_s$) are T1.
\item[(v)] If $(\Delta, \tau)$ is a topological twin building, then (TTB1$+$) and (TTB6) imply (TTB7).
\end{itemize}
\end{prop}
\begin{proof} (i) If (TTB1$+$) holds and $\iota$ denotes the continuous injection from (TTB7), then $\iota(\Delta_\pm) \subset  \prod \mathcal V_s^\pm$ are Hausdorff. Since $\iota$ is continuous, it follows that  $\Delta_\pm$ are Hausdorff as well. 

(ii) and the implication (TTB6+)$\Longrightarrow$(TTB6) in (iii) are obvious. For the converse, let $J = \{s_1, \dots, s_n\}$ and $U \subset \Delta_{\pm}$ open. Denote by $\pi(U)$ the image of $U$ under the projection onto ${\rm Res}_J^{\pm}$ and by $\pi_j(U) \subset {\rm Res}_J^{\pm}$ the preimage of the image of $U$ in $\mathcal V_{s_j}$. Then (TTB6) implies that the $\pi_j(U)$ are open, whence $\pi(U) = \bigcap \pi_j(U)$ is open. 

In order to show (iv), we fix $J \subset S$ and let $R_1$, $R_2$ 
be distinct $J$-residues in $\Delta_{\pm}$. Since $\Delta_{\pm}$ is normal by Proposition~\ref{prop:komega-properties} and
$R_1$, $R_2$ are closed by Lemma~\ref{residueclosed}, there exist open
neighbourhoods $U_1$ and $U_2$ separating $R_1$ and $R_2$ in $\Delta_{\pm}$. By (iii) we know that (TTB6$+$) holds, hence the images $\pi(U_1)$ and $\pi(U_2)$ in  ${\rm Res}_J^{\pm}$ are open. Now $R_1 \in \pi(U_1)$ and $R_2 \not \in \pi(U_1)$ (and vice versa for $U_2$), hence (iv) follows. 

Concerning (v) we first observe that the map $\iota$ is continuous by definition of the quotient topology. By Corollary~\ref{SchubertCompact}, the Schubert varieties $E_{\leq w}(c^{\pm}) \subset \Delta_{\pm}$ are compact, whence mapped homeomorphically by $\iota$ by (TTB1$+$). Now, using (TTB1$+$) and \ref{prop:komega-properties}, each of the spaces $\mathcal V_s$ itself is $k_\omega$. Since in the category of $k_\omega$ spaces direct limits and finite products commute (cf.~\cite[Proposition
3.3]{Gloeckner:2003}, \cite[Proposition~4.7]{final-group-topologies}), the image
of 
$\iota$ is the direct limit of the images of the Schubert varieties, and so the
claim follows.
\end{proof}
There might be further relations between the axioms. For example, we do not know whether the argument in (iv) can be refined to show that (TTB6) implies (TTB1$+$). Also, we suspect that it might be possible to derive (TTB5) from the other axioms (cf. the argument in Proposition \ref{prop:ttb2} and Theorem \ref{thm:top-twin-building}). 
\begin{defin}\label{strongtop} A topological twin building is called a \emph{strong topological twin building} if it satisfies the additional axioms (TTB1$+$), (TTB2$+$),
 (TTB5), (TTB6) --- and hence also (TTB6$+$) and (TTB7).
\end{defin}
Note that in a strong topological twin buildings the residue spaces ${\rm Res}_J^{\pm}$ (and in particular the vertex spaces $\mathcal V_s$) are $k_\omega$, since they are Hausdorff quotients of $\Delta_{\pm}$. Moreover, by (TTB7) we can identify chambers of a strong topological twin building with the corresponding collections of vertices.

\subsection{Algebraic operations} \label{sec:algop}
By
\cite[Proposition~1.1]{Grundhoefer/Knarr/Kramer:1995} punctured panels in a
generalized 
polygon carry a multiplication operation, which can be defined in 
elementary geometric terms, i.e. by intersecting lines and connecting 
points. The following proposition provides an extension to the twin building
case:
\begin{prop}\label{algop}
Let $(\Delta, \tau)$ be a strong topological twin building, let $c_{\pm}$ be opposite 
chambers, let $r,s \in S$ with $3 \leq m_{rs} \leq \infty$, let $0_+ := 
\pro^*_{P_r(c_+)}(c_-)$ and $0_- := \pro^*_{P_s(c_-)}(c_+)$, and let 
$P_r(c_+)^\times := P_r(c_+) \backslash \{ c_+ \}$ and 
$P_s(c_-)^\times := P_s(c_-) \backslash \{ c_- \}$.  

For each choice $1_{-} \in P_s(c_-)^\times \backslash \{ 0_- \}$, there exists
a 
continuous map
\[\bullet: P_r(c_+)^\times \times P_s(c_-)^\times \to P_r(c_+)^\times\]
with the following properties:
\begin{itemize}
\item[(i)] For all $c \in \Delta^+$ we have $c \bullet 0_- = 0_+$ and $c
\bullet 1_- = c$.
\item[(ii)] For every $c' \in P_s(c_-)^\times \setminus\{0_-\}$ the map $c
\mapsto c 
\bullet c'$ is a homeomorphism of $P_r(c_+)^\times$.
\end{itemize}
\end{prop}
In case $(\Delta, \tau)$ is two-spherical, one reduces the proof of Proposition~\ref{algop} to the situation of generalized polygons as follows: Suppose $(\Delta, \tau)$ is a two-spherical topological twin building satisfying (TTB7). Then, by Corollary~\ref{SchubertCompact}, the residues of rank two are compact polygons in the sense of \cite{kramer-polygons}, and one can use the argument presented there. Thus, in this case, in fact axioms (TTB1--4) and (TTB7) suffice in order to deduce the conclusion of Proposition~\ref{algop}. 

Below we present a proof that does not assume two-sphericity, but uses axioms (TTB2$+$), (TTB5) and (TTB6$+$) in addition to axioms (TTB1--4). 

The following lemma extracts the place where these axioms enter the picture:

\begin{lem}\label{TTB2'} For every $s \in S$ the map 
\begin{eqnarray*}
\Delta_{\langle s \rangle} & \to & \Delta_+ \cup \Delta_- \\
\{ P,c \} & \mapsto & \pro^*_P(c)
\end{eqnarray*}
is continuous, where $\Delta_{\langle s \rangle} := \{(P, c) \in {\rm
Pan}_s(\Delta_+) \times
\Delta_{-} \mid \delta^*(P,c) \in \langle s \rangle\} \cup \{(c,P) \in 
\Delta_+ \times {\rm Pan}_s(\Delta_-) \mid \delta^*(c,P) \in 
\langle s \rangle\}$.
\end{lem}
\begin{proof}
The quotient map $(\Delta_{+} \times \Delta_-) \cup (\Delta_- \times
\Delta_+) \to ({\rm Pan}_s(\Delta_+) \times \Delta_{-}) \cup 
(\Delta_+ \times {\rm Pan}_s(\Delta_-))$ is open by 
(TTB6+), and so is its restriction 
$\Delta_1 \to \Delta_{\langle s \rangle}$ to the open subset $\Delta_1$ 
by (TTB5). By \reftw{item:tw-3} this restricted map is 
surjective, so the claim follows from (TTB2$+$).
\end{proof}
The remainder of the proof only uses (TTB2$+$) and the standard axioms:

\begin{proof}[Proof of Proposition~\ref{algop}]
Note that $P_r(1_-)$ is opposite $P_r(c_+)$. 
For
$d_+ := \pro^*_{P_s(c_+)}(c_-)$, the panel $P_r(d_+)$ is opposite $P_r(1_-)$,
because $1_- \neq c_- = \proj^*_{P_s(c_-)}(d_+)$, i.e., $1_-$ is opposite $d_+$.
Define
\begin{eqnarray*}
g : P_r(c_+)^\times \times P_s(c_-)^\times & \to & \Delta_- \\
(x,y) & \mapsto & \pro^*_{P_r(y)} \circ  \pro^*_{P_r(d_+)} \circ
\pro^*_{P_r(1_-)}(x), \\
\bullet: P_r(c_+)^\times \times P_s(c_-)^\times & \to & P_r(c_+)^\times \\
(x,y) & \mapsto & \pro^*_{P_r(c_+)} \circ  \pro^*_{P_r(y)} \circ  
\pro^*_{P_r(d_+)}
\circ
  \pro^*_{P_r(1_-)}(x) \\
& = & \pro^*_{P_r(c_+)}(g(x,y)).
\end{eqnarray*}
Note that the fact $x \bullet y \in P_r(c_+)^\times$ requires a proof that we
will provide below.

As $P_r(d_+)$ is indeed opposite to $P_r(y)$ for all $y \in
P_s(c_-)^\times = P_s(c_-) \backslash \{\proj^*_{P_s(c_-)}(d_+)\}$, by
Proposition~\ref{PanelsHomeo1} and Lemma~\ref{doubleprojection} the map
$g(\cdot, y): P_r(c_+)^\times \to P_r(y) \backslash \{\pro^*_{P_r(y)}(d_+) \}$
is a homeomorphism for every 
$y \in
P_s(c_-)^\times$.

The codistances between elements of $P_r(0_-)$ and of $P_r(c_+)$
lie in the set $\{ s, rs, sr, rsr \}$. Since $rsr \neq s$, the panels $P_r(0_-)$
and $P_r(c_+)$ are not parallel and, thus, the chamber $x \bullet 0_- \in
P_r(c_+)$ is independent of $x$. In fact, $x \bullet 0_-$ equals the unique
element in 
$P_r(c_+)$ for which there exists a chamber in $P_r(0_-)$ at codistance $rsr$,
i.e., $x \bullet 0_- = 0_+ \in P_r(c_+)^\times$ for all $x \in P_r(c_+)^\times$.

For $y \in P_s(c_-)^\times \backslash \{ 0_- \}$ the panels $P_r(y)$ and
$P_r(c_+)$ are parallel, and so Lemma~\ref{doubleprojection} implies
$\pro^*_{P_r(c_+)} \circ  \pro^*_{P_r(y)} \circ  
\pro^*_{P_r(d_+)}
\circ
  \pro^*_{P_r(1_-)}(c_+) = \pro^*_{P_r(c_+)} \circ  \pro^*_{P_r(y)} (d_+) =
c_+$.
We conclude that for $x \in P_r(c_+)^\times$ indeed $x \bullet y \in
P_r(c_+)^\times$ and that $P_r(c_+)^\times \to P_r(c_+)^\times : x \mapsto x
\bullet y$ is a homeomorphism.
 
We compute $g(x,1_-) =  \pro^*_{P_r(1_-)} \circ  \pro^*_{P_r(d_+)} \circ
\pro^*_{P_r(1_-)}(x) = \pro^*_{P_r(1_-)} (x)$, whence $x \bullet 1_- =
\pro^*_{P_r(c_+)} \circ \pro^*_{P_r(1_-)} (x) = x$.

It remains to prove continuity of $\bullet$. Continuity of $g$ 
follows
immediately from (TTB2$+$) and Lemma \ref{TTB2'}. It thus remains to show that
$\pro^*_{P_r(c_+)}$ is continuous on $G := g(P_r(c_+)^\times \times 
P_s(c_-)^\times)$. To this end
we 
claim that the following hold:
\begin{itemize}
\item[($\dagger$)] $\delta^*(c^+, G) = \{1, r, s\}$;
\item[($\dagger\dagger$)] $e \in G \cap E_1^*(d_+) \Rightarrow
\pro^*_{P_r(c_+)}(e) =
 \pro^*_{P_r(c_+)}(\pro^*_{P_{s}(e)}(d_+))$.
\end{itemize}
Let us first show that these claims imply the proposition: By ($\dagger$)
the sets $U_1, U_2$ given by
\[U_1 := G \cap (E_1^*(c_+) \cup E_1^*(0_+)), \quad U_2 := G \cap
E_1^*(d_+)\]
form an open covering of $G$. It is immediate from (TTB2+) that
$\pro^*_{P_r(c_+)}$ is
 continuous on $U_1$ and the map $e \mapsto
\pro^*_{P_r(c_+)}(\pro^*_{P_{s}(e)}(d_+))$ is continuous on $U_2$. Thus
continuity of 
$\pro^*_{P_r(c_+)}$ on all
of 
$G$ follows from ($\dagger\dagger$), and we are left with verifying our
claims. 
As far as $(\dagger)$ is concerned, let $(x,y) \in P_r(c_+)^\times \times
P_s(c_-)^\times$
 and assume $y \neq 0_-$. Since $\delta^*(c_+, y)  = 1$ and $\delta(y,
g(x,y)) \in \langle r \rangle$ it then follows that $\delta^*(c_+, g(x,y)) 
\in
\langle r 
\rangle$. On the other hand, if $y = 0_-$, then 
\[\pro^*_{P_r(c_+)}(g(x,y))= x \bullet y = 0_+,\]
whence $g(x,y) \neq \pro^*_{P_r(0_-)}(0_+)$ and $\delta^*(c_+,
g(x,y)) = s$. To prove ($\dagger \dagger$) we fix $e \in U_2$ and abbreviate 
$a
:= \pro^*_{P_{s}(e)}(d_+)$. 
Since $\delta^*(e, c_+) \in \langle s \rangle$, we see that $b :=
\pro^*_{P_r(c_+)}(e
)$ is the unique element in $P_r(c_+)$ satisfying $\delta^*(b, e) 
\in r\langle s
\rangle$. On the other hand we had assumed $\delta^*(d_+, e) = 1$, which 
implies
$\delta^*(d_+, a) = s$. This in turn implies $\delta^*(c_+, a) = 1$, hence
$\delta^*(\pro^*_{P_r(c_+)}(a), a) = r$ and finally 
$\delta^*(\pro^*_{P_r(c_+)}(a), e) \in
r \langle s \rangle$, showing $\pro^*_{P_r(c_+)}(a)= b$ and finishing the proof.
\end{proof}

\section{Classification methods}\label{SecClass}

\subsection{The Moufang condition}

In this section we approach the classification of topological twin buildings. While abstract rank one buildings without further structure are
not 
of any interest, various obstacles towards a complete classification become apparent in the context of rank two (twin) buildings of irreducible type. 

Firstly, there is an apparent 
difference between twin trees and spherical rank two (twin) buildings: While the 
latter admit a unique twinning by \cite[Proposition~1]{tits-twin-buildings}, the same pair of trees may instead admit many 
different twinnings by \cite[(1.1)]{Ronan/Tits:1999}.

 In order to avoid problems
arising from such
non-uniqueness of twinnings, we restrict our attention to 
buildings of two-spherical type $(W,S)$. This means that all rank two 
residues are assumed to be spherical or, equivalently, that there are no tree
residues. 

A second obstacle is that the automorphism group of a given compact projective plane (e.g., parametrized by a locally compact connected topological ternary field) may be 
small, so that a classification seems impossible. Thus, already in the 
classification of compact projective planes one has to assume some 
homogeneity condition; cf.\ \cite{salzmann}.

In order to obtain a class of topological twin buildings amenable 
to classification we therefore further restrict ourselves to those that admit spherical rank two 
residues which satisfy the \emph{Moufang condition}. We refer the reader 
to \cite{TitsWeiss} for both an introduction to and a complete classification of
such {\em Moufang polygons}. 
\begin{defin}
A topological twin building is called a \emph{Moufang topological twin building} if all its rank two residues are spherical and Moufang.
\end{defin}
The hypotheses that all residues of rank two be spherical and Moufang allow one to apply the local-to-global machinery by M\"uhlherr and Ronan \cite{Muehlherr/Ronan:1995}.
Note that the Moufang condition depends only on the underlying {abstract} twin building, hence guarantees the existence of many {\em abstract} automorphisms, but a priori not necessarily of any {\em continuous} automorphisms. 

It turns out, however, that the topological automorphism group of a Moufang topological twin building is automatically highly transitive; this follows from the following argument which was pointed out to us by
Linus Kramer.
\begin{prop}\label{toptrans}
Let $\Delta$ be a Moufang topological twin building. Then 
the topological automorphism group ${\rm Aut}(\Delta)$ admits a subgroup with an RGD system that acts strongly transitively on $\Delta$, i.e., transitively on the set of pairs of opposite chambers of $\Delta$.
\end{prop}
\begin{proof} By \cite[Proposition~8.19 and Theorem~8.81]{abramenko-brown} 
there exists a group with centred RGD system which acts strongly transitively on the abstract twin building underlying $\Delta$.
We claim that this action is actually by homeomorphisms. In order to prove this it suffices 
to show that each of the root groups acts by homeomorphism. However, for each root group there exists a panel of each type on which it acts trivially. This yields the desired conclusion by Corollary~\ref{AutoContinuity}.
\end{proof}

\subsection{Topological Moufang foundations}
Under various conditions, Moufang twin buildings can be classified by local 
data, so-called \emph{Moufang foundations}, see 
\cite{muehlherr-locally-split}, \cite{MvM}, \cite{Ronan/Tits:1987}. A topological analogue of 
foundations is provided by the following definition:
\begin{defin}\label{DefTMF}
Let $(W, S)$ be a two-spherical Coxeter system and denote by $E \subset
\binom{S}{2}$ the set of edges in the Coxeter graph of $(W,S)$. A
\emph{topological Mou\-fang foundation} of type $(W,S)$ is a triple \[\mathcal F
= (\{\Delta_J\,|\, J \in E\}, \{c_J\,|\, J \in E\}, \{\theta_{jik}\,|\, \{i,j\},
\{j, k\} \in E \})\]with the following properties:
\begin{itemize}
\item[(TMF1)] Each $\Delta_J$ is a topological Moufang polygon and $c_J \in
\Delta_J$.
\item[(TMF2)] $\theta_{ijk}: {\rm Pan}_{j}(c_{i, j}) \to {\rm Pan}_{j}(c_{j,k})$
is a base-point preserving homeomorphism and an isomorphism of Moufang sets.
\item[(TMF3)] The $\theta_{ijk}$ satisfy the cocycle condition $\theta_{kjl}
\circ \theta_{ijk} = \theta_{ijl}$. 
\end{itemize}
\end{defin}
Key examples of topological Moufang foundations arise from Moufang 
topological twin buildings: 
\begin{ex}\label{ExCollapse} Let $\Delta$ be a Moufang topological 
twin building and $c \in \Delta_+ \cup \Delta_-$ a base chamber. We then 
obtain a topological Moufang 
foundation by setting $\Delta_J :={\rm Res}_J(c)$, $c_J := c$ and 
$\theta_{jik}$ the appropriate (co-)restriction of the identity map of $\Delta$.  
\end{ex}
The topological Moufang foundation from Example \ref{ExCollapse} will be denoted $\mathcal F(\Delta, c)$ and  called the \emph{collapse} of $\Delta$ along $c$. By Proposition~\ref{toptrans}, up to isomorphism the foundation $\mathcal F(\Delta, c)$ associated with a
Moufang topological twin building only depends on $\Delta$.

We thus denote its
isomorphism class by $[\mathcal F(\Delta)]$.
\begin{defin} A topological twin building is said to \emph{globalize} a 
topological Moufang foundation $\mathcal F$ if there exists $c \in \Delta_+ 
\cup \Delta_-$ with $\mathcal F(\Delta, c) \cong \mathcal F$; the foundation is
then called \emph{integrable}.
\end{defin}

If one forgets about topologies in Definition \ref{DefTMF} then one obtains 
the definition of an (abstract) Moufang foundation. Under our standing 
two-sphericity assumption an integrable Moufang foundation determines the 
ambient twin building uniquely, provided the Coxeter diagram of the underlying 
type $(W,S)$ has no loops, see \cite[p.~394]{MvM} using \cite[Theorem~1.3]{Muehlherr/Ronan:1995} and \cite[Lemmas~5.1~and~5.2]{Ronan:2000}.\label{discussion} 

Theorem \ref{LTG} allows one to promote this statement to the 
following topological version.  

\begin{thm}\label{ClassificationMain}
Let $(W,S)$ be a (two-spherical) Coxeter system whose associated Coxeter diagram is a tree. Then
every Moufang topological twin building of type $(W,S)$ is uniquely determined by
its topological foundation.
\end{thm}

We stress that we do not claim to be contributing to the solution of the problem that certain abstract Moufang foundations be integrable. We simply state that an {\em abstractly integrable} Moufang foundation endowed with a topology as in Definition~\ref{DefTMF} uniquely determines the topology on its (uniquely determined) twin building.

\begin{proof} By the aforementioned local-to-global machinery any continuous isomorphism
of foundations extends to an isomorphism of abstract twin buildings. This
isomorphism is continuous, since its type is continuous (see Theorem \ref{LTG}).
Applying the same argument to its inverse we see that it is a homeomorphism.
\end{proof}
In the sequel we will say that a Moufang topological building is \emph{of tree
type} if it is of type $(W,S)$ and the Coxeter diagram of $(W,S)$ is a tree. We
use a similar terminology concerning foundations. By means of Theorem
\ref{ClassificationMain} the classification of Moufang topological twin
buildings of tree type is reduced to the following two problems:
\begin{itemize}
\item[(1)] Classify topological Moufang foundations of tree type.
\item[(2)] Decide which of these foundations are integrable.
\end{itemize}
For illustration we will carry out this programme for the class of split Moufang topological
twin buildings in the next section.

\subsection{Split foundations and Dynkin diagrams}\label{SecDynkin}
Throughout this section let $k$ be a local field, i.e., a non-discrete locally
compact
$\sigma$-compact Hausdorff topological field.  Recall that an archimedean local field is isomorphic to either $\mathbb{R}$ or
$\mathbb{C}$, whereas the non-archimedean local fields of characteristic $p$,
respectively $0$, are fields of the form
$\mathbb{F}_q((t))$ where $\F_q$ is a finite field, respectively finite extensions of $\mathbb{Q}_p$ for some prime $p$ (\cite[I, \S3]{weil:1995}).
\begin{defin}\label{DefSplitFoundation}
A topological Moufang foundation  $\mathcal F = (\{\Delta_J\,|\, J \in E\}, 
\{c_J\,|\, J \in E\}, \{\theta_{jik}\,|\, \{i,j\}, \{j, k\} \in E \})$ is 
called \emph{$k$-split} if each rank one residue of each $\Delta_J$ (equipped 
with the induced topology) is isomorphic as a topological Mou\-fang set to the
projective line over $k$ in its natural topology. A Moufang topological twin
building is called \emph{$k$-split} if some (hence any) of its foundations is
$k$-split.
\end{defin}
We remark that a foundation $\mathcal F$ as above is $k$-split if and only if 
all the $\Delta_J$ are isomorphic as compact polygons to either
\begin{itemize}
\item[(S-3)] the compact projective plane over $k$;
\item[(S-4)] the compact generalized quadrangle associated with ${\rm Sp}_4(k)$,
or its dual;
\item[(S-6)] the compact generalized hexagon associated with the split algebraic
group of exceptional type $G_2$ over $k$, or its dual. 
\end{itemize}
If one replaces in Definition \ref{DefSplitFoundation} the topological Mounfang foundation by an abstract one, then one obtains the notion of an abstract $k$-split foundation. Concerning such foundation we have the following result of
M\"uhlherr--Van Maldeghem:
\begin{lem}[{\cite[Proposition 2]{MvM}}]\label{LemmaMvM} Let \[\mathcal F =
(\{\Delta_J \mid 
J \in E\}, \{c_J \mid J \in E\}, \{\theta_{jik} \mid \{i,j\}, \{j, k\} \in E \})
\] be an abstract $k$-split Moufang foundation of tree type. Then 
$\mathcal F$ is uniquely determined by the list $\{\Delta_J \mid J \in E\}$.
\end{lem}
As an immediate consequence of Theorem \ref{ClassificationMain} the lemma carries over to the topological setting:
\begin{cor} A $k$-split Moufang topological twin building $\Delta$ is uniquely determined by the collection $\{\Delta_J \mid J \in E\}$.
\end{cor}

In view of the lemma, we can visualize $k$-split foundations of tree type by \emph{Dynkin trees}. Here, by a Dynkin tree we understand an edge-labelled tree, where edges are labelled $3$, $4$ or $6$ and edges of label $>3$ are given an orientation. Given a $k$-split Moufang foundation with associated list $\{\Delta_J \mid J \in E\}$ we define the associated Dynkin diagram as follows: The set of edges is the underlying index set $S$ of the foundation, and two vertices $i, j$ are joint by an edge if $J := (i,j) \in E$; the labeling of the edge $J$ is $3$, $4$ or $6$ according whether $\Delta_J$ is of type (S-3), (S-4) or (S-6) respectively; in the latter two cases the edge is oriented towards the short root. Conversely, every Dynkin tree determines a unique topological $k$-split Moufang foundation. We call the Dynkin tree \emph{topologically $k$-integrable} if this foundation is integrable. 

We can summarize our discussion as follows:
\begin{cor}\label{CorTreeSuff} Let $k$ be a local field. Then there is a bijection between $k$-split Moufang topological twin buildings and topologically $k$-integrable Dynkin trees.
\end{cor}
We will see in Corollary~\ref{DynkinTreeIntegrable} below that, in fact, every Dynkin tree is topologically $k$-integrable. For this we need to explicitly construct the corresponding Moufang topological twin building, which we will achieve in Section~\ref{toptwinbuilding} using topological split Kac--Moody groups.

\section{Connected topological twin buildings}\label{SecConnected}

\subsection{A dichotomy: Connected vs. totally-disconnected}

In this section we will study (strong) topological twin buildings $(\Delta, \tau)$ for
which the halves $(\Delta_{\pm}, \tau|_{\Delta_{\pm}})$ are connected. By slight
abuse of notation we call such a building $(\Delta, \tau)$ a \emph{connected (strong)
topological twin building}. In a similar way we also define the notion of a 
\emph{totally-disconnected topological twin building}. We recall our standing
assumption that there are no isolated vertices in the Coxeter diagram of 
the underlying type $(W,S)$. We say that a topological twin building is \emph{of irreducible type} if the underlying Coxeter diagram is connected. 

Our goal is to establish the following dichotomy:
\begin{prop}\label{ConnDichotomy}
 Every two-spherical topological twin building of irreducible type is either connected or totally disconnected.
\end{prop}
Note that the irreducibility assumption is necessary, since the product of a connected and a totally disconnected topological twin building is neither connected nor totally disconnected. 

Towards the proof of the proposition we first observe the following:
\begin{lem}\label{Conn-TotDc}
 Let $(\Delta, \tau)$ be a topological twin building.
\begin{enumerate}
 \item If all panels are connected, then $\Delta$ is connected.
 \item If all panels are totally disconnected, then $\Delta$ is
totally disconnected.
\end{enumerate}
\end{lem}
\begin{proof}
 If the panels are connected/totally disconnected, then the gallery spaces are
connected/totally disconnected by Proposition \ref{KBS}. In the former case it
follows immediately that the Schubert varieties are connected. In the latter
case it follows that the set of non-stammering galleries of a given type is
totally disconnected. Since this set is mapped homeomorphically onto the corresponding Schubert cell, it follows that each Schubert
cell is totally disconnected. 

For each $w \in W$ and $c \in
\Delta_{\pm}$ one has
\[E_{\leq w}(c_0) = \bigcup_{c \in E_{\leq w}(c_0)}(E_{w}(c) \cap E_{\leq
w}(c_0)).\]
The subsets $E_{w}(c) \cap E_{\leq w}(c_0)$ are open in $E_{\leq w}(c_0)$ and
totally disconnected,
hence the Schubert variety $E_{\leq w}(c_0)$ is totally disconnected itself.

The claim follows from the fact that direct limits of compact
connected/totally disconnected spaces are connected/totally disconnected.
\end{proof}
In view of Lemma \ref{Conn-TotDc}, the proof of Proposition \ref{ConnDichotomy}
follows from the corresponding statement concerning topological generalized polygons:
\begin{lem}
  Let $(\Delta, \tau)$ be an irreducible spherical topological twin building of rank two. Then either all panels
are connected or all panels are totally disconnected.
\end{lem}
\begin{proof} See \cite[2.2.3]{kramer-polygons}, \cite[Proposition~6.13]{GKMW}.
\end{proof}
We will provide many examples of connected, respectively totally-disconnected Moufang topological twin buildings using split Kac--Moody groups over archimedean, respectively non-archimedean local fields in Theorem~\ref{HartnickConjecture} on page \pageref{HartnickConjecture}. 

For the rest of this section we will focus on the connected case.

\subsection{Smooth topological twin buildings}
We will be interested in the following subclass of connected topological twin buildings:
\begin{defin} A connected topological twin building is called \emph{smooth} if its panels are finite-dimensional real manifolds.
\end{defin}
Note that, by definition, a topological twin
building is smooth if and only if its rank two residues are smooth. In many situations, smoothness is automatic:
\begin{prop}\label{SmoothAutomatic}
 Every connected Moufang topological twin building is smooth.
\end{prop}
\begin{proof}
By Lemma \ref{Conn-TotDc} a topological twin building $(\Delta, \tau)$ is connected if and only if its rank two residues are connected. A topological twin building $(\Delta, \tau)$ is smooth, if its rank two residues are. 
Since a Moufang topological twin building is two-spherical by definition, each of its rank two residues is compact by Corollary~\ref{SchubertCompact}. The claim therefore follows from the classification of flag-homogeneous connected compact polygons in
\cite{Grundhoefer/Knarr/Kramer:1995}, \cite{Grundhoefer/Knarr/Kramer:2000}.
\end{proof}
Although we will not take advantage of this fact, let us briefly mention that
the assumptions of the proposition can be substantially weakened. By
\cite[Theorem A]{Grundhoefer/Knarr/Kramer:1995} every flag-transitive compact
connected polygon is Moufang. Therefore by the extension theorem from \cite{Muehlherr/Ronan:1995} one has:
\begin{cor}
 Let $(\Delta, \tau)$ be a connected two-spherical topological twin building. If
the rank two residues have flag-transitive automorphism groups, then
$(\Delta, \tau)$ is smooth.
\end{cor}
We do not know any natural condition on twin trees which guarantee smoothness,
hence we cannot extend Proposition~\ref{SmoothAutomatic} beyond the
two-spherical case.
An important observation concerning smooth connected strong topological twin buildings is that
their panels are spheres:
\begin{prop}[{cf.~\cite[Theorem 1.6]{Grundhoefer/Knarr/Kramer:1995},
\cite[Lemma 2.1]{Knarr:1990},
\cite[Proposition~4.1.2]{kramer-polygons}}]
\label{panelsphere}
Let $(W,S)$ be a Coxeter system without isolated points and let
$\Delta$ be a smooth connected strong topological twin building of type $(W,S)$. Then
each panel of
$\Delta$ is a sphere.
\end{prop}
\begin{proof}
Since compact connected manifolds of positive dimension do not admit cutpoints,
each punctured panel is connected.
Let $P \subset \Delta_+$ be a panel of type $r$ and let $c_+ \in P$, so that
$P=P_r(c_+)$. By 
hypothesis there exists a type $s \in S$ such that $m_{rs} \geq 3$. For 
$c_- \in E^*_1(c_+)$ define $0_+ := 
\pro^*_{P_r(c_+)}(c_-)$ and $0_- := \pro^*_{P_s(c_-)}(c_+)$ and choose $1_-
\in P_s(c_-)^\times \cap E_1^*(c_+)$. Proposition \ref{algop} provides a
continuous
map
\[\bullet: P_r(c_+)^\times  \times P_s(c_-)^\times \to P_r(c_+)^\times.\]
 Since 
$P_s(c_-)^\times$ is a connected manifold, there exists a continuous path 
$\gamma(t)$ from $1_-$ to $0_-$. Then the map $H_t(x) := x 
\bullet \gamma(t)$ defines a pseudo-isotopic contraction of $P_r(c_+)^\times$ 
in the sense of \cite[p.~186]{Harrold}. Therefore by
\cite[Theorem]{Harrold} there exists $n \in \mathbb N$ such that 
$P_r(c_+)^\times \cong \R^n$ and, thus, $P = P_r(c_+) \cong \mathbb{S}^n$.
\end{proof}

Recall from Corollary \ref{PanelsHomeo} that panels of the same type are
pairwise homeomorphic, so that the following definition is meaningful.

\begin{defin}
Let $\Delta$ be a smooth connected strong topological twin building. For each $s \in S$
the number
$d(s) \in \mathbb{N}$ denotes the dimension of an $s$-panel. 
\end{defin}

Proposition~\ref{panelsphere} allows one to identify a CW structure on a smooth
real twin 
building and its geometric realization.

\begin{prop}[{\cite[Proposition~7.9]{kramer2002}, 
cf.~\cite[Theorem 4.1.3]{kramer-polygons}}]\label{SchCW} 
Let $(\Delta, \tau)$ be a smooth connected strong topological twin building.
Then the combinatorial Schubert decomposition of the halves of $\Delta_\pm$ with
respect to base chambers $c_{\pm}$ are CW decompositions. More precisely, there
exists a CW structure on
$\Delta_{\pm}$ with the following properties:
\begin{enumerate}
\item if $w \in W$ and $w = s_1 \cdots s_n$ is a reduced expression, then
there exists a cell of dimension $d(w) := d(s_1)+\dots +d(s_n)$ with attaching
map
\[\phi_w: (D^{d(w)}, S^{d(w)-1}) \to (E_{\leq w}(c_{\pm}), E_{<w}(c_{\pm}));\]
the corresponding open cell is the Schubert cell $E_w(c_{\pm})$;
\item every cell is of the form $\phi_w$ for some $w \in W$.
\end{enumerate}
\end{prop}

The proof of the proposition is based on the findings of Section~\ref{BSdes} and
the following observation.

\begin{lem}[{\cite[Lemma~6.2.12]{kramer-polygons}}]\label{cwbundle} Let $p: E
\to B$ be an 
$\mathbb{S}^d$-fibre bundle over a CW complex $B$ which admits a global 
section $s : B \to E$. Then there exists a unique CW structure on $E$ with 
the following properties:
\begin{enumerate}
\item $s(B)$ is a subcomplex of $E$ and $s: B \to s(B)$ is an 
isomorphism of CW complexes. 
\item Let $B^{k-1}$ be the $(k-1)$-skeleton of $B$ and 
$\mu: (D^k, S^{k-1}) \to (B, B^{k-1})$ be a $k$-cell. Then there exists a 
unique $(k+d)$-cell $\hat{\mu}: (D^{k+d}, \mathbb{S}^{k+d-1}) \to (E,
E^{k+d-1})$ with
\[\hat{\mu}(D^{k+d}) = p^{-1}(\mu (D^k)).\]
\item Every cell is either of type (i) or type (ii).
\end{enumerate}
\end{lem}

\begin{proof}[Proof of Proposition \ref{SchCW}]
Applying Lemma \ref{cwbundle} to the Bott--Samelson desingularization 
(Proposition~\ref{KBS}) yields a CW structure on each gallery space
$\Gall(s_1,...,s_k;c_{\pm})$ by induction on $k$, starting from 
the trivial CW structure of the point. Composing the attaching maps with the 
respective endpoint maps (Remark \ref{ttb4needed}) one obtains a
CW structure on Schubert varieties with centre $c_{\pm}$;
by 
(TTB3) these patch together to a CW structure on 
$\Delta_{\pm}$.
\end{proof}

\begin{rem}[{cf.~\cite[Theorem]{Knarr:1990}}]
Proposition \ref{SchCW} yields severe restrictions on the possible values of
the dimensions $d(s)$, because $d(w)$ has to be independent of the reduced 
expression for $w$. For instance, $d(\cdot)$ is constant on subsets of $S$
contained in a single $W$-conjugacy class, i.e., subsets of the connected 
components of the subgraph of the Dynkin diagram containing the simple edges
only, as the dihedral group $\langle s, t \mid s^2 = t^2 = (st)^3 = 1 \rangle$
admits the relation $sts =tst$.
\end{rem}

Another application of Proposition \ref{SchCW} concerns the geometric 
realizations $|\Delta_{\pm}|$ of the two halves of the twin building defined as
follows: 
Let $\Delta = ((\Delta_+,\delta_+), (\Delta_-,\delta_-), \delta^*)$ be a 
topological twin building of type $(W,S)$ and $n:= |S|-1$. Define the 
standard 
simplex and its faces as 
\[\blacktriangle^n := \{v \in \R^{n+1} \mid \sum_{i=1}^{n+1} v_i = 1\},\]
respectively, 
\[\blacktriangle^n[j] := \left\{v \in \R^{n+1} \mid \sum_{i=1}^{n+1} v_i = 1, \quad
v_j = 0\right\}, \quad (j= 1, \dots, n+1).\]
Then the geometric realizations $|\Delta_{\pm}|$ of $\Delta_{\pm}$ are given by
the following construction: Enumerate $S = \{s_1, \dots, s_{n+1}\}$, equip
$\Delta_{\pm} \times
\blacktriangle^n$ with the product 
topology and identify the $j$th faces of chambers which are contained in the 
same $s_j$-panel. In the sequel we will denote by $q: \Delta_{\pm} \times
\blacktriangle^n \to
|\Delta_{\pm}|$ the canonical quotient maps.

\medskip

If we  equip $\blacktriangle^n$ 
with the CW structure given by its faces, then the product CW structure on 
$\Delta_{\pm} \times \blacktriangle$ descend to the geometric realizations 
$|\Delta_{\pm}|$. Observe that the subsets
\[|E|_{\leq w}(c) := q(E_{\leq w}(c) \times \blacktriangle^n), 
|E|_{< w}(c) := q(E_{< w}(c) \times \blacktriangle^n)\]
are subcomplexes of $|\Delta_{\pm}|$ for every $w \in W$ and $c \in
\Delta_{\pm}$.
 Since the inclusion of a CW subcomplex is a cofibration (see e.g. \cite[Theorem
7.12]{spanier}), we obtain:
\begin{cor}[{cf.~\cite[Theorem 2.22(c)]{Mitchell:1988}}]\label{CofibrationLemma}
For every $w \in W$ and $c \in \Delta_{\pm}$ the inclusion $|E|_{< w}(c) 
\hookrightarrow |E|_{\leq w}(c)$ is a cofibration.
\end{cor}

\subsection{A topological Solomon--Tits theorem} \label{topsoltits} 
The goal of this subsection is to establish the following topological variant of
the Solomon--Tits theorem:
\begin{thm}[{\cite[Corollary 7.11]{kramer2002}}]\label{SalTits} Let $\Delta$ 
be a smooth connected strong topological twin building of type $(W,S)$, whose Coxeter
graph contains 
no isolated points. Then $|\Delta_{\pm}|$ is a homotopy sphere of dimension 
$d(w_0) + |S|-1$ if $W$ is finite and $w_0 \in W$ is the longest word, and 
contractible if $W$ is infinite.
\end{thm}

We start with some preliminary observations and reductions. By  \cite{k-omega}, \cite{final-group-topologies} direct 
limits in the category of $k_\omega$ spaces commute with finite products and 
with quotient maps. We deduce 
that for any fixed chamber $c \in \Delta_{\pm}$,
\[|\Delta_{\pm}| = \lim_{\to} |E|_{\leq w}(c).\]
If $W$ is infinite, we even obtain
\begin{eqnarray}\label{InfiniteContractionLimit}
|\Delta_{\pm}| &=& \lim_{\to} |E|_{< w}(c).\end{eqnarray}
The key step in the proof of Theorem \ref{SalTits} is to show that for every 
non-maximal $w \in W$ the quotient 
\[B_{w}(c) :=  |E|_{\leq w}(c)/ |E|_{< w}(c)\]
is contractible. 

Let us assume this for the moment and explain how to deduce 
Theorem~\ref{SalTits}: We claim that our assumption implies that 
$|E|_{< w}(c)$ itself is always contractible, even if $w$ is maximal. Indeed, 
for $l(w) \leq 1$ this is clear. Now let $l := l(w) > 1$ and let 
$w_1, \dots, w_N$ be the maximal elements with respect to the Bruhat order 
subject to the condition $w_j < w$. By induction hypothesis, 
$|E|_{< w_j}(c)$ is contractible for every $j=1, \dots, N$. Since $w_j < w$, 
the $w_j$ are non-maximal, whence by our assumption also 
$|E|_{\leq w_j}(c)/ |E|_{< w_j}(c)$ is contractible for every $j$. In view 
of Corollary \ref{CofibrationLemma} this implies that each of the sets 
$|E|_{\leq w_j}(c)$ is contractible. Since
\[|E|_{\leq w_j}(c) \cap |E|_{\leq w_l}(c) = |E|_{\leq \inf\{w_j, w_l\}}(c),\]
the same argument shows that finite intersections of the $|E|_{\leq w_j}(c)$ 
are contractible. Using Corollary \ref{CofibrationLemma} this implies
\[|E|_{< w}(c) = \bigcup_{j=1}^n |E|_{\leq w_j}(c) \simeq \bigvee_{j=1}^n 
|E|_{\leq w_j}(c) \simeq \{*\}\]
and, thus, establishes our claim. 

In the infinite case we can combine our claim and 
\eqref{InfiniteContractionLimit} to deduce
\[|\Delta_{\pm}| \simeq \{*\};\]
if $W$ is finite with longest word $w_0$ then another application of Corollary
\ref{CofibrationLemma} yields
\[|\Delta_{\pm}| \cong |E|_{\leq w_0}(c)/|E|_{< w_0}(c) = B_{w_0}(c).\]
We have thus reduced the proof of Theorem \ref{SalTits} to the following 
lemma:
\begin{lem}[{\cite[Proposition 7.10]{kramer2002}, 
cf.~\cite[2.10--2.15]{Knarr:1990}, \cite[Theorem 2.16]{Mitchell:1988}}]
Let $w \in W$. Then $B_{w}(c) \simeq  S^{d(w)+|S|-1}$, if $w$ is maximal; otherwise, $B_w(c)$ is contractible.
\end{lem}
\begin{proof} We fix an enumeration $S = \{s_1, \dots, s_N\}$ and for 
$1 \leq i \leq N$ denote by 
\[\blacktriangle^{N-1}[i] := \{(t_1, \dots, t_N) \in \R^n\,|\, \sum_{j=0}^N t_j
= 1, t_i = 0\}\]
the $i$th face of the standard $(N-1)$-simplex. We then denote by 
$q: \Delta_{\pm} \times \blacktriangle^{N-1} \to |\Delta_{\pm}|$ the quotient 
map given by identifying the $i$-th faces of $s_i$-adjacent chambers.
Furthermore we denote by $\pi: |E|_{\leq w}(c) \to B_{w}(c)$ the canonical
projection and set
\[p := \pi \circ q|_{E_{\leq w}(c) \times \blacktriangle^{N-1}}: E_{\leq
w}(c)\times \blacktriangle^{N-1} \to B_w(c).\]
We observe that $p$ maps all points in $E_{< w}(c) \times \blacktriangle^{N-1}$ to the
basepoint $* := p(c)$ of $B_w(c)$. In particular $p$ factors through a map
\[p_0: E_{\leq w}(c)/E_{< w}(c) \times \blacktriangle^{N-1} \to B_w(c).\]
Since $E_{\leq w}(c)\times \blacktriangle^{N-1}$ is compact, the maps $p$ and,
consequently, $p_0$ are quotient maps.

\medskip
\noindent \textsc{Claim 1}: If $l(ws_i) < l(w)$ then $p(d, (t_j)) = *$ for all $d \in
E_w(c)$ and $(t_j) = (t_1, \dots, t_N) \in \blacktriangle^N[i]$ with $t_i =
0$. \\[1ex]
Indeed, if $l(ws_i) < l(w)$ then there exists a reduced expression $w = r_1
\cdots r_M$ with $r_j \in S$, $r_M = s_i$. Let $(c = x_0, x_1, \dots, x_{M-1},
x_m = d)$ be a gallery of type $(r_1, \dots, r_M)$. Then $x_{M-1} \in
E_{ws_i}(c) \subset E_{< w}(c)$. Since $x_{m-1}$ and $d$ share their $i$th face
in $|E|_{\leq w}(c)$, the claim follows.

\medskip
\noindent \textsc{Claim 2}: If $d \sim_i e$ for some $d \in E_{w}(c)$ and $e \in E_{\leq
w}(c)$, then $l(ws_i) < l(w)$. \\[1ex]
For $e \in E_{< w}(c)$ this is clear. Now assume $d,e \in E_w(c)$ and $l(ws_i) =
l(w)+1$; take a reduced expression $(r_1, \dots, r_M)$ for $w$ and let $(c =
c_0, \dots, c_M = d)$ be a gallery of this type. Then $(c, \dots, c_{M-1}, d,
e)$ is of reduced type, whence $l(\delta(c, e)) < l(w)$, contradicting the choice
of $e$.

\medskip

Now let $I^- := \{i \in \{1, \dots, N\}\,|\, l(ws_i) < l(w) \}$ and 
\[\blacktriangle^{N-1}[I^-] := \bigcup_{i \in I^-} \blacktriangle^{N-1}[i].\]
By Claim 2 the map $p_0$ maps the set
\[[c] \times \blacktriangle^{n-1} \cup E_{\leq w}/E_{< w}(c)  \times
\blacktriangle^{N-1}[I^-]\]
to $*$ and is one-to-one on the complement of this set. By Proposition
\ref{SchCW} we also have 
\[E_{\leq w}(c)/E_{< w}(c) \cong S^{d(w)}.\]
Since $p_0$ is a covering map, we obtain
\begin{eqnarray*}
B_w(c) &\cong& p_0(E_{\leq w}(c)/E_{< w}(c) \times \blacktriangle^{N-1})\\
&\cong& \frac{E_{\leq w}(c)/E_{< w}(c) \times \blacktriangle^{N-1}}{[c] \times
\blacktriangle^{n-1} \cup E_{\leq w}/E_{< w}(c)  \times
\blacktriangle^{N-1}[I^-]}\\
&\cong& (S^{d(w)} \times \blacktriangle^{N-1})/(\{p\} \times 
\blacktriangle^{N-1} \cup S^{d(w)} \times  \blacktriangle^{N-1}[I^-]).
\end{eqnarray*}

If $w$ is maximal, then $I^- = \{1, \dots, N\}$, whence
\[B_w(c) \cong  (S^{d(w)} \times \blacktriangle^{N-1})/(\{p\} \times 
\blacktriangle^{N-1} \cup S^{d(w)} \times  \partial\blacktriangle^{N-1}) \cong
S^{d(w)+N-1}.\]
Otherwise we can find $i \in \{1, \dots, N\} \setminus I^-$; then we obtain a
contracting homotopy
\[H_t:  (S^{d(w)} \times \blacktriangle^{N-1})/(\{p\} \times 
\blacktriangle^{N-1} \cup S^{d(w)} \times 
\blacktriangle^{N-1}[I^-])\circlearrowleft\]
by the formula
\[H_t([x, (t_1, \dots, t_N)]) := [x, ((1-t)t_1, \dots t_i+t(1-t_i), \dots,
(1-t)t_N)]. \qedhere \]
\end{proof}

\section{Topological RGD systems and topological twin buildings}
\label{section:3.7}

We now turn to the problem of actually constructing (strong) topological twin buildings. We start by studying topological RGD systems. 

Throughout this section let $G$ be a topological group with associated RGD system 
$(\{U_\alpha\}_{\alpha \in \Phi}, T)$ and denote by 
$\Delta = \Delta(G, \{U_\alpha\}_{\alpha \in \Phi}, T)$ the associated 
twin building. We equip both halves $\Delta_{\pm}$ of $\Delta$ with the
quotient topology induced by $G$. The goal of this section is to give conditions on the
topology of $G$ which guarantee that $\Delta$ is a (strong) topological twin building. 

\subsection{Orbit closure relations}
We will first be concerned with the question concerning the  openness of the big
cells $B_+B_-$ and $B_-B_+$. The quotient map
$G \to G/B_\pm$ allows one to relate the big cell $B_\mp B_\pm$ in the building
to the big cell
$B_\mp B_\pm$ in the group: $B_\mp B_\pm$ is open considered as a subset of $G$
if and
only if $B_\mp B_\pm$ is open considered as a subset of $\Delta_\pm$. We
conclude that 
for questions concerning
the openness (and closedness) of unions of $B_\mp$-$B_\pm$-double cosets it is
in fact
irrelevant whether one uses the group or the building topology. As our first
important reduction step we will establish
the following result:
\begin{lem}\label{OpennessMain}\label{cor:unique-open-dense-orbit} Let $G$,
$\Delta$, $B_{\pm}$ as above. If $\Delta$ satisfies axioms \refttb{item:ttb-1}, 
(TTB2) and \refttb{item:ttb-3} and if the panels of $\Delta$ are non-discrete, then $B_+B_-$ and
$B_-B_+$ are open in $G$. 
\end{lem}

%
We will actually provide a more precise result: We will compute the closures of arbitrary
$B_-$-$B_+$ double cosets, culiminating in Theorem~\ref{thm:closure-relation-wrt-borel-subgroup} below, that contains Lemma~\ref{OpennessMain} as a special case. 

We begin with the following observation:
\begin{prop} \label{prop:open-orbits}
Let $c_\mp \in \Delta_\mp$, $d \in \Delta_\pm$, let $\delta^*(c_\mp,d) = w$, 
let 
$B_\mp$ be the Borel subgroup associated to $c_\mp$, let $B_\pm$ be a 
Borel subgroup opposite $B_\mp$, and let $v \in W$ such that $w\ngeq v$ in 
the
Bruhat order. 
Then there exists an open neighbourhood of $d$ in $\Delta_\pm$ disjoint from 
$B_\mp vB_\pm$.
\end{prop}
\begin{proof}
Let $\Sigma$ be a twin apartment containing $d$ and $c_\mp$ and let
$\tilde{w}^{-1} \in \Stab_{G}(\Sigma)$ be a representative 
of $w^{-1}$ that maps $d$ to the chamber in $\Sigma$ opposite $c_\mp$.
Any chamber $x \in B_\mp vB_\pm$ of $\Delta_\pm$ satisfies 
$\delta^*(c_\mp,x) = v$,
so $\delta^*(\tilde{w}^{-1}.c_\mp,\tilde{w}^{-1}.x) = v$.
As $\delta_\mp(c_\mp,\tilde{w}^{-1}.c_\mp) = w^{-1}$, 
Lemma~\ref{lemma:building-combinatorics}
allows us to conclude
\[
	\delta^*(c_\mp,\tilde{w}^{-1}.x) \in \{ w_1v \mid w_1 \text{ is a 
subexpression of } w^{-1} \}.
\]
Hence, for $X := \tilde{w}^{-1}B_\mp vB_\pm$, the hypothesis $w \ngeq v$ yields 
$1_W \notin \delta^*(c_\mp,X)$. 
Therefore, for each $a \in \delta^*(c_\mp,X)$, 
Lemma~\ref{lemma:nice-chamber-exists} provides an 
open neighbourhood $U_a$ of $\tilde w^{-1}.d$ which intersects 
$E_a^*(c_\mp)$ trivially. As $\delta^*(c_\mp,X)$ is finite, the set 
\[
	U := \bigcap_{a \in \delta^*(c_\mp,X)} U_a
\]
is an open neighbourhood of $\tilde w^{-1}.d$. Since $X \subseteq
\bigcup_{a\in \delta^*(c_\mp,X)}
E_a^*(c_\mp)$, moreover $X \cap U = \emptyset$. 
Hence $\tilde w .U$ is an open neighbourhood of $d$ satisfying 
$B_\mp vB_\pm \cap \tilde w.U = \emptyset$, as claimed.
\end{proof}
\begin{lem} \label{lemma:b-closure-first-inclusion}
Let $w \in W$, let $s \in S$, and assume that the panels of $\Delta$ are
non-discrete.
\begin{enumerate}
\item \label{item:b-closure-1}If $l(ws) > l(w)$, then the following inclusions
hold:
	\begin{enumerate}[(a)]
	\item $\overline{B_-wB_+} \supseteq B_-wsB_+$,
	\item $\overline{B_+wB_-} \supseteq B_+wsB_-$.
	\end{enumerate}
\item \label{item:b-closure-2}If $l(sw) > l(w)$, 
then the following inclusions
hold:
	\begin{enumerate}[(a)]
	\item $\overline{B_-wB_+} \supseteq B_-swB_+$,
	\item $\overline{B_+wB_-} \supseteq B_+swB_-$.
	\end{enumerate}
\end{enumerate}
\end{lem}
\begin{proof}
\begin{enumerate}
\item Let $\epsilon \in \{ +, -\}$. By Lemma \ref{lemma:b-orbits} the orbits of
$B_{-\epsilon}$ on $\Delta_\epsilon$ are given by the
$B_{-\epsilon}$-$B_\epsilon$-double cosets $B_{-\epsilon}wB_\epsilon$, $w \in
W$.
Let $c_{-\epsilon}$ be the fundamental chamber in $\Delta_{-\epsilon}$ and 
let $c \in \Delta_\epsilon$ such that that $\delta^*(c_{-\epsilon},c) = w$, 
i.e., $c$ is a representative of the $B_{-\epsilon}$-orbit
$B_{-\epsilon}wB_\epsilon$
in $\Delta_\epsilon$. Let $s \in S$ such that $l(ws) > l(w)$ and consider 
the $s$-panel $P_s(c)$ around $c$.
The projection
$d := \proj_{P_s(c)}(c_{-\epsilon})$ is the unique chamber of $P_s(c)$ 
satisfying $\delta^*(c_{-\epsilon},d) = ws$ and the group 
$\mathrm{Stab}_{B_{-\epsilon}}(P_s(c))$ acts transitively on 
the set $P_s(c) \backslash \{d\}$ of chambers distinct from $d$.
Since $P_s(c)$ is non-discrete, it follows that $d$ is contained 
in $\overline{P_s(c) \backslash \{d\}}$ and, thus, in 
$\overline{B_{-\epsilon}wB_\epsilon}$. 

We conclude that, for each $s \in S$ such that $l(ws) > l(w)$, the closure 
of $B_-wB_+$ intersects the orbit
$B_{-\epsilon}wsB_\epsilon$. Since this closure is a union of orbits, one
has for all $s \in S$ with $l(ws) > l(w)$
\[
	\overline{B_{-\epsilon}wB_\epsilon} \supseteq
B_{-\epsilon}wsB_\epsilon.
\]

\item As inversion in $G$ is a homeomorphism, one has 
$\overline{B_{-\epsilon}wB_\epsilon} \supseteq B_{-\epsilon}swB_\epsilon$ if
and only if 
$\overline{B_{\epsilon}w^{-1}B_{-\epsilon}} \supseteq
B_{\epsilon}w^{-1}sB_{-\epsilon}$. Hence the
inequality $l(w^{-1}s)=l(w^{-1}s^{-1})=l(sw) > l(w) = l(w^{-1})$ allows one 
to immediately conclude \ref{item:b-closure-2} from \ref{item:b-closure-1}.
\qedhere
\end{enumerate}
\end{proof}
Now we can establish the following theorem, which contains Lemma
\ref{OpennessMain} as a special case:
\begin{thm} \label{thm:closure-relation-wrt-borel-subgroup} \index{topological
twin building!orbit structure!of Borel subgroups}
Let $G$ be a topological group with RGD system 
$(\{U_\alpha\}_{\alpha \in \Phi}, T)$, let 
$\Delta = \Delta(G, \{U_\alpha\}_{\alpha \in \Phi}, T)$ be the associated 
twin building, equip both halves $\Delta_{\pm}$ of $\Delta$ with the
quotient topology, and assume that this twin building
topology satisfies axioms \refttb{item:ttb-1},
(TTB2) and \refttb{item:ttb-3} and that panels of
$\Delta$ are non-discrete.
Let $W$ be its Weyl group, let $\leq$ the Bruhat order of $W$, and let 
$w \in W$.
Then the following hold:
\begin{enumerate}
\item \label{item:borel-orbits-1}
\[
	\overline{B_- w B_+} = \bigsqcup_{w' \geq w} B_- w' B_+.
\]
\item \label{item:borel-orbits-2} The smallest 
open union of $B_-$-$B_+$-double cosets containing $B_-wB_+$ is
\[
	\bigsqcup_{w' \leq w} B_- w' B_+,
\]
which consists of finitely many $B_-$-$B_+$-double cosets.
\end{enumerate}
\end{thm}

This theorem is essentially \cite[Lemma
3.4]{kac-peterson-regular-functions}. A special case is 
\cite[Theorem~23 (p.~127)]{steinberg}.

\begin{proof}
\begin{enumerate}
\item 
An induction using Lemma \ref{lemma:b-closure-first-inclusion} shows that 
\[
	\overline{B_- w B_+} \supseteq \bigsqcup_{w' \geq w} B_- w' B_+.
\]

Conversely, let $x$ be an element of the 
complement $$X := \bigsqcup_{w' \ngeq w} B_- w' B_+$$ of 
$\bigsqcup_{w' \geq w} B_- w' B_+$.
We will show that $x$ lies in the interior of 
$X$.
Let $n \in \N$ be sufficiently large such that 
$x \in \bigcup_{l(v) \leq n} B_+vB_+$. 
The intersection $\left( \bigsqcup_{w' \geq w} B_- w' B_+ \right) \cap 
\left( \bigcup_{l(v) \leq n} B_+vB_+ \right)$ meets finitely many 
$B_-$-$B_+$-double cosets. Let $A \subset W$ be a finite set such that these
double
cosets are given by the family $\{B_-aB_+\}_{a \in A}$. For every $a \in A$,
Proposition \ref{prop:open-orbits} provides an open neighbourhood 
$U_a$ of $x$ in $G$ disjoint from $B_- a B_+$.

Then $\bigcap_{a \in A} U_a \cap \left( \bigcup_{l(v) \leq n} B_+vB_+ \right)$
is open in $\bigcup_{l(v) \leq n} B_+vB_+$, contains $x$ and intersects 
$\bigsqcup_{w' \geq w} B_- w' B_+$ trivially. 
Thus, this intersection is an open neighbourhood of $x$ in 
$X\cap \left( \bigcup_{l(v) \leq n} B_+vB_+ \right)$, and hence $x$ is an 
interior
point of $X \cap \left( \bigcup_{l(v) \leq n} B_+vB_+ \right)$.
As $x$ was arbitrary, we conclude that $X \cap \left(\bigcup_{l(v) \leq n}
B_+vB_+\right)$ is 
open in $\bigcup_{l(v) \leq n} B_+vB_+$ for each $n \in \N$.
By axiom \refttb{item:ttb-3} 
$\Delta_+ = \lim_\to \bigcup_{l(v) \leq n} B_+vB_+$, and so $X$ is open in 
$\Delta_+$ and, thus, in $G$.

\item Define the finite set $X_w := \{ v \in W \mid v \nleq w, \exists s \in S
\mbox{
such that $sv \leq w$ or $vs \leq w$} \}$. Then, for any $w' \in W$, one has
$w' \nleq w$ if and only if there exists $v \in X_w$ such that $v \leq w'$. 
Hence, by \ref{item:borel-orbits-1},
\[
\bigsqcup_{w' \leq w} B_- w' B_+ = G \setminus \bigcup_{x \in X_w} \overline{B_- x
B_+}.
\]
Since $\bigcup_{x \in X_w} \overline{B_- x B_+}$ is a finite union of closed sets,
 it is closed, and so its complement, $\bigsqcup_{w' \leq w} B_- w' B_+$, is 
open. 

Moreover, if $U$ is an arbitrary open union of $B_-$-$B_+$-double cosets
containing $B_-wB_+$, then by \ref{item:borel-orbits-1} for each $w' \leq w$
one necessarily has $B_-w'B_+ \subseteq U$. \qedhere
\end{enumerate}
\end{proof}

\begin{rem}
Theorem~\ref{thm:closure-relation-wrt-borel-subgroup} states that there
exists a closed $B_-$-$B_+$-double coset if and only if there exists a
maximal element of $W$ with respect to the Bruhat order. This is the case if
and only if $W$ is spherical. 
\end{rem}

\subsection{A group-theoretic criterion for twin building topologies}
\label{section:5.2}

As before, let $G$ be a topological group with RGD system 
$(\{U_\alpha\}_{\alpha \in \Phi}, T)$ and denote by $\Delta$ the associated twin
building. Our explicit projection formulae
(Theorem~\ref{thm:projection-formula}) immediately yield the following
proposition:
\begin{prop} \label{prop:ttb2}
If the bijective product map 
$m : U_+ \times T \times U_- \to B_+B_-$ is open, then $\Delta$
satisfies (TTB2).
If, moreover, $B_+B_-$ is open in $G$, then 
$\Delta$ also satisfies (TTB2$+$) and (TTB5).  
\end{prop}
\begin{proof} The set of pairs of chambers in $G/B_+ \times G/B_-$ at 
codistance $1 \in W$ is equal to
$\{ (gB_+, hB_-) \mid g^{-1}h \in B_+ B_-\}$.
Given $s \in S$,
Theorem~\ref{thm:projection-formula} 
implies $\proj^*_{P_s(hB_-)}(gB_+) = h \rho_w(g^{-1}h)^{-1}sB_-$. 
By Lemma~\ref{ContinuityFoldingMain} therefore $\Delta$ satisfies
(TTB2).

If $B_+B_-$ is
open in $G$, then by continuity of multiplication and inversion the set $\{ (g,h) \in G \times G \mid g^{-1}h \in B_+B_- \}$ is open in $G \times G$. As the canonical quotient map $G \times G \to (G \times G)/(B_+ \times B_-)$ is open, the set $\{ (gB_+, hB_-) \mid g^{-1}h \in B_+ B_-\}$ of opposite pairs of chambers is open in $\Delta_+ \times \Delta_- = G/B_+ \times G/B_-$, i.e. (TTB5) holds.

In this case the set $B_-B_+/B_+ \times B_+B_-/B_- \subset \Delta_+ \times \Delta_-$ is open, whence also $O := (B_-B_+/B_+ \times B_+B_-/B_-) \cap \Delta_1$ is open. Note that for $(c,d) \in O$ the chamber $c$ is opposite the chamber $c_- := B_- \in G/B_-$, as $B_-B_+$ is exactly the set of chambers of $\Delta_+$ opposite $c_-$. Therefore, by sharp transitivity of $U_-$ on the chambers opposite $c_-$, there exists a unique $g_{(c,d)} \in U_-$ with $g_{(c,d)}(c) = c_+ := B_+ \in G/B_+$. The resulting map $\nu_O : O \mapsto U_- : (c,d) \mapsto  g_{(c,d)}$ is continuous, since the map $U_- \times T \times U_+ \to B_-B_+ : (u_-,t,u_+) \mapsto u_-tu_+$ is open by hypothesis. 

Considering the restriction  ${p_s}_{|O} : O \to \Delta_+ \cup \Delta_- : (c,d) \mapsto \pro^*_{P_s(c)}(d)$ of the projection map from axiom (TTB2$+$) we compute
\begin{eqnarray*}
\pro^*_{P_s(c)}(d) & = & g^{-1}_{(c,d)}\left(\pro^*_{P_s(g_{(c,d)}(c))}(g_{(c,d)}(d))\right) \\  & = & \nu_O(c,d)^{-1}\left(\pro^*_{P_s(c_+)}(\nu_O(c,d)(d))\right).
\end{eqnarray*}
Therefore (TTB2) and continuity of $\nu_O$ imply continuity of ${p_s}_{|O}$. Since, by the strong transitivity of the action of $G$ on $\Delta$, the set $\Delta_1$ is covered by $G$-translates of the open set $O$, the validity of (TTB2$+$) follows.
%
%
%
\end{proof}
Combining the proposition and Lemma \ref{OpennessMain} we obtain:
\begin{thm} \label{thm:top-twin-building}
 Let $G$ be a topological group with root group datum $(\{U_\alpha\}_{\alpha
\in 
\Phi}, T)$, let $\Delta = \Delta(G, \{U_\alpha\}_{\alpha \in \Phi}, T)$
be the associated twin building, and equip both halves with the 
quotient topology. Assume the following:
\begin{itemize}
 \item[(i)] $B_{\pm}$ are closed.
 \item[(ii)] $G = \lim_\to \left( \bigcup_{l(w) \leq n} B_+wB_+ \right) = 
\lim_\to \left( 
\bigcup_{l(w) \leq n} B_-wB_- \right)$.
 \item[(iii)] The multiplication map $m\colon U_+ \times T \times U_- \to
B_+B_-$ is open.
\item[(iv)] Panels in $\Delta$ are compact.
\end{itemize}
Then $\Delta$ is a strong topological twin building, and the parabolic subgroups
\[P^J_{\pm} := B_{\pm}W_JB_{\rm}\quad (J \subset S)\]
are closed in $G$.
\end{thm}
\begin{proof} We first consider the axioms which follow immediately from our assumptions: (i) implies (TTB1), (ii) implies (TTB3) and (iv) implies (TTB4). Using (iii) and the first part of
Proposition~\ref{prop:ttb2} we deduce that also (TTB2) holds. By Lemma \ref{OpennessMain} this implies that
$B_+B_-$ and $B_-B_+$ are open. Thus the second part of Proposition~\ref{prop:ttb2} applies and shows that also (TTB2$+$) and
(TTB5) hold. (TTB6) follows from the general fact that given subgroups $H_1<H_2$ of a topological group the canonical map $\pi: G/H_1 \to G/H_2$ is open, since for an open set $UH_1 \subset G/H_1$ the saturation $\pi^{-1}(\pi(UH_1)) = UH_2 = \bigcup UH_1h_2$ is open as the union of (open) translates. Finally, residues in $\Delta$ are closed by Lemma~\ref{residueclosed}, whence their stabilizers, i.e., parabolic subgroups are closed, which in turn implies (TTB1$+$).
\end{proof}
\begin{rem}\label{ttb4holds}
For an $\mathbb{F}$-locally split root group datum there is an easy
condition which guarantees property (iv) above. Indeed, the panels
of $\Delta$ are homeomorphic to $\mathbb{P}_1(\mathbb{F})$, the building of
$\mathrm{SL}_2(\F)$.  
If $\mathbb{F}$ is Hausdorff, non-discrete, $\sigma$-compact and locally 
compact,
these are compact (cf.\  \cite[Proposition 14.5 and Corollary
14.7]{salzmann}).
\end{rem}

\section{Topological split Kac--Moody groups} \label{section6}

In this section we return to the classification problem for $k$-split topological twin buildings over a local field $k$, which we studied in Section \ref{SecDynkin}. Our goal is to show that every Dynkin tree is  topologically $k$-integrable. 

Our starting point is the observation that every Dynkin tree is abstractly $k$-integrable; in fact the corresponding twin building is the twin building of the associated split Kac--Moody group. In view of this observation, our task is to topologize the twin buildings of split Kac--Moody groups over local fields. We will in fact topologize the Kac--Moody groups themselves and then apply the criterion developed in Theorem~\ref{thm:top-twin-building} to obtain topological twin buildings. 

Along the way we develop a topological theory of Kac--Moody groups which is of independent interest. In particular, we will revisit various results and techniques developed and announced in  \cite{kac-peterson-regular-functions}.



\subsection{Abstract split Kac--Moody groups} \label{abstractKM}

For the convenience of the reader let us start with the definition of a Kac--Moody root datum.

\begin{defin} \label{cartanmatrix}
A \textbf{generalized Cartan matrix} is a
matrix $A = (a_{ij})_{1 \leq i, j \leq n} \in \Z^{n \times n}$ satisfying
$a_{ii} = 2$, $a_{ij} \leq 0$ for $i \neq j$, and $a_{ij} = 0$ if and only
if $a_{ji} = 0$.

Let $I = \setn$ and let $A = (a_{ij})_{1\leq i,j\leq n}$ be a generalized 
Cartan matrix. A quintuple $\calD = (I, A, \Lambda, \{c_i\}_{i\in I}, 
\{h_i\}_{i\in I})$ is called a \textbf{Kac--Moody root datum} if $\Lambda$ 
is a free $\Z$-module, each $c_i$ is an element of $\Lambda$ and 
each $h_i$ is in the $\Z$-dual $\Lambda^\vee$ of $\Lambda$ such that for 
all $i,j \in I$ one has $h_i(c_j) = a_{ij}$.

\end{defin}

Following \cite[3.6]{tits-algebra} to a Kac--Moody root datum $\calD$
one associates
 a triple $\calF = (\calG, \{\phi_i\}_{i\in I}, \eta)$, where $\calG$ is a
group functor on the category of commutative unital rings, the $\phi_i$ are 
maps $\mrm{SL}_2(R) \to \calG(R)$, and $\eta$ is a natural transformation 
$\Hom_{\Z-{\rm alg}}(\Z[\Lambda], -) \to \calG$ such that the following assertions hold:
	\begin{enumerate}[(KMG1)]
		\item \label{item:kmg-1}If $\F$ is a field, then the group
$\calG(\F)$ is generated by the images of the $\phi_i$ and $\eta(\F)$.
		\item \label{item:kmg-2}For all rings $R$ the homomorphism
$\eta(R): \Hom_{\Z-{\rm alg}}(\Z[\Lambda],R) \rightarrow \calG(R)$ is injective.
		\item \label{item:kmg-3}Given a ring $R$, $i\in I$ and 
$u\in R^\times$, one has $$\phi_i\left(
				\begin{array}{cc}
					u & 0     \\
					0 & u^{-1}
				\end{array}\right) = \eta\left(\lambda \mapsto
u^{h_i(\lambda)}\right).$$
		\item \label{item:kmg-4}If $R$ is a ring, $\F$ is a field and
$\iota: R \rightarrow \F$ is a monomorphism, then $\calG(\iota): \calG(R)
\rightarrow \calG(\F)$ is a monomorphism as well.
		\item \label{item:kmg-5}If $\gotg$ is the complex Kac--Moody
algebra 
of type $A$, then there exists a 
homomorphism $\mathrm{Ad}: \calG(\C) \rightarrow \Aut(\gotg)$ such that 
$\ker(\mathrm{Ad})\subseteq \eta(\mathbb{C})(\Hom_{\Z-{\rm alg}}(\Z[\Lambda],\mathbb{C}))$ and for a given 
$z\in \C$ one has
			\begin{align*}
				\mathrm{Ad}\left(\phi_i\left(
					\begin{array}{cc}
						1 & z \\
						0 & 1
					\end{array}\right)\right) & =
\exp(\mathrm{ad}_{ze_i}),\\
				\mathrm{Ad}\left(\phi_i\left(
					\begin{array}{cc}
						1 & 0 \\
						z & 1
					\end{array}\right)\right) & =
\exp(\mathrm{ad}_{-zf_i});
			\end{align*}
where $\{e_i, f_i\}$ are part of a standard $\mathfrak{sl}_2$-triple
for the fundamental Kac--Moody sub-Lie algebra corresponding to the simple
root $\alpha_i$; furthermore, for every homomorphism 
$\gamma \in \Hom_{\Z-{\rm alg}}(\Z[\Lambda],\mathbb{C})$ one has
			\[
				\mathrm{Ad}(\eta(\mathbb{C})(\gamma))(e_i) = \gamma(c_i)
\cdot e_i, \quad \mathrm{Ad}(\eta(\mathbb{C})(\gamma))(f_i) = \gamma(-c_i) \cdot
f_i.
\]
\end{enumerate}
The Kac--Moody root datum $\calD$ is called \index{Kac--Moody root
datum!centred}\textbf{centred} if the following stronger version of
\refkmg{item:kmg-1} is
satisfied: If $\F$ is a field, then the group $\calG(\F)$ is generated by the
images of the $\phi_i$.

For a given Kac--Moody root datum $\mathcal{D}$ the group 
$G_\calD(R) := \calG(R)$ is called a \textbf{split Kac--Moody group} of type 
$\calD$ over $R$.

The main result of \cite{tits-algebra} states that under some non-degeneracy
assumptions any functor defined on the category of fields satisfying the above
axioms must coincide with $\calG$.

A split Kac--Moody group over a field is an example of a group with 
an RGD system by the following result.  

\begin{prop}[{\cite[Proposition 8.4.1]{remy}, \cite[Lemma 1.4]{main-caprace}}]
\label{prop:kmg-has-rgd} \index{RGD system!of a Kac--Moody group}
Let $\F$ be a field, let $\calD = (I, A, \Lambda, \{c_i\}_{i\in I}, 
\{h_i\}_{i\in I})$ be a Kac--Moody root datum, and let 
$G_\calD(\F) := \calG(\F)$ be the corresponding split Kac--Moody group of 
type $\calD$ over $\F$.  Then $G_\calD(\F)$ admits an RGD system as follows. Let
$M(A)$ be the 
associated Coxeter matrix of type $(W, S)$ and choose a set of simple 
roots $\Pi = \{ \alpha_i \mid i \in I\}$ such that the reflection 
associated to $\alpha_i$ is $s_i \in S$. Define the {\bf set of real roots} as
$\Phi^{re} := W.\Pi$. Given $i \in I$, let $U_{\alpha_i}$ 
and $U_{-\alpha_i}$ be the respective images of the subgroups of strictly upper, resp.\ strictly lower triangular 
matrices of the matrix group $\mrm{SL}_2(\F)$ under the map $\phi_i$, and denote by $T$ the image of $\eta(\mathbb F)$ in $G_\calD(\F)$.

Then $T = \bigcap_{\alpha \in \Phi^{re}} N_{G_\calD(\F)}(U_\alpha)$, $W
\cong N_{G_\calD(\F)}(T)/T$ and
$(G_\calD(\F), \{U_\alpha\}_{\alpha \in \Phi^{re}}, T)$ is an RGD system. \qed
\end{prop}

In the sequel we refer to $T$ as the \emph{standard maximal torus} of $G_\calD(\F)$.
Note that, since the matrix group $\mathrm{SL}_2$ over a field is generated by its subgroups of strictly upper and strictly lower triangular matrices, the concept of being centred coincides with the one introduced in Section~\ref{RGD system}.

\begin{rem}\label{adjointform}
In general the action of a split Kac--Moody group $G_\calD(\F)$ on the
associated twin building will not be effective; however the kernel $\mathcal
Z_{\calD}(\F)$ of this action always equals the centre of $G_\calD(\F)$, which in turn is contained in $T$ (cf.\
\cite[Proposition~9.6.2]{remy}). 
\end{rem}



From now on we will reserve the letter $G$
to denote a
split Kac--Moody group $G_{\calD}(\F)$ over a field $\F$. Note that Propositions~\ref{prop:rgd-twin-tits} and \ref{prop:kmg-has-rgd} in particular imply that $G$ admits a twin $BN$-pair. Using Proposition~\ref{prop:rgd-twin-tits} we introduce the following additional notation. 

For every $k \geq 0$ we set
\[
	G_k^\pm := \bigcup_{l(w) \leq k} B_\pm w B_\pm
\]
and $G_k := G_k^+ \cap G_k^-$. For every $k$-tuple  $\overline{\alpha} = (\alpha_1, \ldots, \alpha_k) \in
\Pi^k$ of simple roots we will denote by
\[G_{\overline{\alpha}} := 
G_{\alpha_1}\cdots G_{\alpha_k} \subseteq G\]
the subset of $G$ consisting of products of the form
\[g = g_1 \cdots g_k, \quad g_j \in G_{\alpha_j}.\]
Note that as a special case we have $G_{\alpha} = G_{(\alpha)}$. 

Similarly, we denote by $TG_{\overline{\alpha}}$ the image of $T \times G_{\overline\alpha}$ under the product map.

Note that the set of indices carries a natural partial order: We write
$\overline{\alpha} \leq
\overline{\beta}$ provided $\overline{\alpha}$ appears as an ordered subtuple of
$\beta$; in this case there is an obvious embedding $G_{\overline \alpha}
\hookrightarrow G_{\overline \beta}$. We record the following inclusion
relations for later use:
\begin{prop} \label{prop:G-I-in-G-n}
Let $B_+$, $B_-$ be the standard Borel subgroups of the adjoint form of
$G_\calD(\F)$, let 
$k \in \N$, and let $\overline{\alpha} = (\alpha_1,...,\alpha_k) \in
\Pi^k$. Then $TG_{\overline{\alpha}} \subseteq G_k^+ \cap G_k^-$.
\end{prop}

\begin{proof}
We prove the result by induction on $|\overline{\alpha}| = k$. For $k = 0$ 
the claim follows from the inclusion $T \subseteq B_+ \cap B_-$, cf.\ Proposition~\ref{prop:rgd-twin-tits}. Let $k>0$ and assume that for all 
$\overline{\beta}$ 
with $|\overline{\beta}| < |\overline{\alpha}|$ the set $TG_{\overline{\beta}}$
is contained in $G_{|\overline{\beta}|}$. Hence, for 
$\overline{\alpha_0} = (\alpha_1, \ldots, \alpha_{k-1})$, the induction 
hypothesis yields $TG_{\overline{\alpha_0}} \subseteq
G_{|\overline{\alpha_0}|}$. For $\epsilon \in \{\pm\}$ we have
Bruhat decompositions $\mathrm{SL}_2(\mathbb{F}) \cong G_{\alpha_k} = B_{\alpha_k}^\epsilon \cup 
B_{\alpha_k}^\epsilon s_{\alpha_k} B_{\alpha_k}^\epsilon$, where 
$B_{\alpha_k}^\epsilon := B_\epsilon \cap G_{\alpha_k}$. Hence
\begin{eqnarray*}
	TG_{\overline{\alpha}} = TG_{\overline{\alpha_0}} \cdot G_{\alpha_k} &
\subseteq & \left(\bigcup_{l(w) \leq k-1} B_\epsilon w B_\epsilon\right) \cdot
\left(B^\epsilon_{\alpha_k} \cup B^\epsilon_{\alpha_k} s_{\alpha_k}
B^\epsilon_{\alpha_k}\right)\\
		& \subseteq & \left(\bigcup_{l(w) \leq k-1} B_\epsilon w
B_\epsilon\right) \cdot \left(B_\epsilon \cup B_\epsilon s_{\alpha_k}
B_\epsilon\right)\\
		& \stackrel{\text{\reftbn{item:tbn-1}}}{\subseteq} & \bigcup_{l(w) \leq
k-1} B_\epsilon w B_\epsilon \cup B_\epsilon ws_{\alpha_k} B_\epsilon\\
		& \subseteq & \bigcup_{l(w) \leq k} B_\epsilon w B_\epsilon.
\end{eqnarray*}
The claim follows.
\end{proof}

\subsection{The centre and the adjoint representation} \label{adjointrep}

While there is a unique twin building associated with a given generalized Cartan matrix $A$ over a given field $\mathbb F$, there exist several corresponding root group data and hence several corresponding Kac--Moody groups, like in the spherical case;
%
an example to keep in mind are the groups ${\rm SL}_n(\mathbb{F})$ and ${\rm GL}_n(\mathbb{F})$. Their central quotients  ${\rm PSL}_n(\mathbb{F})$ and ${\rm PGL}_n(\mathbb{F})$ are isomorphic for $\mathbb{F} = \mathbb{C}$ but distinct for $\mathbb{F} = \mathbb{R}$. For the study of topological buildings arising from Kac--Moody groups these differences do not matter much, but there are some subtleties.

\medskip

Using the notation of  Definition \ref{cartanmatrix} and generalizing classical terminology in the spherical case, a Kac--Moody root datum $\calD$ will be called 
\textbf{simply connected} if the set $\{h_i \mid i\in I\}$ is a $\Z$-basis 
of $\Lambda^\vee$ and
\textbf{adjoint} if the set $\{c_i \mid i\in I\}$ is a $\Z$-basis of 
$\Lambda$. We will also refer to the corresponding Kac--Moody groups as simply connected or adjoint. For example, the algebraic group ${\rm SL}_n$ is simply connected and the algebraic group ${\rm PGL}_n$ is adjoint in this sense.

\medskip

In general, the centre of an adjoint Kac--Moody group will be trivial, whereas the centre of a simply connected group is typically non-trivial. However, we observe:
\begin{lem}\label{ToriIsogeneous} The centre of a simply connected Kac--Moody group over an arbitrary field $\mathbb F$ is finite.
\end{lem}
\begin{proof}
Let $G := G_{\calD}(\F)$ be a simply connected Kac--Moody group. Its centre is contained in any of its maximal tori by \cite[Proposition~9.6.2]{remy}. Moreover, also by \cite[Proposition~9.6.2]{remy}, the centre of $G$ equals the kernel of the adjoint representation. To prove the claim it is therefore enough to show that the adjoint representation induces an isogeny from a maximal torus $T^{\mathrm{sc}}$ onto its image.

By \cite[Definition~7.1.1]{remy} the group of characters $\Lambda$ of a maximal torus of a Kac--Moody group is a free abelian group of finite rank and by \cite[8.2.1]{remy} the torus functor is defined as the group functor $\mathrm{Hom}_{\text{$\mathbb{Z}$-alg}}(\mathbb{Z}[\Lambda],-)$. By \cite[7.1.2]{remy} there exists a natural embedding $\Lambda^\mathrm{ad} \to \Lambda^\mathrm{sc}$ of the group of characters of an adjoint torus into the group of characters of a simply connected torus of the same type. This yields an injective $\mathbb{Z}$-algebra homomorphism $\mathbb{Z}[\Lambda^{\mathrm{ad}}] \to \mathbb{Z}[\Lambda^{\mathrm{sc}}]$, which in turn provides a surjective morphism $\mathrm{Hom}_{\text{$\mathbb{Z}$-alg}}(\mathbb{Z}[\Lambda^\mathrm{sc}],-) \to \mathrm{Hom}_{\text{$\mathbb{Z}$-alg}}(\mathbb{Z}[\Lambda^\mathrm{ad}],-)$.. As the ranks of $\Lambda^\mathrm{ad}$ and $\Lambda^\mathrm{sc}$ coincide, this implies that the kernel of this morphism is $0$-dimensional, i.e., finite.
\end{proof}
As before, let $\calD$ be a root group datum associated with a generalized Cartan matrix $A$, let $\mathbb F$ be a field of arbitrary characteristic, let $G := G_{\calD}(\F)$ and let $Z := Z(G)$ the centre of $G$. We recall from \cite[Section 7.3.1]{remy} how to construct a faithful representation of $G/Z(G)$ starting from the 
complex Kac--Moody algebra $\gotg$ associated 
to the generalized Cartan matrix $A$, which generalizes the classical adjoint representation in the spherical case.


For this we denote by $\calU := \calU(\gotg)$
the
\index{Kac--Moody algebra!universal enveloping algebra}universal 
enveloping algebra of $\gotg$. For each $u \in \calU$, let 
$u^{[n]} := (n!)^{-1}u^n$ and
$\binom u n := (n!)^{-1} \cdot u \cdot (u-1) \cdot \cdots \cdot (u-n+1)$.

Let $Q := \sum_{\alpha \in \Pi} \Z \alpha$ be the free abelian group 
generated by the simple roots. Then, as in \cite[Section 7.3.1]{remy}, the
algebras $\calU$ and $\gotg$ admit an abstract 
\index{Kac--Moody algebra!universal enveloping algebra!$Q$-grading}$Q$-grading
by declaring $e_i$ and $f_i$ to be of degree $\alpha_i$ and $-\alpha_i$, 
respectively, and extending linearly.

With this notation, set $\calU_0$ to be the subring of $\calU$ generated 
by the elements of degree $0$ of the form $\binom h n$, where 
$h \in \goth, n \in \N$. Moreover, define $\calU_{\alpha_i}$ and 
$\calU_{-\alpha_i}$ to be the subrings $\sum_{n \in \N} \Z e_i^{[n]}$ and 
$\sum_{n \in \N} \Z f_i^{[n]}$, respectively.
Let $\calU_\Z$ be the subring of $\calU$ generated by $\calU_0$ and 
$\{\calU_{\pm \alpha} \mid \alpha \in \Pi\}$. Then 
\index{Kac--Moody algebra!$\Z$-form}$\calU_\Z$ is a $\Z$-form of $\calU$,
i.e., the canonical map $\calU_\Z \otimes_\Z \C \to \calU$ is a bijection, 
cf.\ \cite[Section~4]{tits-algebra}, \cite[Proposition~7.4.3]{remy}. 

This construction allows one to replace the field $\C$ with an arbitrary
field $\F$: 
define $\calU_\F := \calU_\Z \otimes_\Z \F$. Let $\Aut_\filt(\calU_\F)$ be the
group of $\F$-linear automorphisms of $\calU_\mathbb{F}$ which preserve
the above $Q$-grading.

The resulting adjoint action of a Kac--Moody group has the following nice properies.

\begin{prop}[{\cite[Proposition 9.5.2]{remy}}] \label{prop:representation-remy}
\index{Kac--Moody group!adjoint representation}
Let $G_\calD(\F)$ be a split Kac--Moody group over a field $\F$ and let $T$
denote
its standard maximal torus. Then there exists a morphism of groups
\[
	\mrm{Ad}: G_\calD(\F) \to \Aut_\filt(\calU_\F)
\]
which is characterized by the following axioms, where $\alpha_i$ is a real root,
$r \in \F$ and $h \in T$:
\begin{enumerate}
\item \label{item:root-group-action}$\Ad(x_{\alpha_i}(r)) = \exp (\ad_{e_i}
\otimes r) = \sum_{n = 0}^\infty \frac{(\ad_{e_i})^n}{n!} \otimes r^n$,
\item \label{item:torus-fixes-U0}$\Ad(T)$ fixes $\calU_0$,
\item \label{item:action-of-torus}$\Ad(h)(e_i \otimes r) = h^*(\alpha_i^\vee)
(e_i \otimes r)$. 
\end{enumerate}
The kernel of this representation coincides with the centre of the group  $G_\calD(\F)$.\qed
\end{prop}
In the sequel we will often tacitly identify $G/Z(G)$ with its image under the adjoint representation and thereby consider $G/Z(G)$ as a subgroup of $\Aut_\filt(\calU_\F)$.

\subsection{The Kac--Peterson topology I: $k_\omega$-property}


Given a local field $\F$ we are going to construct a group topology on every
split Kac--Moody group over $\F$.
Throughout we will imitate closely the arguments given in \cite[Section
6]{final-group-topologies} for unitary forms of complex Kac--Moody groups. The
main difference is that we
use the adjoint representation on the
associative algebra $\mathcal{U}_\F$
instead of the Kac--Moody Lie algebra. This will allow us to include the case
of positive characteristic as well.

Throughout we will reserve the letter $\F$ to denote a local field, the
letter $\calD$ to denote a Kac--Moody root datum, and the
letter $G$ to denote the
associated Kac--Moody group $G := G_\calD(\F)$. As above, we then denote by $Z(G)$ the centre of $G$ and identify $G/Z(G)$ with its image under the adjoint representation. 

Our goal is to define a group topology on $G$ that is induced by the topology on $\mathbb{F}$. We will actually provide two constructions of such a group
topology on $G$ in Definitions~\ref{kpdef} and \ref{def:kac-peterson-topology} below. In Proposition~\ref{universalKP}, however, we will show that
these two topologies in fact coincide.

We equip
the space 
$\F^{n \times n}$ of $(n \times n)$ matrices over $\F$ with the product 
topology and obtain a Hausdorff group topology $\mathcal O_\F$  on the open
subset $\mathrm{GL}_n(\F)$. We also obtain Hausdorff group topologies on
$\mathrm{SL}_n(\F) < \mathrm{GL}_n(\F)$ and the central quotient $\mathrm{PSL}_n(\F)$,
which we denote by the same letter $\mathcal O_\F$. The rank one subgroups $G_\alpha$ of $G$ are isomorphic to $\mathrm{SL}_2(\F)$ or $\mathrm{PSL}_2(\F)$ and, hence, by this definition can be considered as topological groups.


Starting from this topology on the rank one subgroups and the torus we now give our first definition of a topology, which is inspired by a construction
of Kac and Peterson in the complex case \cite{final-group-topologies}, \cite{kac-peterson-topology}. We start by defining topologies on the pieces
$G_{\overline \alpha}$ introduced above in Section~\ref{abstractKM}.

\begin{defin} \label{topologyonG}
Let $G$ be a split Kac--Moody group over a
local field $\F$ and
let $\overline{\alpha} = (\alpha_1, \ldots, \alpha_k) \in \bigcup_{l \in \mathbb{N}}\Pi^l$. Equip the torus $T$ with its Lie group topology $\tau_{\F}$ and the rank one subgroups $G_{\alpha_1}, \dots, G_{\alpha_k}$ with the topology $\mathcal O_\F$. Then we denote by $\tau_{\overline \alpha}$ the quotient topology on $TG_{\overline \alpha}$ with respect to the surjective map
\[p_{\overline{\alpha}}:(T, \tau_\F) \times (G_{\alpha_1}, \mathcal O_\F) \times \dots \times (G_{\alpha_k}, \mathcal O_\F) \to TG_{\overline \alpha}.\]
\end{defin}
Note that the topological spaces $(G_{\overline{\alpha}},
\tau_{\overline{\alpha}})$ introduced in Definition~\ref{topologyonG} form a directed system with respect to the sub-tuple relation $\overline{\alpha} \leq \overline{\beta}$
defined in Section~\ref{abstractKM}. Moreover, the set underlying $G$ is the ascending union (i.e.\ direct limit) of the underlying directed system of sets.
\begin{defin} \label{kpdef}
The {\bf Kac--Peterson topology}\footnote{see Remark \ref{KacPetersonOriginal} for the relation to their original definition} $\tau_{KP}$ on $G$ is the direct limit
topology with respect to the directed system $\{(TG_{\overline{\alpha}},
\tau_{\overline{\alpha}})\}_{\overline{\alpha}}$.
\end{defin}
It is not obvious at all that this topology defines a group topology on $G$. The main problem one encounters in establishing continuity of the multiplication is that 
for general topological spaces $G_i$,
\[\lim_\to G_i \times \lim_\to G_j \not \cong \lim_\to (G_i \times G_j).\]
However, such an exchange of limits is possible provided the pieces $G_i$ are $k_\omega$ \cite[Propositions~4.2, 4.7]{final-group-topologies}. Our strategy for showing that $\tau_{KP}$ is a group topology will thus be to establish the $k_\omega$-property for certain subsets of $(G, \tau_{KP})$. Once we can exchange direct limits and finite products freely, the proof becomes trivial. In order to establish the desired $k_\omega$-property we will have to assume $\F$ to be locally compact and $\sigma$-compact. Somewhat surprisingly, the hardest part in establishing the $k_\omega$-property of $\tau_{KP}$ is to show that it is Hausdorff. For this we need to use some topological properties of the adjoint representation, which we summarize in the following proposition.
Here we topologize all finite-dimensional $\F$-vector spaces $V \cong \F^n$ with the natural product topology with respect to the locally compact topology of $\F$ and all corresponding general linear groups ${\rm GL}(V)$ with the topology $\mathcal O_\F$ defined above.
\begin{prop}\label{prop:ad-finite-dimensional-submodule}
\label{cor:submodules-and-continuity}\label{
rem:prop:ad-finite-dimensional-submodule}
Let $\F$ be a local field and let $G$ be a split Kac--Moody group over $\mathbb{F}$.
For each $v \in \mathcal{U}_\F$ there exists a family of subspaces
$\{V^v_{\overline{\alpha}}\}_{\overline{\alpha}}$ of $\calU_\F$
with the following properties:
\begin{enumerate}
 \item $\dim V^v_{\overline{\alpha}} < \infty$.
\item The image of the orbit map $TG_{\overline{\alpha}} \to \mathcal{U}_\F : g \mapsto g(v)$ is contained in  $V^v_{\overline{\alpha}}$.
\item The orbit map $(TG_{\overline{\alpha}},\tau_{\overline{\alpha}}) \to (V^v_{\overline{\alpha}},\mathcal{O}_\F) : g \mapsto g(v)$ is continuous.
 \item If $\overline{\alpha}\leq \overline{\beta}$ then
$V^v_{\overline{\alpha}} \leq V^v_{\overline{\beta}}$.
\item $V^v_{\overline{\alpha}}$ is $TG_{\alpha_1}$-invariant, where $\overline{\alpha} = (\alpha_1,\alpha_2,...,\alpha_k)$.
 \item The kernel of the
map $\rho_{\alpha_1} := {\rm
Ad}|_{TG_{\alpha_1}}^{{\rm GL}(V^v_{\overline{\alpha}})}:
TG_{\alpha_1} \to {\rm GL}(V^v_{\overline{\alpha}})$ is contained in $Z(G)$. 
\item The image of $U_{\pm \alpha_1} \subset G_{\alpha_1}$ under $\rho_{\alpha_1}$ is a closed subgroup of $({\rm GL}(V^v_{\overline
\alpha}),
\mathcal O_\F)$.
\end{enumerate}
\end{prop}

\begin{proof}
By Proposition~\ref{prop:representation-remy}(ii)(iii) we may disregard the finite-dimensional torus $T$.

Let $v_1, ....., v_t \in \mathcal{U}_\F$ and $\alpha \in \Pi$. The Gauss algorithm/Bruhat decomposition of the
fundamental rank one subgroup $G_{\alpha}$ of
$G$ implies 
that the product map $U_\alpha \times U_{-\alpha} \times U_\alpha 
\times U_{-\alpha} \to G_{\alpha}$ is surjective
(\cite[Lemma~24]{steinberg}). 
The adjoint action of $u_\alpha(q) \in U_\alpha \subseteq G_{\alpha}$ on
$\calU_\mathbb{F}$ is given by $\Ad(u_\alpha(q)) = \sum_{n = 0}^\infty 
\left(\frac{(\ad_{e_i})^n}{n!} \otimes q^n\right)$, see 
Proposition~\ref{prop:representation-remy}.
Hence, by the Bruhat decomposition, the vector space
\[
	\sum_{i=1}^t \sum_{k,l,m,n \in \N} \left<
\left(\frac{(\ad_{e_\alpha})^k}{k!}
\otimes 1\right) \left(\frac{(\ad_{f_\alpha})^l}{l!} \otimes 1\right)
\left(\frac{(\ad_{e_\alpha})^m}{m!} \otimes 1\right)
\left(\frac{(\ad_{f_\alpha})^n}{n!} \otimes 1\right)
v_{i}\right>_\F
\]
contains $\sum_{i=1}^t \langle G_{\alpha}.v_i \rangle_\F$. 
By construction this vector space is $\mathrm{Ad}|_{G_{\alpha}}$-invariant. From the
local 
nilpotency of $\ad_{e_\alpha}$ and $\ad_{f_\alpha}$, we may conclude that the above 
sum is finite and hence this vector space has finite dimension.

Using bases, this argument shows that each finite-dimensional subspace of $\mathcal{U}_\F$ is contained in a $G_\alpha$-invariant finite-dimensional subspace of $\mathcal{U}_\F$.

It is therefore immediate from Proposition \ref{prop:representation-remy} that for $v \in \mathcal{U}_\F$ there exists a family of subspaces
$\{V^v_{\overline{\alpha}}\}_{\overline{\alpha}}$ of $\calU_\F$ such that (i), (ii), (iv), (v) hold. 

Given a simple root $\alpha$ the corresponding rank one subgroup $G_\alpha$ falls in one of three classes and we will choose $g_\alpha \in G_\alpha$ accordingly: If $G_\alpha \cong {\rm PSL}_2(\F)$, then we let $g_\alpha$ be an arbitrary non-trivial element. If $G_\alpha \cong {\rm SL}_2(\F)$ and $-1$ is not in the centre of $G$, then we choose $g_\alpha := -1$. Finally, if $G_\alpha \cong {\rm SL}_2(\F)$ and $-1$ is contained in the centre of $G$, then we choose $g_{\alpha}$ different from $\pm 1$. In any case there exists $v_\alpha \in \mathcal{U}_\F$ with ${\rm Ad}(g_\alpha)(v_\alpha) \neq v_\alpha$.  In the first two cases, the orbit-stabilizer formula and the fact that $g_\alpha$ is contained in any non-trivial normal subgroup of $G_\alpha$ imply that $G_\alpha \to \mathcal{U}_\F : g \mapsto g(v_\alpha)$ is injective. Therefore, for any $G_\alpha$-invariant subspace $U$ of $\mathcal{U}_\F$ containing $v_\alpha$, the group homomorphism $\rho : G_\alpha \to \mathrm{GL}(U)$ induced by the adjoint action is injective. In the third case, a similar argument shows that the kernel is contained in the centre of $G$.

We have shown that for $v \in \mathcal{U}_\F$ there exists a family of subspaces
$\{V^v_{\overline{\alpha}}\}_{\overline{\alpha}}$ of $\calU_\F$ such that (i), (ii), (iv), (v), (vi) hold. 

It remains to show that such a family also satisfies (iii) and (vii).

Applying Borel--Tits \cite{boreltits} to $G_{\alpha_1}$ and the Zariski closure of its image under $\rho_{\alpha_1}$ in $\mathrm{GL}(V^v_{\overline
\alpha})$ we conclude that $\rho_{\alpha_1} : G_{\alpha_1} \to {\rm GL}(V^v_{\overline{\alpha}})$ is algebraic and, thus, $\rho_{\alpha_1} : (G_{\alpha_1},\tau_{\alpha_1}) \to ({\rm GL}(V^v_{\overline{\alpha}}),\mathcal{O}_\F)$ is continuous. It follows from an open mapping theorem (cf.\ \cite[Lemma~2.1]{final-group-topologies}) that the subspace topology on $\rho_{\alpha_1}(G_{\alpha_1})$ is locally compact. Therefore, by \cite[Theorem~II.5.11]{Hewitt/Ross:1963}, the image $\rho_{\alpha_1}(G_{\alpha_1})$ is a closed subgroup of $({\rm GL}(V^v_{\overline{\alpha}}),\mathcal{O}_F)$ and, hence, so are $\rho_{\alpha_1}(U_{\pm\alpha_1})$. This establishes (vii).

Since $\rho_{\alpha_1} : (G_{\alpha_1},\tau_{\alpha_1}) \to ({\rm GL}(V^v_{\overline{\alpha}}),\mathcal{O}_\F)$ is continuous, so is the orbit map $(G_{{\alpha_1}},\tau_{{\alpha}_1}) \to (V^v_{\overline{\alpha}},\mathcal{O}_\F) : g \mapsto g(v)$, i.e., we have shown (iii) for each $1$-tuple $\overline{\alpha}$. Decompose an arbitrary tuple $\overline{\alpha}$ of length more than $1$ as $(\alpha_1,\overline{\beta})$. By induction the orbit map $(G_{\overline{\beta}},\tau_{\overline{\beta}}) \to (V^v_{\overline{\beta}},\mathcal{O}_\F) : g \mapsto g(v)$ is continuous and, hence, so is $\phi : (G_{\overline{\beta}},\tau_{\overline{\beta}}) \to (V^v_{\overline{\alpha}},\mathcal{O}_\F) : g \mapsto g(v)$ by (iv). Since the action $\epsilon : (\mathrm{GL}(V^v_{\overline{\alpha}}),\mathcal{O}_\F) \times (V^v_{\overline{\alpha}},\mathcal{O}_\F) \to (V^v_{\overline{\alpha}},\mathcal{O}_\F)$ is continuous, it follows that $$\epsilon \circ (\rho_{\alpha_1} \times \phi) : (G_{\alpha_1},\tau_{\alpha_1}) \times (G_{\overline{\beta}},\tau_{\overline{\beta}}) \to (V^v_{\overline{\alpha}},\mathcal{O}_\F) : (x_{\alpha_1},x_{\overline{\beta}}) \to x_{\alpha_1}x_{\overline{\beta}}(v)$$ is continuous. As $G_{\overline{\alpha}}$ carries the quotient topology $\tau_{\overline{\alpha}}$, this means that also $$(G_{\overline{\alpha}},\tau_{\overline{\alpha}}) \to (V^v_{\overline{\alpha}},\mathcal{O}_\F) : x \mapsto x(v)$$ is continuous, proving (iii).
\end{proof}
Now we can deduce:
\begin{prop}\label{KPBasic}
Let $\F$ be a local field and let $G=G_{\mathcal D}(\F)$ be a
split Kac--Moody group over $\F$. Then the Kac--Peterson topology is a $k_\omega$
group topology on
$G$, which is moreover independent of the choice of the system $\Pi$ of simple roots. 
\end{prop}
\begin{proof} To simplify the argument, we first assume that $G$ has trivial centre. In a second step we will then remove this assumption by an easy embedding argument. 

Thus let $G$ be centre-free and fix a $k$-tuple $\overline{\alpha} = ( \alpha_1, \alpha_2, ..., \alpha_k)$ of simple roots. Our first claim is that $(TG_{\overline{\alpha}}, \tau_{\alpha})$ is Hausdorff. For this, let $g\neq h\in G_{\overline{\alpha}}$. Since $G$ is centre-free there then exists $v \in \mathcal{U}_\F$ such that $g(v) \neq h(v)$. By Proposition~\ref{prop:ad-finite-dimensional-submodule}(iii) there exists a finite-dimensional subspace $V^v_{\overline{\alpha}}$ of $\mathcal{U}_\F$ that yields a continuous orbit map $f : (TG_{\overline{\alpha}},\tau_{\overline{\alpha}})\to (V^v_{\overline{\alpha}},\mathcal{O}_\F) : x \mapsto x(v)$.
Taking preimages under $f$ of suitable open neighbourhoods of $g(v)$, resp.\ $h(v)$ 
provides disjoint open neighbourhoods of $g$ and $h$. Hence $G_{\overline{\alpha}}$ is Hausdorff. It then follows from \cite[Proposition~4.2(d)]{final-group-topologies} that the spaces $(TG_{\overline{\alpha}}, \tau_{\alpha})$ are $k_\omega$. Moreover, the space $(G, \KP)$ is T1 as a direct limit of Hausdorff spaces.

Since the multiplication map
$TG_{(\alpha_1, \dots, \alpha_k)} \times TG_{(\beta_1, \dots, {\beta_m})}
\ra TG_{(\alpha_1, \dots, \beta_m)}$ is continuous and these pieces are $k_\omega$, it follows from 
\cite[Propositions~4.2, 4.7]{final-group-topologies} that multiplication on $G$ is continuous with respect to $\KP$. A similar argument shows that also inversion is continuous, whence $(G, \KP)$ is a topological group. As a T1 topological group it is in fact Hausdorff, hence $k_\omega$ as a Hausdorff direct limit of $k_\omega$-spaces.

The indepence of the topology $\KP$ of the choice of $\Pi$ now follows from the fact that up to conjugation there is a unique choice (\cite[Theorem~10.4.2]{remy}) and the fact that conjugation in a topological group is a homeomorphism. This finishes the proof in the case where $G$ is centre-free.

For the general case, we observe that the above proof goes through provided we are able to show that the pieces $(TG_{\overline{\alpha}}, \tau_{\alpha})$ are Hausdorff. For this, in turn, it suffices to embed them continuously into a Hausdorff topological space. We will use an embedding into a centre-free Kac--Moody group of larger rank, which generalizes the embedding of ${\rm SL}_n(\F)$ into ${\rm PGL}_{n+1}(\F)$.

Given a simply connected Kac--Moody group $G_\mathcal{D}(\mathbb{F})$ with simple roots $\Pi$, then, for any non-empty subset $\Sigma \subseteq \Pi$, the group $\langle U_\alpha, U_{-\alpha} \mid \alpha \in \Sigma \rangle$ again is a simply connected Kac--Moody group. If $\Pi$ is irreducible and $\Sigma$ is a proper subset of $\Pi$, then the centre of $\langle U_\alpha, U_{-\alpha} \mid \alpha \in \Sigma \rangle$ intersects the centre of $G_\mathcal{D}(\mathbb{F})$ trivially. Therefore, by \cite[Proposition~9.6.2]{remy}, the corresponding centre-free adjoint group $G_{\mathcal{D}^\mathrm{ad}}(\mathbb{F})$ contains the simply connected Kac--Moody group $\langle U_\alpha, U_{-\alpha} \mid \alpha \in \Sigma \rangle$ as a subgroup. This argument shows that each simply connected Kac--Moody group can be embedded into an adjoint Kac--Moody group whose type is obtained by adding one further vertex to each connected component of the diagram.
\end{proof}

The following result is now an immediate consequence of Proposition
\ref{prop:G-I-in-G-n}:
\begin{cor} \label{cor:equivalent-description}
Let $\mathbb{F}$ be a local field, and let $G$ be a split Kac--Moody group over $\mathbb{F}$. Then, with notation as in Proposition
\ref{prop:G-I-in-G-n}, 
$(G, \KP)$ is the direct limit of each one of the following three directed systems:
$$\bigcup_{k \in \mathbb{N}} (G_k^+,\KP|_{G_k^+}), \quad 
\bigcup_{k \in \mathbb{N}} (G_k^-,\KP|_{G_k^-}), \quad 
\bigcup_{k \in \mathbb{N}} (G_k^+ \cap G_k^-,\KP|_{G_k^+ \cap G_k^-}).$$
\end{cor}

\begin{rem}\label{KacPetersonOriginal}
\label{def:rg-top}\label{cor:kp-top=rg-top}\label{prop:rg-final-topology}\label{
lemma:first-inclusion}\label{prop:kp-rg-coincide}
If $\mathcal D$ is a {\em centred} Kac--Moody root datum and $G = \mathcal G_{\mathcal D}(\F)$, then the torus can be recovered from the root groups. Therefore, the Kac--Peterson topology coincides with the final topology with respect to the directed system $(G_{\overline \alpha}, \tau_{\overline \alpha}|_{G_{\overline \alpha}})_{\overline \alpha}$. In this case, our definition of the Kac--Peterson topology is equivalent to the one given in  
\ref{kpdef} is equivalent to the one given in 
\cite[Section 4G]{kac-peterson-regular-functions} using parametrizations of 
the root groups as follows. 

Let $\mathbb{F}$ be a local field and $G= \mathcal G_{\mathcal D}(\F)$ a centred Kac--Moody group. For each simple root $\alpha$, choose a parametrization 
$x_{\pm\alpha}\colon \F \to U_{\pm\alpha}$ of the root groups. For each finite
sequence of positive or negative simple 
roots $\overline{\beta} = (\beta_1, \ldots, \beta_k)$ denote by 
$$x_{\overline{\beta}} : \F^k \to G : (t_1, \ldots, t_k) \mapsto 
x_{\beta_1}(t_1) \cdots x_{\beta_k}(t_k)$$ the composition of the chosen 
parametrizations with the product map of $G$, and let
$U_{\overline{\beta}}$ denote the image of $x_{\overline{\beta}}$.
As, by the Gauss algorithm/Bruhat decomposition, for each simple root $\alpha$ one has $G_\alpha =
U_\alpha U_{-\alpha} U_\alpha U_{-\alpha}$, the final 
topology on $G$ with respect to the maps 
$x_{\overline{\beta}}$ coincides with the Kac--Peterson topology.

In the {\em non-centred} case, one additionally has to prescribe the topology of the torus. To this end, one classically realizes an $n$-dimensional split $\mathbb{F}$-torus with group of characters $\Lambda \cong \mathbb{Z}^n$ as the affine variety $$\{ (a_1, b_1, ..., a_n, b_n ) \in \mathbb{F}^{2n} \mid \forall 1 \leq i \leq n : a_ib_i = 1 \}.$$ The parametrization in order to obtain $\Hom_{\Z-{\rm alg}}(\Z[\Lambda],\mathbb{F})$ is given by the map that sends the $a_i$ to the free abelian generators of $\Lambda$ and the $b_i$ to their inverses.    
\end{rem}

An important tool in the study of the Kac--Peterson topology, which goes back to the original work of Kac and Peterson (see \cite[\S2]{kac-peterson-regular-functions}), are weakly regular functions in the sense of the following definition:

\begin{defin} A function $f : G \to \mathbb{F}$ is called \emph{weakly regular}, if $f \circ x_{\overline{\beta}} : \mathbb{F}^k \to \mathbb{F}$ is a polynomial function for all $\overline{\beta} \in (\Pi\cup-\Pi)^k$ and all $k \in \mathbb{N}$ and $f|_T$ is a regular function in the usual sense.
\end{defin}
Note that for centred $G$ the last condition is automatic; thus we recover the original definition from \cite[\S2]{kac-peterson-regular-functions} in this case. The link between weakly regular functions and the Kac--Peterson topology is provided by the following lemma:

\begin{lem}\label{weaklyregular} Every weakly regular function is continuous with respect to the Kac--Peterson topology.
\end{lem}
\begin{proof} Let $A \subset \mathbb F$ be a closed subset and $f : G \to \mathbb{F}$ be a weakly regular function. Then for each $\overline{\beta} = (\beta_1,...,\beta_k) \in (\Pi \cup -\Pi)^k$ the preimage $(f \circ x_{\overline{\beta}})^{-1}(A)$ is closed with respect to the Hausdorff topology on $\mathbb{F}^k$, because polynomial functions are continuous. As this set equals the preimage under $x_{\overline{\beta}}$ of $f^{-1}(A)$, its image in $U_{\beta_1}U_{\beta_2}\cdots U_{\beta_k}$ under $x_{\overline{\beta}}$ is closed. This shows that $f^{-1}(A)$ is closed, as $\KP$ equals the direct limit topology. 
\end{proof}


\subsection{Topology of spherical subgroups}

We retain the notation of the preceding section. In particular, $\F$ denotes a local
field and $G$ a split Kac--Moody group over
$\F$ endowed with the
Kac--Peterson topology $\KP$. We also denote by $\Delta_{\pm}$ the two halves of the associated twin building. Before we can continue our study of this topology we need to
identify various closed subgroups. 

We start with the following observation:
\begin{prop}\label{prop:B-closed} The subgroups $B_{\pm}<G$ are closed with respect to $\KP$.
\end{prop}
\begin{proof} 
Observe that for $v \in \mathcal{U}_\mathbb{F}$ and $v^* \in (\mathcal{U}_\mathbb{F})^*$, the vector space dual of $\mathcal{U}_\mathbb{F}$, the map $f_{v,v^*} : G \to \mathbb{F} : g \mapsto v^*(g(v))$ is a weakly regular function, because $\mathrm{ad}_{e_\alpha}$ and $\mathrm{ad}_{f_\alpha}$ are locally nilpotent (cf.\ the proof of Proposition~\ref{prop:ad-finite-dimensional-submodule}).  Denote by $\mathcal{U}_\mathbb{F}^{\geq 0}$ the subspace of $\mathcal{U}_\mathbb{F}$ consisting of the non-negative vectors with respect to the $Q$-grading (cf.\ Section~\ref{adjointrep}). Then, for each $g \in G \backslash B_+$ there exists a $v \in \mathcal{U}_\mathbb{F}^{\geq 0}$ such that $g(v) \not\in \mathcal{U}_\mathbb{F}^{\geq 0}$. This means that there exists $v^* \in \mathrm{Ann}(\mathcal{U}_\mathbb{F}^{\geq 0})$, the annihilator of $\mathcal{U}_\mathbb{F}^{\geq 0}$ in $(\mathcal{U}_\mathbb{F})^*$,  such that $f_{v,v^*}(g) \neq 0$. We conclude that $B_+$ is the set of common zeros of the family of weakly regular functions $(f_{v,v^*})_{v \in \mathcal{U}_\mathbb{F}^{\geq 0}, v^* \in \mathrm{Ann}(\mathcal{U}_\mathbb{F}^{\geq 0})}$. It then follows from Lemma \ref{weaklyregular} that $B_+$ is closed with respect to $\KP$.
%
%
\end{proof}

The proposition implies in particular that the halves $\Delta_{\pm} = G/{B_{\pm}}$ are Hausdorff when equipped with the quotient topology with respect to $\KP$. We will see in the next corollary that 
this has massive consequences for the topology of \emph{spherical} subgroups --- i.e., fundamental subgroups of the form $G_{\alpha_1, \dots, \alpha_r} = \langle G_{\alpha_1}, ..., G_{\alpha_r} \rangle$ for spherical subsets $\{ \alpha_1, ..., \alpha_r\}$ of the Coxeter system $(W,S)$, and their conjugates.

Note that in particular all rank one subgroups are spherical. Any spherical subgroup carries a unique semisimple Lie group topology $\mathcal{O}_\F$ over the ground field $\mathbb F$.

\begin{cor}\label{cor-panels-cpt} Let $\mathbb{F}$ be a local field and let $G$
be a split Kac--Moody group over $\F$. Equip the halves $\Delta^{\pm}$ of the
associated twin building with the quotient topology with respect to the
Kac--Peterson topology. Then:
\begin{enumerate}
 \item $\Delta^{\pm}$ are $k_\omega$-spaces.
 \item Panels in $\Delta^{\pm}$ --- and, more generally, spherical residues --- are compact.
 \item For every real root $\alpha$ the restriction of $\KP$ to
$G_{\alpha}$ coincides with $\mathcal O_\F$.
\item If $H$ is a spherical subgroup, then the restriction of $\KP$ to $H$  coincides with its Lie group topology.
 \item Spherical subgroups are closed.
\end{enumerate}
\end{cor}
\begin{proof}
\begin{enumerate}
\item This is immediate by Proposition~\ref{KPBasic}, Proposition~\ref{prop:B-closed} and
\cite[Proposition~4.2(d)]{final-group-topologies}.
\item  The (continuous) action of the group
$(G_{\alpha}, \mathcal O_\F)$ on the twin building preserves the panel
$P_{\alpha}$. Denote by $B_{\alpha}< G_{\alpha}$ the point stabilizer of a
basepoint in $P_{\alpha}$, so that we obtain a continuous bijection between
$G_\alpha/B_\alpha$ and $P_{\alpha}$. The former is compact (see Remark
\ref{ttb4holds}) and the latter is Hausdorff by (i), whence the latter is
compact as it is a Hausdorff quotient of a compact space. The same argument works for spherical residues.
\item Denote by $\tau_{co}$ the compact-open topology on $G_{\alpha}$ with
respect to the action on $P_\alpha$. Then we have continuous maps
\[(G_\alpha, \mathcal O_\F) \to (G_\alpha, \KP) \to (G_\alpha, \tau_{co}).\]
However, we have $\mathcal O_\F = \tau_{co}$, so in fact all three topologies
coincide. 
\item Since the Lie group topology is precisely the compact-open topology with respect to the action on the corresponding spherical residue, the same argument as in (iii) applies.
\item By (iii) and (iv) the subgroups in question are locally compact, hence must be
closed by \cite[Theorem~II.5.11]{Hewitt/Ross:1963}. \qedhere
\end{enumerate}
\end{proof}
From this in turn we deduce:
\begin{cor}\label{RootGroupsLcp}
\begin{enumerate}
\item For every root $\alpha$ the root group $U_\alpha$ is closed with respect to $\KP$.
\item For every root $\alpha$ the canonical map $\mathbb F \to (U_\alpha, \KP|_{U_\alpha})$ is a homeomorphism.
\item $T$ is closed with respect to $\KP$.
\end{enumerate}
\end{cor}
\begin{proof} (i) and (ii) are immediate since $U_\alpha < G_\alpha$ and $\KP|_{G_\alpha}$ is the Lie group topology. (iii) then follows from (i) and $T = \bigcap_{\alpha \in \Phi^{re}} N_G(U_\alpha)$, cf.\ Proposition~\ref{prop:kmg-has-rgd}.
\end{proof}
Getting the analogous statement of (ii) for the torus is slightly more involved:
\begin{prop}\label{toptorus} Let $G$ be a centre-free Kac--Moody group, i.e., a subgroup of an adjoint Kac--Moody group, or a central quotient of a simply connected Kac--Moody group. Then the map $(T, \tau_\F) \to (T, \KP|_T)$ is a homeomorphism. In particular, if $G$ is centred, then there exists a finite group $F$ such that $(T,  \KP|_T) \cong (\F^\times)^n/F$.
\end{prop}
\begin{proof} Consider first the case that $G$ is centre-free. By definition $(T, \KP|_T) = \lim_\to (T, \tau_{\overline \alpha})$. It thus suffices to show that the continuous maps $(T, \tau_\F) \to (T, \tau_{\overline \alpha})$ are open. Now Proposition~\ref{prop:representation-remy} yields a finite-dimensional vector space $V^v_{\overline{\alpha}}$ and a homomorphism $(T, \tau_{\overline \alpha}) \to {\rm GL}(V^v_{\overline{\alpha}})$. Since $G$ is assumed adjoint, this homomorphism is actually injective. It remains to show only that the map $(T, \tau_\F) \to (T, \tau_{\overline \alpha}) \to {\rm GL}(V^v_{\overline{\alpha}})$ is a homeomorphism onto its image. 

By the explicit formulae, the map in question is algebraic, whence continuous and has closed and, consequently, locally compact image. By the open mapping theorem, this map is therefore open. This finishes the proof in the centre-free case. 

Now we consider the second case, where $G$ is assumed to be a central quotient of a simply connected Kac--Moody group. The same argument as in the first case shows that $\KP$ and $\tau_\F$ coincide on the quotient ${\rm Ad}(T)$. By Lemma~\ref{ToriIsogeneous}, this is a finite quotient, hence the proposition follows from standard topological extension theory.
\end{proof}

\subsection{Kac--Peterson topology II: Universality}

At this point we have assembled enough information about the Kac--Peterson topology to characterize it in terms of a universal property. 
\begin{defin} \label{def:kac-peterson-topology} 
\index{Kac--Peterson topology} 
Let $\F$ be a local field and let $G$ be a
split Kac--Moody group 
over $\F$. Then 
the \textbf{universal topology} $\tau$ on $G$ is defined to be 
the \index{final group topology}final group topology with respect to the 
maps $$\phi_\alpha : \mathrm{SL}_2(\F) \to G, \, \, \alpha \in \Phi^{re}, \quad \quad \eta(\mathbb{F}) : \Hom_{\Z-{\rm alg}}(\Z[\Lambda], \F) \to G,$$ where $\mathrm{SL}_2(\F)$ and $\Hom_{\Z-{\rm alg}}(\Z[\Lambda], \F) \cong (\F^\times)^{\mathrm{rk}(\Lambda)}$ are equipped with their Lie group
topologies.
\end{defin}

Note that, as before, one obtains the same universal topology if one considers simple roots $\alpha \in \Pi$ only.

\begin{lem}[{cf.\ \cite[Lemma 6.2]{final-group-topologies}}]
\label{lemma:topology-same-wrt-simple-roots}
Let $\F$ be a local field, let $G$ be a
split Kac--Moody 
group, and let $\Pi = \{\alpha_1, \ldots, \alpha_n\}$ be a basis of simple 
roots of $\Phi^{re}$. Then the universal topology on $G$ is the final
group topology with 
respect to the maps $(\phi_{\alpha_i})_{1 \leq i \leq n}$ and $\eta(\F)$.
\end{lem}
\begin{proof} It suffices to observe that for every real root $\alpha$ there 
exists $w \in W$ and $\alpha_i \in \Pi$ such that $\alpha = w.\alpha_i$, whence
for any 
representative $\tilde w$ of $w$ in $G$, one has 
$G_\alpha = G_{w.\alpha_i} = \tilde w G_{\alpha_i}\tilde w^{-1}$.
\end{proof}

Again the universal topology can be defined for general topological fields, but 
it is unclear to us whether it has any
good properties in general; we do not even know whether it is Hausdorff.
However, over local fields we can show the following:
\begin{prop}\label{universalKP}
Assume $\F$ is a local field and that $G$ is centre-free or a central quotient of a simply connected Kac--Moody group. Then the universal topology and the
Kac--Peterson topology coincide. In particular, $(G, \tau)$ is Hausdorff and
$k_\omega$.
\end{prop}
\begin{proof}
Since the inclusion maps $G_\alpha \to (G, \tau_{KP})$ are continuous, we obtain
a continuous map $(G, \tau) \to (G, \tau_{KP})$. It remains to show that this
map
is open. In view of Corollary \ref{cor-panels-cpt}(iii), the topologies coincide
on
each $G_{\alpha}$. Since multiplication is continuous, the map
$(G_{\overline{\alpha}}, \tau_{\overline{\alpha}}) \to
(G_{\overline{\alpha}},\tau|_{G_{\overline{\alpha}}})$ is continuous. On the
other hand the map
\[(G_{\overline{\alpha}},\tau|_{G_{\overline{\alpha}}}) \to
(G_{\overline{\alpha}},\tau_{KP}|_{G_{\overline{\alpha}}})=
(G_{\overline{\alpha}}, \tau_{\overline{\alpha}})\]
is continuous as the restriction of a continuous map. Altogether we have shown
that $\tau$ and $\tau_{KP}$ coincide on each $G_{\overline{\alpha}}$. It then
follows that they coincide globally and the subgroup of $G$ generated by the root subgroups. The claim therefore follows from Proposition~\ref{toptorus}.
\end{proof}
In the two-spherical case we can reformulate the universal property of the Kac--Peterson topology in the form of an amalgamation result that generalizes
\cite[Theorem 6.20]{final-group-topologies}.
\begin{thm}[Topological Curtis--Tits Theorem] \label{theorem:curtis-tits}
\index{Curtis--Tits Theorem}
Let $\mathbb{F}$ be a local field and let $G$ be a
two-spherical simply connected split Kac--Moody group. 
Let $\Phi^{re}$ be the set of real roots and let $\Pi$ be a basis of 
simple roots for $\Phi^{re}$. For $\alpha, \beta \in \Pi$, set 
$G_\alpha := \phi_\alpha(\mrm{SL}_2(\F))$ and 
$G_{\alpha \beta} := \langle G_\alpha \cup G_\beta \rangle$. Moreover, 
let $\iota_{\alpha \beta}\colon G_\alpha \hookrightarrow G_{\alpha \beta}$ be 
the canonical inclusion morphisms.

Then the group $(G, \KP)$ is a universal enveloping group of the 
amalgam $\{G_\alpha, G_{\alpha\beta}; \iota_{\alpha\beta} \}$ in the 
categories of 
\begin{enumerate}
\item \label{item:curtis-tits-1}abstract groups,
\item \label{item:curtis-tits-2}Hausdorff topological groups and
\item \label{item:curtis-tits-3}$k_\omega$ groups.
\end{enumerate}
\end{thm}
\begin{proof}
\begin{enumerate}
\item This is the main result of \cite{abr-muh:bn-pairs}.
\item By Lemma
\ref{lemma:topology-same-wrt-simple-roots} and Proposition \ref{universalKP} the group 
$(G,\tau_{KP})$ is the direct limit of the amalgam $\{G_\alpha, G_{\alpha\beta};
\iota_{\alpha\beta} \}$ in the 
category of topological groups. Since $\tau_{KP}$ is Hausdorff by
Proposition~\ref{universalKP} the claim
follows. 
\item By \ref{item:curtis-tits-2}, the claim follows from 
\cite[Corollary 5.10]{final-group-topologies}.\qedhere
\end{enumerate}
\end{proof}
In Theorem \ref{theorem:curtis-tits} the hypothesis of simple connectedness is important, since otherwise the torus will not be the universal enveloping group of the amalgam of the tori of ranks one and two. For their central quotients it is, of course, possible to derive a compact presentation as well by incorporating the finite group of toral relations manually (cf.\ Lemma~\ref{ToriIsogeneous}). 

We obtain:
\begin{cor} Central quotients of simply connected two-spherical split Kac--Moody groups over local fields are compactly presented in the sense of \cite[Definition~3.1]{Cornulier}.
\end{cor}

\begin{rem} \label{Rem:SphericalSubgroups}
We have seen in Corollary \ref{cor-panels-cpt} that in the spherical case the Kac--Peterson topology coincides with the Lie group topology. Thus Theorem \ref{theorem:curtis-tits} applies in particular to semisimple Lie groups over local fields. We emphasize once more that in the non-spherical case it is not locally compact and, thus, not metrizable, cf.\ Remark~\ref{notlocallycompact}; see also \cite[Example~6.14]{final-group-topologies} for a detailed discussion of group topologies on affine Kac--Moody groups.
\end{rem}

\subsection{Topology of bounded and some non-spherical subgroups}
We now return to the study of the topology of subgroups of $G$.
Corollary~\ref{cor-panels-cpt} implies that $(G, \KP)$ induces the natural topology on the obvious finite-dimensional pieces of $G$, the spherical subgroups. In this section we intend to understand the topologies induced on some infinite-dimensional subgroups and on the less obvious finite-dimensional pieces, the bounded subgroups. 

Our first goal is to understand the groups $U_{\pm}$. The key idea is to consider them as direct limits of finite-dimensional pieces. In this context the key observation is as follows:
\begin{lem}\label{RootGroupProductOpen} Let $\beta_1, \dots, \beta_n$ be distinct positive real roots and  denote by $X$ the image of the multiplication map
\[
m: U_{\beta_1} \times \dots \times U_{\beta_n} \to G : (x_1, \dots, x_n) \mapsto x_1 \cdots x_n.
\]
Assume that the labeling is compatible with the Bruhat order in the sense that $i<j$ if $\beta_i < \beta_j$.
Then:
\begin{itemize}
\item[(i)] The map $m$ is injective.
\item[(ii)] If $U_{\beta_j}$ and $X$ are equipped with the respective restrictions of $\tau_{KP}$ and their product is equipped with the product topology, then $m$ is a homeomorphism onto its image.
\end{itemize}
\end{lem}
We warn the reader that the above product is not to be confused with the products $U_{\overline \alpha}$ studied before. Unlike the $\alpha_j$, the $\beta_j$ are assumed to be all positive and distinct and not necessarily simple. 
\begin{proof} We are going to establish both (i) and (ii) by induction on $n$, the case $n=1$ being trivial. Now let $n \geq 2$, assume that both (i) and (ii) hold for all products of length at most $n-1$, 
and consider $X = U_{\beta_1} \cdots U_{\beta_n}$. By assumption, $\beta_1$ is Bruhat minimal among $\beta_1, \dots, \beta_n$. We can therefore conjugate by a suitable element $g \in U^{-}$ to ensure that $\beta_1$ is in fact a positive simple root and the other $\beta_j$ are still positive and distinct. Then there exists a torus element $t \in T$ whose conjugation preserves the root group $U_{\beta_1}$ and contracts all other root groups $U_{\beta_j}$. We deduce that if $x_j \in U_{\beta_j}$ and $x = x_1 \cdots x_n$, then
 \[\lim_{n \to \infty} g^{-1}t^ngxg^{-1}t^{-n}g = x_1.\]
This shows that there is a continuous map
\[\phi_1: X \to U_{\beta_1} \times (U_{\beta_2}\cdots U_{\beta_n}) : x \mapsto (x_1, x_1^{-1}x).\]
By induction hypothesis we thus find a continuous map
\[\phi: X \to U_{\beta_1} \times \dots \times U_{\beta_n},\]
which is inverse to $m$. This shows (i) and (ii).
\end{proof}
There is a slight variant of the lemma, which incorporates the torus and will be useful later:
\begin{lem}\label{TorusRootGroupProductOpen} Let $\beta_1, \dots, \beta_n$ be distinct positive real roots and  denote by $Y$ the image of the multiplication map
\[
m: T \times U_{\beta_1} \times \dots \times U_{\beta_n} \to G : (t, x_1, \dots, x_n) \mapsto tx_1 \cdots x_n.
\]
Assume that the labeling is compatible with the Bruhat order in the sense that $i<j$ if $\beta_i < \beta_j$.
Then:
\begin{itemize}
\item[(i)] The map $m$ is injective.
\item[(ii)] If $T$, $U_{\beta_j}$ and $Y$ are equipped with the respective restrictions of $\tau_{KP}$ and their product is equipped with the product topology, then $m$ is a homeomorphism onto its image.
\end{itemize}
\end{lem}
\begin{proof} Since $T$ is self-centralizing the same argument applies.
\end{proof}
From now on we fix an enumeration
\[\mathbb N \to (\Phi^{re})^+ : n \mapsto \beta[n]\]
which is compatible with the Bruhat order and the Bruhat length in the following sense: We require that $i<j$ whenever $\beta[i] < \beta[j]$ in the Bruhat order and, moreover, that $\beta[i] < \beta[j]$ whenever $l(\beta[i]) < l(\beta[j])$.
\begin{cor} 
\begin{itemize}
\item[(i)] For all $n \in \mathbb N$ the map
\[\mathbb F^n \to X_{[n]} := \prod_{j=1}^n U_{\beta[j]} : (t_1, \dots, t_n) \mapsto u_{\beta[1]}(t_1)\cdots u_{\beta[n]}(t_n)\]
is a homeomorphism, where $X_{[n]}$ is equipped with the restriction of $\KP$.
\item[(ii)]  $(U_+, \KP|_{U_+}) = \lim_{\to} X_{[n]} \cong \lim_\to \mathbb F^n$.
\end{itemize}
\end{cor}
\begin{proof} (i) follows by combining Corollary \ref{RootGroupsLcp} and Lemma \ref{RootGroupProductOpen}. Concerning (ii) we observe that the action of $G$ on the twin building implies that 
$G_k^{-} = \bigcup_{l(w) \leq k} B_{-} w B_{-}$ intersects 
$U_+$ in exactly those root groups $U_\alpha$ with 
$l(\alpha) \leq k$. Hence, if we choose $n(k)$ such that $\beta[n(k)]$ is the largest root of length $k$, then
\[
	U_+ \cap G_k^{-} = X_{[n(k)]}.
\]
Thus $U_+ = \lim_\to X_{[n]}$ by Corollary \ref{cor:equivalent-description}. \end{proof}
A symmetric argument now shows:
\begin{cor} 
\begin{itemize}
\item[(i)] For all $n \in \mathbb N$ the map
\[\mathbb F^n \to X^-_{[n]} := \prod_{j=1}^n U_{-\beta[j]} : (t_1, \dots, t_n) \mapsto u_{-\beta[1]}(t_1)\cdots u_{-\beta[n]}(t_n)\]
is a homeomorphism, where $X^-_{[n]}$ is equipped with the restriction of $\KP$.
\item[(ii)] $(U_-, \KP|_{U_-}) = \lim_{\to} X^-_{[n]} \cong \lim_\to \mathbb F^n$.
\end{itemize}
\end{cor}
Applying Lemma \ref{TorusRootGroupProductOpen} instead of Lemma \ref{RootGroupProductOpen} we obtain:
\begin{prop}\label{uclosed}
\begin{enumerate}
\item The maps
\[T \times \mathbb F^n \to TX^{\pm}_{[n]}\]
are homeomorphisms.
\item $B_\pm = \lim_\to TX^{\pm}_{[n]} \cong T \times U_{\pm}$; in case $G$ is a central quotient of a simply connected Kac--Moody group, the latter is isomorphic to  $(\F^\times)^n/F \times \lim_\to \F^n$, where $F$ is some finite group.
\item The multiplication map $T \times U_{\pm} \to B_{\pm}$ is open.
\item $U_{\pm}$ are closed in $B_{\pm}$, hence in $G$.
\end{enumerate}
\end{prop}
\begin{proof} (i) is immediate from Lemma \ref{TorusRootGroupProductOpen}, for (ii) we need only to observe that $B_{\pm}$ intersects $G_n^{\mp}$ precisely in $TX^{\pm}_{[n]}$. (iii) and (iv) are immediate from (ii).
\end{proof}

\begin{rem}\label{notlocallycompact}
We conclude from Proposition~\ref{uclosed} that in the non-spherical case $(G, \KP)$ cannot be locally compact. Indeed, in this case the closed subgroup $U_+$ would be locally compact, which is absurd as it is homeomorphic to an infinite-dimensional vector space. Proposition~\ref{universalKP} and \cite[\S 21]{k-omega} therefore imply that $(G, \KP)$ is not metrizable and, in particular, not Polish.
\end{rem}

We conclude this section with an analysis of the topology of some further subgroups of $G$. The following is a straighforward generalization of Proposition~\ref{prop:B-closed}.

\begin{prop}\label{paraclosed} Parabolic subgroups are closed.
\end{prop}
\begin{proof} The proof is identical to the one given in Proposition~\ref{prop:B-closed}. 
Let $P$ be a parabolic subgroup of positive sign and denote by $\mathcal{U}_\mathbb{F}^{P}$ the subspace of $\mathcal{U}_\F$ spanned by the $P$-orbit of the subspace $\mathcal{U}_\mathbb{F}^{\geq 0}$ of the non-negative vectors with respect to the $Q$-grading. Then, for each $g \in G \backslash P$ there exists a $v \in \mathcal{U}_\mathbb{F}^{P}$ such that $g(v) \not\in \mathcal{U}_\mathbb{F}^{P}$. This means that there exists $v^* \in \mathrm{Ann}(\mathcal{U}_\mathbb{F}^{P})$ such that $f_{v,v^*}(g) \neq 0$. We conclude that $P$ is the set of common zeros of the family of weakly regular functions $(f_{v,v^*})_{v \in \mathcal{U}_\mathbb{F}^{P}, v^* \in \mathrm{Ann}(\mathcal{U}_\mathbb{F}^{P})}$. It then follows from Lemma~\ref{weaklyregular} that $P$ is closed with respect to $\KP$.
\end{proof}

\begin{cor}\label{bounded} Bounded subgroups, i.e., intersections of spherical parabolics of opposite signs, are algebraic Lie groups. Its Levi decomposition is a semi-direct product of closed subgroups.
\end{cor}
\begin{proof}
The fact that bounded subgroups are algebraic follows from \cite[Section~6.2]{remy}, as does the existence of a  Levi decomposition. A Levi factor is a spherical subgroup and, hence, a (closed) Lie group by Corollary~\ref{cor-panels-cpt}. The unipotent radical is closed by Propositions~\ref{uclosed} and \ref{paraclosed}; moreover, it equals a finite product of real root subgroups, hence it is homeomorphic to a finite-dimensional $\F$-vector space.  
\end{proof}

\subsection{The topological twin building of a split Kac--Moody group}\label{toptwinbuilding}

The goal of this section is to establish the first main result of this paper:
\setcounter{main}{0}
\begin{main}\label{HartnickConjecture2}\label{HartnickConjecture} Let $G$ be a
two-spherical simply connected split
Kac--Moody group over a local field
and let $\KP$ be the Kac--Peterson topology on $G$. Then the associated twin
building endowed with the quotient topology is a strong topological twin building (cf.\ Definition~\ref{strongtop}).

If the local field equals the field of real or of complex numbers, then $G$ is connected, otherwise totally disconnected.
\end{main}
The key observation is the following:
\begin{prop} \label{prop:multiplication-open}
Let $\mathbb{F}$ be a local field and let $G$ be a
two-spherical simply connected split Kac--Moody group over $\F$ endowed with the
Kac--Peterson topology. Then the map $m\colon U_+ \times T \times
U_- \to B_+B_- : (u_+,t,u_-) \mapsto
u_+tu_-$ is open.
\end{prop}

This result has been announced in \cite[Theorem~4]{kac-peterson-regular-functions}, actually without the restriction to the two-spherical case. This more general version would in fact allow one to remove the requirement that $G$ be two-spherical in the statement of Theorem~\ref{HartnickConjecture2}.

\begin{proof}[Proof of Proposition \ref{prop:multiplication-open}]
Let $\alpha_1$, $\alpha_2$, ..... be a cofinal sequence with respect to the sub-tuple relation.
Proposition~\ref{lemma:finiteprolongation} implies that for each $n \in \mathbb{N}$ there exists an $m \in \mathbb{N}$ with
$$B_+B_- \cap U_{\overline{\alpha'}} \subseteq (U_+ \cap
U_{\overline{\alpha''}})(T \cap U_{\overline{\alpha''}})(U_- \cap
U_{\overline{\alpha''}}),$$ where $\overline{\alpha'} = (\alpha_1,...,\alpha_n)$ and
$\overline{\alpha''} = (\alpha_1,...,\alpha_m)$. Conversely, for each $m \in \mathbb{N}$ there exists an $n \in \mathbb{N}$ with $$(U_+ \cap
U_{\overline{\alpha''}})(T \cap U_{\overline{\alpha''}})(U_- \cap
U_{\overline{\alpha''}}) \subseteq B_+B_- \cap U_{\overline{\alpha'}}.$$
We conclude that $$B_+B_- = \lim_\rightarrow (U_+ \cap
U_{\overline{\alpha}})(T \cap U_{\overline{\alpha}})(U_- \cap
U_{\overline{\alpha}}).$$
Moreover, since $U_\pm = \lim_\rightarrow
U_\pm \cap U_{\overline{\beta}}$ and $T = \lim_\rightarrow T \cap
U_{\overline{\alpha}}$, by \cite[Proposition~4.7]{final-group-topologies} one has
$$U_+ \times T \times U_- = \lim_\rightarrow \left((U_+ \cap U_{\overline{\alpha}}) \times (T \cap
U_{\overline{\alpha}}) \times (U_- \cap U_{\overline{\alpha}})\right).$$
As each
\begin{eqnarray*}
m_{\overline{\alpha}} : (U_+ \cap U_{\overline{\alpha}}) \times (T \cap
U_{\overline{\alpha}}) \times (U_- \cap U_{\overline{\alpha}}) & \to&  (U_+ \cap
U_{\overline{\alpha}})(T \cap U_{\overline{\alpha}})(U_- \cap
U_{\overline{\alpha}}) \\ (u_+,t,u_-) & \mapsto &
u_+tu_-
\end{eqnarray*}
 is open, this implies that $m : U_+ \times T \times
U_- \to B_+B_- : (u_+,t,u_-) \mapsto
u_+tu_-$ is open.
\end{proof}
Now Theorem \ref{HartnickConjecture2} is immediate:

\begin{proof}[Proof of Theorem  \ref{HartnickConjecture2}]
Concerning the first part of Theorem \ref{HartnickConjecture2} it suffices to establish the hypotheses
of Theorem \ref{thm:top-twin-building}. Condition (iii) is precisely Proposition \ref{prop:multiplication-open}. Since we have already established
Condition (i) in Proposition \ref{prop:B-closed},
Condition (ii) in Corollary \ref{cor:equivalent-description} and Condition (iv)
in Corollary \ref{cor-panels-cpt}, the first part of Theorem \ref{HartnickConjecture2} follows. 

Concerning the second part we first note that the underlying roots groups and the torus are connected if and only if $k$ is archimedean. In the archimedean case it then follows that the pieces $TG_{\overline \alpha}$, and hence their limit $G$ are connected. In the non-archimedean case we see from Proposition~\ref{uclosed} and Proposition~\ref{prop:multiplication-open} that the open subset $B_+B_-$ is totally disconnected. This implies that $G$ itself is totally disconnected, finishing the proof.
\end{proof}

Note that the underlying topological foundation of the topological twin building of $G$ is precisely the corresponding topological $k$-split foundation. This shows:
\begin{cor}\label{DynkinTreeIntegrable}
Let $k$ be a local field. Then every Dynkin tree is topologically $k$-integrable.
\end{cor}
\begin{proof}
Let $\mathcal F$ be a topological $k$-split Moufang foundation corresponding to a given Dynkin tree. By the main
result of \cite{muehlherr-locally-split} there exists a (two-spherical) split
Kac--Moody group $G = G_{\calD}(k)$ such that the associated twin building
globalizes the abstract foundation underlying $\mathcal F$. Equip $G$ with the
Kac--Petersen topology and its twin building with the associated quotient
topology. In view of Theorem~\ref{HartnickConjecture} this is a topological twin
building. Since the topology on the root subgroups $U_\alpha$ is the standard one by Corollary~\ref{cor-panels-cpt}(iii), this
topological twin building realizes the given topological foundation.
\end{proof}
A combination of Corollary~\ref{DynkinTreeIntegrable} and Corollary~\ref{CorTreeSuff} now yields the second main result of the present article:
\begin{main}\label{ClassificationIntro}
Let $k$ be a local field. The maps $[\Delta]\mapsto [\mathcal D(\Delta)]$ induces a bijection
between isomorphism classes of $k$-split topological twin buildings of tree type and
isomorphism classes of simply connected simple $\{3,4,6\}$-labelled graphs, where edges
labelled $4$ or $6$ are directed.  
\end{main}

\subsection{Kac--Moody symmetric spaces}

We conclude with a couple of observations related to 
Kac--Moody symmetric spaces using the theory of flips introduced in
\cite{max-diss} and \cite{max}.

\begin{lem} \label{prop:theta-codistances}
Let $\mathbb{F}$ be a field, let $\mathcal{D}$ be a centred Kac--Moody root
datum, let $G_\mathcal{D}(\mathcal{F})$ be the corresponding split Kac--Moody
group, let
$\theta$ be a quasi-flip of the Kac--Moody group such that
$\theta(B_+) = B_-$, let $\tau_\theta\colon
G_\mathcal{D}(\mathbb{F}) \to G_\mathcal{D}(\mathbb{F}) : g \mapsto
\theta(g)^{-1}g$ be the corresponding Lang map, let $W$ be its Weyl group,
let $w \in W$, and let $x\in G_\mathcal{D}(\mathbb{F})$. 
Then $\delta^*(\theta(xB_+),xB_+)= w$ if and only if 
$x \in \tau_\theta^{-1}(B_-wB_+)$.
\end{lem}
\begin{proof}
One has the following chain of equivalences:
\begin{eqnarray*}
\delta^*(\theta(xB_+),xB_+) = w & \Longleftrightarrow & B_-\theta(x)^{-1}xB_+ =
B_-wB_+\\
& \Longleftrightarrow & \theta(x)^{-1}x \in B_-wB_+\\
	& \Longleftrightarrow & \tau_\theta(x) \in B_-wB_+\\
	& \Longleftrightarrow & x \in \tau_\theta^{-1}(B_-wB_+). \hfill \qedhere
\end{eqnarray*}
\end{proof}

\begin{thm} \label{thm:top-rel-wrt-theta-codistance} \index{topological twin
building!orbit structure!of $G_\theta$}
Let $\mathbb{F}$ be a local field, let $G$ be the adjoint form of 
a simply connected split Kac--Moody group endowed
with the Kac--Peterson topology $\KP$, let $W$
be its Weyl group, let
$\theta$ be a continuous quasi-flip of $G$ such that
$\theta(B_+) = B_-$, and let $G_\theta := 
C_{G}(\theta) = \mathrm{Fix}_G(\theta)$.
Moreover, for $w \in W$, let $\Delta_w := 
\{ c \in \Delta_+ \mid \delta^*(\theta(c),c) = w \}$.
Furthermore, let $\cod := \{ w \in W \mid \Delta_w \neq \emptyset \}$.
Then 
the following hold:
\begin{enumerate}
\item \label{item:g-theta-3}For $w \in \cod$ one has 
\[
\overline{\Delta_w} = \bigcup_{w' \geq w} \Delta_{w'}.
\]
\item \label{item:g-theta-4}For $w \in \cod$ the smallest open 
$G_\theta$-invariant subset of $\Delta_+$ containing $\Delta_w$ is 
\[
\bigcup_{w' \leq w} \Delta_{w'}.
\]
\end{enumerate}
\end{thm}
\begin{proof}
Using the quotient map $q \colon
G \to \Delta_+ = G/B_+$,
Lemma~\ref{prop:theta-codistances} states that for each $w \in W$ the
set $q(\tau^{-1}_\theta(B_-wB_+))$ equals $\Delta_w$. 
As $\theta$ is continuous, so is the Lang map $\tau_\theta \colon
G
\to G : g \mapsto \theta(g)^{-1}g$, whence all claims follow immediately from 
Theorem~\ref{thm:closure-relation-wrt-borel-subgroup}.
\end{proof}


\end{document}